\documentclass[a4paper,oneside,12pt]{article}

\usepackage{amsmath,amsthm,amsfonts,amscd,amssymb}
\usepackage{longtable,geometry}
\usepackage[english]{babel}
\usepackage[utf8]{inputenc}
\usepackage[active]{srcltx}
\usepackage[T1]{fontenc}
\usepackage{graphicx}
\usepackage{pstricks}
\usepackage{bbm}
\usepackage{mathtools}

\usepackage{MnSymbol}
\usepackage{stmaryrd}
\usepackage{nicefrac}
\usepackage{calrsfs}

\usepackage{rotating}
\usepackage{xcolor}
\usepackage{framed}
\usepackage{hyperref}
\hypersetup{
    colorlinks=false,
    pdfborder={0 0 0},
}
\usepackage{enumitem}
\usepackage{wasysym}

\colorlet{shadecolor}{blue!15}

\newtheorem{thm}{Theorem}[section]

\newtheorem{cor}[thm]{Corollary}
\newtheorem{lem}[thm]{Lemma}
\newtheorem{prop}[thm]{Proposition}

\newtheorem{rem}[thm]{Remark}

\newtheorem{claim}[thm]{Claim}
  
\newcommand{\be}[1]{\begin{equation}\label{#1}}
\newcommand{\ee}{\end{equation}}
\numberwithin{equation}{section}

\newcommand{\ba}[1]{\begin{align}\label{#1}}
\newcommand{\ea}{\end{align}}
\numberwithin{equation}{section}

\newcommand{\ben}{\begin{equation*}}
\newcommand{\een}{\end{equation*}}
\numberwithin{equation}{section}

\renewenvironment{proof}[1][\relax]
  {\paragraph{Proof\ifx#1\relax\else~of #1\fi}}%
  {~\hfill$\square$\par\bigskip}

\newcommand{\calA}{\mathcal{A}}
\newcommand{\calB}{\mathcal{B}}
\newcommand{\calC}{\mathcal{C}}
\newcommand{\calD}{\mathcal{D}}
\newcommand{\calE}{\mathcal{E}}

\newcommand{\calG}{\mathcal{G}}
\newcommand{\calH}{\mathcal{H}}
\newcommand{\calI}{\mathcal{I}}
\newcommand{\calJ}{\mathcal{J}}

\newcommand{\calL}{\mathcal{L}}

\newcommand{\frm}{\mathfrak{m}}

\newcommand{\frp}{\mathfrak{p}}

\newcommand{\bbE}{\mathbb{E}}

\newcommand{\bbH}{\mathbb{H}}

\newcommand{\bbN}{\mathbb{N}}

\newcommand{\bbP}{\mathbb{P}}

\newcommand{\bbR}{\mathbb{R}}

\newcommand{\bbT}{\mathbb{T}}

\newcommand{\bbZ}{\mathbb{Z}}
\newcommand{\sfT}{{\sf T}}
\newcommand{\sfB}{{\sf B}}
\newcommand{\sfD}{{\sf D}}
\newcommand{\sfU}{{\sf U}}

\newcommand{\eps}{\epsilon}

\newcommand{\rk}[1]{\bgroup\color{red}%
  \par\medskip\hrule\smallskip%
  \noindent\textbf{#1}%
  \par\smallskip\hrule\medskip\egroup}

\newcommand{\lra}{\leftrightarrow}
\newcommand{\xlra}{\xleftrightarrow}
\newcommand\concel[2]{\ooalign{$\hfil#1\mkern0mu/\hfil$\crcr$#1#2$}}  
\newcommand\nxlra[1]{\mathrel{\mathpalette\concel{\xlra{#1}}}}
\newcommand{\ind}{\mathbf{1}}
\newcommand{\Int}{\mathrm{Int}}

\renewcommand{\int}{\mathrm{in}}
\newcommand{\out}{\mathrm{out}}
\newcommand{\Ann}{\text{Ann}}

\renewcommand{\frp}{\textcolor{red}{\varhexagon\hspace{-2.58mm} +}}
\renewcommand{\frm}{\textcolor{red}{\varhexagon\hspace{-2.58mm} -}}

\newcommand{\sfrp}{\textcolor{red}{\scriptsize\varhexagon\hspace{-1.8mm} +}}
\newcommand{\sfrm}{\textcolor{red}{\scriptsize\varhexagon\hspace{-1.8mm} -}}

\renewcommand\rm{\includegraphics[page=1, scale=0.5]{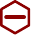}} 
\newcommand\rp{\includegraphics[page=2, scale=0.5]{spins.pdf}} 
\newcommand\bm{\includegraphics[page=3, scale=0.5]{spins.pdf}} 
\newcommand\bp{\includegraphics[page=4, scale=0.5]{spins.pdf}} 
\newcommand\srm{\includegraphics[page=1, scale=0.3]{spins.pdf}} 
\newcommand\srp{\includegraphics[page=2, scale=0.3]{spins.pdf}}
\newcommand\sbm{\includegraphics[page=3, scale=0.3]{spins.pdf}}
\newcommand\sbp{\includegraphics[page=4, scale=0.3]{spins.pdf}}

\newcommand{\La}{\Lambda}
\newcommand{\Ga}{\Gamma}

\newcommand{\Rect}{{\sf Rect}} 
\newcommand{\Cyl}{{\sf Cyl}}   
\newcommand{\Strip}{{\sf Strip}} 
\newcommand{\Par}{{\sf Par}} 
\newcommand{\Bottom}{{\sf Bottom}} 
\newcommand{\Top}{{\sf Top}}
\newcommand{\Left}{{\sf Left}} 
\newcommand{\Right}{{\sf Right}}

\newcommand{\dbp}{\mathrm{dp}}
\newcommand{\dbm}{\mathrm{dm}}
\newcommand{\Circ}{{\sf Circ}}

\newcommand\eqonD{\mathrel{\overset{\makebox[0pt]{\mbox{\normalfont\tiny\sffamily~$\calD$}}}{=}}}

\usepackage{enumitem}
\setlist[itemize]{itemsep=1pt, topsep=4pt}
\setlist[enumerate]{itemsep=1pt, topsep=4pt}

\title{Uniform Lipschitz functions on the triangular lattice have logarithmic variations}
\date{\today}
\author{Alexander Glazman\thanks{D\'epartement de Math\'ematiques, 
Universit\'e de Fribourg, 
23 Chemin du Mus\'ee, 
CH-1700 Fribourg, 
Switzerland.
\url{glazmana@gmail.com}}
~and Ioan Manolescu\thanks{D\'epartement de Math\'ematiques, 
Universit\'e de Fribourg, 
23 Chemin du Mus\'ee, 
CH-1700 Fribourg, 
Switzerland.
\url{ioan.manolescu@unifr.ch}}}

\begin{document}

\maketitle
\begin{abstract}
	Uniform integer-valued Lipschitz functions on a domain of size~$N$ of the triangular lattice 
	are shown to have variations of order~$\sqrt{\log N}$.
	
	The level lines of such functions form a loop~$O(2)$ model on the edges of the hexagonal lattice with edge-weight one. 
	An infinite-volume Gibbs measure for the loop~$O(2)$ model is constructed as a thermodynamic limit and is shown to be unique. 
	It contains only finite loops and has properties indicative of scale-invariance: macroscopic loops appearing at every scale.  
	The existence of the infinite-volume measure carries over to height functions pinned at the origin; the uniqueness of the Gibbs measure does not. 
		
	The proof is based on a representation of the loop~$O(2)$ model via a pair of spin configurations that are shown to satisfy the FKG inequality. 
	We prove RSW-type estimates for a certain connectivity notion in the aforementioned spin model.		
\end{abstract}
\tableofcontents

\section{Introduction}

Height functions occupy a central role in statistical mechanics models on lattices. 
Indeed, the Ising, six-vertex and dimer models are only some of the lattice models involving height function representations. 
The predicted conformal invariance of these models is tightly linked to the convergence of their associated height functions to the Gaussian Free Field (GFF) or variations of it. 
Both statements were proved only in a handful of cases, and remain fascinating conjectures in general.

In this paper we study integer-valued height functions defined on the vertices of the two-dimensional triangular lattice~$\bbT$,
or equivalently on the faces of the hexagonal lattice~$\bbH$.
It is then natural to impose that height functions are Lipschitz, that is functions whose difference between any two adjacent vertices is at most~$1$; see Figure~\ref{fig:lipschitz}.
More specifically, for any finite domain~$\calD$ of~$\bbH$, 
we will consider a uniformly chosen Lipschitz function among those with values~$0$ outside of~$\calD$.
The question of interest is the behaviour of such a function, especially as the domain~$\calD$ increases towards~$\bbH$.

\begin{figure}[!b]
    \centering
        \includegraphics[scale=1.4]{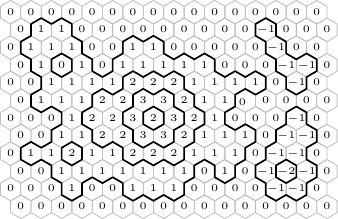}
    \caption{Lipschitz function~--- values at any two adjacent faces differ by at most~$1$. Level lines are shown in bold; they form a loop configuration distributed according to the loop~$O(2)$ measure with edge-weight one.}
    \label{fig:lipschitz}
\end{figure}

Our goal is to show that the variance of the value at the origin of a uniformly chosen Lipschitz function is of order~$\log N$, 
where~$N$ is the radius of the largest ball centred at the origin and contained in~$\calD$. 
This result, termed  {\em delocalisation} (or logarithmic delocalisation to be precise) is in agreement with the conjectural convergence of uniform Lipschitz functions to the GFF.

The essential tool here is a re-interpretation of the uniform Lipschitz functions as the loop~$O(2)$ model,
which in turn is represented as the superposition of two site percolations on~$\bbT$ interacting with each other -- 
below we view these as~$\pm$-spin assignments.
This {\em double-spin representation} is obtained by colouring the loops of the loop~$O(2)$ model in two colours (as was done in~\cite{ChaMac98}), 
then deriving a spin configuration from the families of loops of each colour. 
This may be viewed as the infinite-coupling-limit of the Ashkin--Teller model on the triangular lattice~\cite{HuaDenJacSal13}.

The loop~$O(n)$ model is defined on collections of non-intersecting simple cycles (loops) on a finite domain~$\calD$ of~$\bbH$ and has two real parameters~$n,x>0$. The probability of each configuration is proportional to~$n$ to the number of loops times~$x$ to the number of edges in it. The loop~$O(n)$ model has a rich conjectural phase diagram~\cite{Nie82,BloNie89} that remains mostly open; see \cite{PelSpi19} for an overview of the topic. 

For~$n = 2$, the model is expected to exhibit macroscopic loops when~$x\geq \tfrac{1}{\sqrt{2}}$ and exponential decay of loop sizes when~$x< \tfrac{1}{\sqrt{2}}$.
The former is confirmed in this paper for~$x = 1$ and in~\cite{DumGla17} for~$x=\tfrac{1}{\sqrt{2}}$.
The latter behaviour is shown to hold for~$x < \tfrac{1}{\sqrt{3}}+\eps$ for some~$\eps > 0$ in~\cite{GlaMan18}. 
The correspondence between the loop~$O(2)$ model and Lipschitz functions holds for any~$x > 0$, but the corresponding height functions are not uniform:
they are weighted by~$x$ to the number of pairs of adjacent faces of~$\bbH$ having different values.
The regime of exponential decay of loop sizes corresponds to localisation for the height function; that of macroscopic loops corresponds to logarithmic delocalisation. 

A main difficulty in the study of the loop~$O(n)$ model is the lack of monotonicity and positive association. 
These type of properties are however expected to hold in convenient representations of the model, as illustrated by the present paper and by \cite{DumGla17}. 
Indeed, a core ingredient of our arguments is the FKG inequality, which we show for the marginals of the double-spin representation.
We then develop Russo--Seymour--Welsh (or RSW)-type results for these marginals, which translate to similar statements for the loop and height function models. 

The RSW theory was initially developed for percolation \cite{Rus78,SeyWel78}, 
and later generalised to other models via more robust arguments (see for instance \cite{BefDum12b,DumSidTas17,DumSidTas16,Tas16,DinGos16}). 
It has become increasingly clear that for the latter type of arguments to apply the essential feature of the model is an instance of the FKG inequality. Indeed, other restrictions such as independence, symmetries and planarity have been, in some forms, relaxed in recent works. 
In this paper we do yet another step towards generalising this approach by considering a case where the Spatial Markov property applies only in a limited way. 

The FKG inequality mentioned above extends to the case of certain non-uniform distributions on Lipschitz functions (corresponding to the loop~$O(2)$ model with~$x<1$) and more generally to the loop~$O(n)$ model with~$n\geq 2, x\leq \tfrac{1}{\sqrt{n-1}}$. Thus, we hope that this instance of the FKG inequality, together with the strategy of our proofs can be useful in other studies of the loop~$O(n)$ model. 

Finally, we want to emphasise that in this work we do not attempt to prove convergence to the GFF. In general, the RSW theory can be viewed as a robust technique based on geometric constructions, but is not expected to lead to subtle convergence results. Indeed the seminal proofs of convergence of~\cite{Sch00,Ken00,Smi01,Smi10,CheSmi12,CheDumHon12a} are all based on some form of exact solvability, which is missing in our case.

\subsection{Uniform Lipschitz functions}\label{sec:height_functions}

Let~$\bbH$ denote the hexagonal lattice, embedded in~$\bbR^2$ with the origin~$0$ being the center of a face and the distance between the centres of any adjacent faces being~$1$. Write~$F(\bbH)$ for the set of faces of~$\bbH$. 
A subgraph~$\calD = (V(\calD),E(\calD))$ of~$\bbH$ without isolated vertices is called a \emph{domain} if there exists a self-avoiding polygon in~$\bbH$ denoted by~$\partial_E\calD$ such that~$E(\calD)$ is the set of edges surrounded by~$\partial_E\calD$ (excluding those of~$\partial_E\calD$). 
Denote by~$F(\calD)$ the set of faces adjacent to at least one edge of~$\calD$. 
The inner (and outer) face boundary of~$\calD$, written~$\partial_\int \calD$ (and~$\partial_\out\calD$, respectively) 
is the set of faces of~$\calD$ (and~$\bbH \setminus \calD$, respectively) bounded by at least one edge in~$\partial_E \calD$. 
The faces strictly in the interior of~$\calD$ are~$\Int(\calD) = F(\calD) \setminus \partial_\int \calD$.

For a domain~$\calD$, a Lipschitz function on~$\calD$ with zero boundary conditions is an integer-valued function~$\phi$ 
on the faces of~$\calD$ with the constraint that
\begin{itemize}
    \item if~$u,v \in F(\calD)$ are two adjacent faces, then~$|\phi(u) - \phi(v)|\leq 1$;
    \item for each~$u\in \partial_\int \calD$ we have~$\phi(u) = 0$.
\end{itemize} 
Since~$\calD$ is finite, only finitely many such functions exist. 
Write~$\pi_\calD$ for the uniform measure on such functions, and let~$\Phi_\calD$ denote a random variable with law~$\pi_\calD$. 

\begin{thm}\label{thm:lipschitz}~
	\begin{itemize}
		\item[(i)] There exist constants~$c, C > 0$ such that, for any finite domain~$\calD$ with~$0 \in \calD$, 
		\begin{align*}
			c  \, \log\mathrm{dist}(0,\calD^c) \leq  \mathrm{Var}(\Phi_\calD(0)) \leq C\,\log \mathrm{dist}(0,\calD^c).
		\end{align*}
		\item[(ii)] For any increasing sequence of domains~$(\calD_n)_{n\geq 1}$ with~$0 \in \calD_1$ and~$\bbH = \bigcup_n \calD_n$,
		the sequence of variables~$\Phi_{\calD_n} - \Phi_{\calD_n}(0)$ converges in law as~$n \to \infty$ to a 
		random Lipschitz function~$\Phi_\bbH : \bbH \to \bbZ$ that is equal to~$0$ at~$0$. Write~$\pi_\bbH$ for the law of~$\Phi_\bbH$.
		\item[(iii)] There exists~$c,C> 0$ such that, for any distinct~$x,y \in F(\bbH)$, 
		\begin{align*}
			c\log |x- y|\leq \mathrm{Var}(\Phi_\bbH(x) - \Phi_\bbH(y)) \leq C\log |x- y|.
		\end{align*}
		The same holds for~$\Phi_\calD$ for any domain~$\calD$ containing the ball of radius~$2|x- y|$ around~$x$.
	\end{itemize}
\end{thm}

One may wish to study height functions with different values imposed on the boundary via so-called boundary conditions. 
While we do not attempt to provide the most general form of our result, let us briefly mention some direct generalisations. 
First, for constant boundary conditions -- that is if we study uniform height functions with~$\phi(u) = c$ for all~$u\in \partial_\int \calD$ 
-- the law obtained is that of~$c + \Phi_\calD$, and the results above adapt readily. 
Versions of the results above may also be deduced for ``flat'' boundary conditions, 
that is boundary conditions whose maximum and minimum differ by at most a constant, independently of~$\calD$. 
The results for such boundary conditions may be obtained using the FKG inequality for the height function;
we refer the reader to the upcoming paper \cite{DCKaMaOu} for formulations and proofs of such results in a slightly different context. 

In addition to the theorem above, RSW-type statements may be proved for ~$\Phi_\calD$, see Theorem~\ref{thm:RSW_heights}.
These may be used to prove bounds on the tail of~$\frac{1}{\sqrt{ \log N}}\Phi_\calD(0)$ in a domain where~$\mathrm{dist}(0,\calD^c) =N$. 
\medskip


To the best of our knowledge this is the first instance when a uniformly distributed Lipschitz function is proven to have logarithmically diverging variance. Previously known results establish that the variance is bounded (referred to as \emph{localisation}) in high dimensions~\cite{Pel17}, or when the underlying graph is a tree~\cite{PelSamYeh13b} or an expander~\cite{PelSamYeh13a}.
The conjectured convergence of the height function to the GFF indicates that localisation should also hold on lattices in dimensions three and above. 

Recently it was established in~\cite{DumGla17}  that the variance is logarithmic in a very similar setup~--- also on the hexagonal lattice, though the distribution is not uniform but instead the probability of a function~$\phi$ is proportional to~$(1/\sqrt{2})^{\#\{u\sim v\colon \phi(u) \neq \phi(v)\}}$.
This result follows from~\cite[Thm.~1]{DumGla17} for~$n=2$.

On the square lattice~$\bbZ^2$, one may also consider the related model of graph homomorphisms from~$\bbZ^2$ to~$\bbZ$, which are defined as functions on the faces of~$\bbZ^2$ restricted to differ by \emph{exactly} one between any two adjacent faces. These functions may be viewed as height functions of the six-vertex model that has parameters~$a,b,c>0$. When~$a=b=1$ and~$c>0$ is general, the height functions are weighted by~$c^{n_5+n_6}$, where~$n_5+n_6$ is the number of vertices of~$\bbZ^2$ for which the four adjacent faces contain only two values. For the uniform model~$c=1$ (termed square ice) a non-quantitative delocalisation result is proved in \cite{ChaPelSheTas18} based on an approach described in \cite{She05}. 
In \cite{DumHarLasRauRay18} a dichotomy theorem similar to our Theorem~\ref{thm:dicho} is developed and logarithmic delocalisation is shown. In~\cite{GlaPel18} logarithmic delocalisation at~$c=2$ and localisation for~$c>2$ are shown, based on the Baxter--Kelland--Wu coupling~\cite{BaxKelWu76} with the random-cluster model and results of \cite{DumSidTas17} and~\cite{DumGanHar16}, where the order of the phase transition in the latter model is computed.
In the upcoming \cite{DCKaMaOu}, the logarithmic delocalisation result is generalised to all~$c \in [1,2]$. 

Convergence of the height function of the dimer model to the GFF was proven in a seminal work by Kenyon~\cite{Ken00} and was recently extended to the case of a weak interaction~\cite{GiuMasTon17}. On the square lattice, this corresponds to graph homomorphisms to~$\bbZ$ with~$c \approx \sqrt{2}$. Proving convergence of delocalised discrete-valued height functions outside of the free-fermion solution remains a major open problem.

The case of real-valued height functions is better understood. In particular, convergence to the GFF was established for uniformly convex symmetric potentials (under additional regularity assumptions)~\cite{Mil10} and the delocalisation was proven for some non-convex nearest-neighbour potentials~\cite{MilPel15}.

\subsection{The loop~$O(2)$ model}\label{sec:loop_model_intro}

Let~$\calD$ be a domain of~$\bbH$. 
A \emph{loop configuration on~$\calD$} is a subgraph of~$\calD$ in which every vertex has even degree. 
Thus, a loop configuration is a disjoint union of loops (i.e., subgraphs which are isomorphic to cycles) that are contained entirely in~$\calD$. In particular, none of these loops contain edges of~$\partial_E\calD$. Denote by~$\calL(\calD)$ the set of all loop configurations on~$\calD$.

The loop~$O(2)$ model on~$\calD$ with edge-weight~$1$ (and empty boundary conditions) is the measure~$\bbP_{\calD}$ on~$\calL(\calD)$ given by 
\[
	\bbP_{\calD}(\omega) = \frac1{Z(\calD)}\,2^{\ell(\omega)},
\]
where~$\ell(\omega)$ is the number of loops in~$\omega$.
The normalising constant~$Z(\calD)$, chosen so that~$\bbP_{\calD}$ is a probability measure, is called the partition function.

\begin{figure}
    \begin{center}
        \includegraphics[width=0.65\textwidth, page=6]{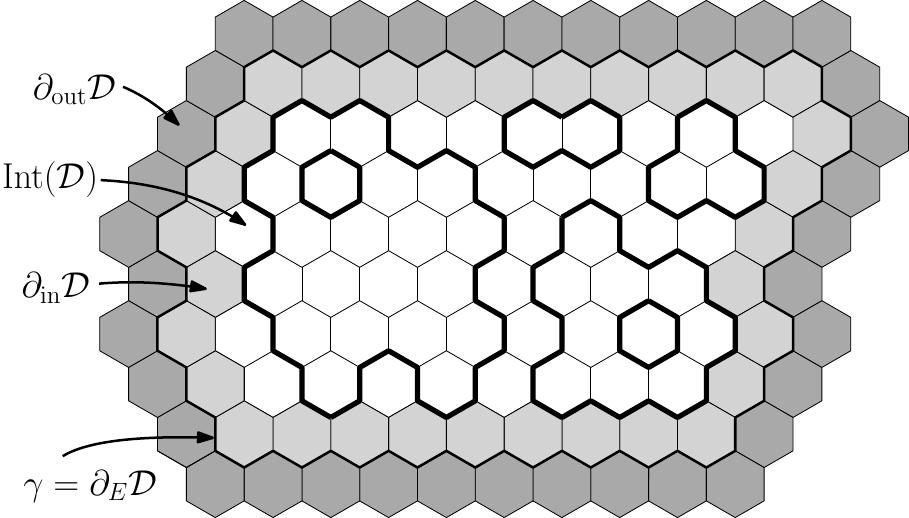}
        \caption{A loop configuration in a domain~$\calD$ bounded by the path~$\gamma$.
        The domain~$\calD$ is formed of all the edges {\em strictly} in the interior of~$\gamma$. 
        Loops inside~$\calD$ are contained in the interior of~$\gamma$ and are not allowed to intersect~$\gamma$. 
        The hexagons or~$\partial_{\int} \calD$ and~$\partial_{\out}\calD$ are marked by light and dark gray, respectively.}
        \label{fig:domain}
    \end{center}
\end{figure}

Write~$\La_n$ for the domain defined by a self-avoiding contour going around the set of faces at distance~$n$ from~$0$ (for the graph distance on the dual~$\bbH^*=\bbT$ of~$\bbH$). 
A sequence of domains~$(\calD_n)_{n\geq 1}$ is said to converge to~$\bbH$ if, for all~$k$, all except finitely many domains of~$(\calD_n)_{n\geq 1}$ contain~$\La_k$. 

\begin{thm}[Existence of Gibbs measure and delocalisation]\label{thm:loops}~
	\begin{itemize}
		\item[(i)] For any increasing sequence of domains~$(\calD_n)_{n\geq 1}$ converging to~$\bbH$,
		$\bbP_{\calD_n}$ has a limit denoted by~$\bbP_\bbH$.
		\item[(ii)] The measure~$\bbP_{\bbH}$ is supported on even subgraphs of~$\bbH$ that contain only finite loops.
		\item[(iii)] The measure~$\bbP_{\bbH}$ is ergodic and invariant under translations and rotations by~$\pi/3$.
		\item[(iv)] There exists~$c > 0$ such that, for any even integer~$n$ and any finite domain~$\calD$ containing~$\La_{n}$, 
		or for~$\calD = \bbH$,
		\begin{align}\label{eq:RSW_for_loops}
			\bbP_{\calD}(\text{there exists a loop in~$\La_{n}$ surrounding~$\La_{n/2}$ }) &\geq c.
		\end{align}
		Moreover, there exists~$\rho < 1$, such that, for any finite domain~$\calD$, if we set~$n = \mathrm{dist} (0,\partial_E \calD)$, 
		we have 
		\begin{align}\label{eq:RSW_for_loops_ub}
			\bbP_{\calD}(\text{there exist two loops surrounding~$\La_{\rho n}$}) &\leq 1 - c.
		\end{align}
		\item[(v)]
		Write~$N_\calD$ for the number of loops surrounding~$0$, contained in some domain~$\calD$. 
		There exist constants~$c,C >0$ such that for any domain~$\calD$, 
		\begin{align}\label{eq:E_surrounding_loops}
			c  \, \log\mathrm{dist}(0,\calD^c) \leq \bbE_{\calD}(N_\calD) \leq  C  \, \log\mathrm{dist}(0,\calD^c),
		\end{align}
		where~$\bbE_\calD$ denotes the expectation with respect to~$\bbP_\calD$.
		The same holds if we replace~$\bbE_\calD(N_\calD)$ with~$\bbE_\bbH(N_\calD)$. 
		In particular,~$\bbP_\bbH$-a.s., there are infinitely many loops surrounding the origin.
	\end{itemize}
\end{thm}

Any limit of measures of the type~$\bbP_{\calD}$ is supported on even subgraphs of~$\bbH$. 
Such graphs are in general disjoint unions of loops and infinite paths on~$\bbH$. 
Thus, point \textit{(ii)} of the above states that no infinite path exists~$\bbP_\bbH$-a.s.

Point \textit{(iv)} of the theorem above resembles a RSW-type statement for the loops of the~$O(2)$ model;
indeed, it stems from an actual RSW result for a related model (see Corollary~\ref{cor:RSW_strong}).
Due to the imperfect correspondence between the models, the upper bound of~\eqref{eq:RSW_for_loops_ub} takes this slightly odd form. 
We believe that a similar bound should apply to any~$\rho < 1$ (with~$c$ depending on~$\rho$), 
for any~$n$ and for a single loop instead of two. This statement is of an independent interest as it would in particular imply, via Aizenman--Burchard~\cite{AizBur99}, tightness of interfaces under Dobrushin 0/1 boundary conditions.


Point \textit{(v)} is a direct consequence of \textit{(iv)}. 
Moreover, bounds on the deviation of~$N_\calD$ from~$\log\mathrm{dist}(0,\calD^c)$ may be obtained in a straightforward manner. 
\smallskip 

\newcommand{\bc}{\mathrm{bc}}

Finally, we discuss the issue of Gibbs measures for the loop~$O(2)$ model. 
Consider a measure~$\eta$ on~$\{0,1\}^{\bbH}$ supported on even configurations.
Recall that these are disjoint unions of bi-infinite paths and finite loops. 
Let~$\omega$ be a configuration in the support of~$\eta$ and~$\calD$ be a finite domain. 
Then~$\omega \cap \calD^c$ induces certain connections between the vertices of~$\partial_E\calD$. 
Indeed, each such vertex may be connected to another such vertex, to infinity, or be isolated. 
These connections constitute a boundary condition on~$\calD$. Formally we describe boundary conditions as follows. 

For~$\xi_1$ and~$\xi_2$ two restrictions to~$\calD^c$ of even configurations on~$\bbH$, 
write~$\xi_1 \sim \xi_2$ if they induce the same connections on~$\partial_E \calD$.  
A measure~$\eta$ on even configurations on~$\bbH$ is called a Gibbs measure for the loop~$O(2)$ model with edge weight~$1$ if, 
for any finite domain~$\calD$ of~$\bbH$ and any restriction~$\xi$ of an even configuration to~$\calD^c$, 
\begin{align}\label{eq:DLR}\tag{DLR}
    \eta\big(\omega \cap \calD = \omega_0 \,\big|\, \omega\cap\calD^c \sim \xi\big)
    = \frac{1}{Z_\calD^{\xi}} \,2^{\#\text{finite loops of~$\omega_0 \cup \xi$ that intersect~$\calD$}} 
   \, \ind_{\{\text{$\omega_0 \cup \xi$ is even}\}}
\end{align}
for all~$\omega_0 \in \{0,1\}^{E(\calD)}$.
The above equation needs only to hold when the conditioning is not degenerated. 
Write~$\bbP_\calD^{\xi}$ for the measure on~$\{0,1\}^{E(\calD)}$  described by the right-hand side above; 
it is the loop~$O(2)$ measure on~$\calD$ with boundary conditions~$\xi$. 
It is immediate that~$\bbP^\xi_\calD$ does not depend on the choice of~$\xi$ within its equivalency class for~$\sim$. 

Notice that the infinite paths do not contribute to the right-hand side of \eqref{eq:DLR}.
One may be tempted to add a term of the form~$(n')^{\#\text{infinite paths of~$\omega_0 \cup \xi_0$ that intersect~$\calD$}}$ 
in~\eqref{eq:DLR} for some~$n' >0$. 
This would be superfluous, as the number of infinite paths intersecting~$\calD$ is imposed by the boundary conditions. 

\begin{thm}[Uniqueness of Gibbs measure]\label{thm:Gibbs}
	There exists only one Gibbs measure for the loop~$O(2)$ model on~$\bbH$ with edge-weight~$1$, namely~$\bbP_{\bbH}$.
\end{thm}

In particular, any Gibbs measure is supported on configurations formed entirely of finite loops. 
Notice that we do not require that the Gibbs measure be translation invariant or ergodic for it to be equal to~$\bbP_\bbH$. 
However, we do not claim that for any sequence of domains~$\calD_n$ converging to~$\bbH$ and any sequence of boundary conditions~$\xi_n$ on these domains,~$\bbP_{\calD_n}^{\xi_n}$ tends to~$\bbP_\bbH$. 
This is a stronger statement than Theorem~\ref{thm:Gibbs}; we believe it to be true, but have no proof.
It may appear surprising, but limits of measures~$\bbP_{\calD_n}^{\xi_n}$ need not be Gibbs in the sense of \eqref{eq:DLR}. 
Theorem~\ref{thm:Gibbs} does not imply the uniqueness of the Gibbs measure for height functions. 
We will discuss more on this point in the following section. 
\smallskip

The loop model studied here is part of the larger class of loop~$O(n)$ models with edge weight~$x$, where~$n$ and~$x$ are positive parameters.
The loop~$O(n)$ model with edge-weight~$x$ in a domain~$\calD$ is the measure on loop configuration given by
\[
	\bbP_{\calD,n,x}(\omega) = \frac1{Z(\calD,n,x)}\,x^{|\omega|}n^{\ell(\omega)},
\]
where~$|\omega|$ is the number of edges in~$\omega$ and~$Z(\calD,n,x)$ is called the partition function. 

Results similar to Theorems~\ref{thm:loops} and~\ref{thm:Gibbs} were proved in~\cite[Theorems~1 and~2]{DumGla17} 
for the loop~$O(n)$ model with~$n \in [1,2]$ and~$x = \frac{1}{\sqrt{2 + \sqrt{2-n}}}$.
They are based on the (single) spin representation of the loop~$O(n)$ model, 
which is shown to satisfy the FKG inequality for~$n\geq 1$ and~$x\leq 1/\sqrt{n}$.
This is then used to prove a dichotomy similar to our Theorem~\ref{thm:dicho}.
For~$x = \frac{1}{\sqrt{2 + \sqrt{2-n}}}$ and~$n \in [1,2]$, the parafermionic observable is then used to exclude exponential decay of loops, 
thus proving the equivalent of Theorem~\ref{thm:loops}.
The uniqueness of the Gibbs measure is shown via the stronger statement which we are unable to prove here: 
convergence to the unique infinite-volume measure of finite-volume measures on any increasing sequence of domains, with any boundary conditions. 

The point~$n=2$,~$x=1$ is clearly outside of the FKG regime determined in \cite{DumGla17}, 
and a more complicated spin representation is required.
This representation will involve two spin configurations, and will therefore be sometimes referred to as the double-spin representation (see Section~\ref{sec:red_blue} for precise definitions).

Let us also mention that \cite{DumPelSamSpi17} proves that for~$n$ large enough and any~$x >0$, the loops of the loop~$O(n)$ model with edge-weight~$x$ exhibit exponential decay. 
Moreover, for~$n\,x^6$ large enough, it is shown that at least three distinct, linearly independent infinite-volume Gibbs measures exist. 
For~$n\,x^6$ small enough (and~$n$ large) it was shown in the same paper that at least one Gibbs measure exists, but its uniqueness (though expected) was not proved. 

\subsection{Relation between the loop~$O(2)$ model and random Lipschitz functions}

Fix a domain~$\calD$. 
For a Lipschitz function~$\varphi$ on~$\calD$, define an edge configuration~$\omega_\varphi$ by~$\omega_\varphi(e) = 1$ 
if and only if the two faces separated by~$e$ have different values of~$\varphi$. 
It is straightforward to check that~$\omega_\varphi$ is indeed a loop configuration.

\begin{prop}\label{prop:lip-to-loop}
	\begin{itemize}
	\item[(i)] If~$\Phi$ has law~$\pi_\calD$, then~$\omega_\Phi$ has law~$\bbP_\calD$.
	\item[(ii)] Given some loop configuration~$\omega$, the law of~$\Phi$ conditionally on~$\omega_\Phi = \omega$ is obtained as follows: 
	define~$\overrightarrow{\omega}$ by choosing a clockwise or a counter-clockwise orientation uniformly and independently for each loop~$\ell$ of~$\omega$. Then, for each face~$u$ of~$\calD$, set~
	\begin{equation}\label{eq:ol-to-lip}
		\Phi(u) = \ell_\circlearrowright(\overrightarrow{\omega}; u) - \ell_\circlearrowleft(\overrightarrow{\omega}; u),
	\end{equation}
	where~$\ell_\circlearrowright(\overrightarrow{\omega}; u)$ and~$\ell_\circlearrowleft(\overrightarrow{\omega}; u)$ 
	stand for the number of clockwise (resp. counter-clockwise) oriented loops of~$\overrightarrow{\omega}$ surrounding~$u$.
	\end{itemize} 
\end{prop}

\begin{proof}
	The correspondence between oriented loop configurations and Lipschitz functions defined by~\eqref{eq:ol-to-lip} is in fact a bijection.  
	Indeed, the reverse mapping can be defined as follows: given a Lipschitz function~$\varphi : \calD \to \bbZ$, 
	each loop of the corresponding (unoriented) loop configuration~$\omega_\varphi$ is oriented 
	clockwise if the values of~$\varphi$ inside of the loop are higher than those outside, 
	and is oriented counter-clockwise otherwise. 
	
	The push-forward of~$\pi_\calD$ under this bijection is a uniform measure on all oriented loop configurations on~$\calD$. Considering the projection on the set of unoriented loop configurations we obtain~$\bbP_\calD$, since each loop has two possible orientations. This proves~$(i)$, and~$(ii)$ follows readily.
\end{proof}

Using the correspondence between Lipschitz functions and loop configurations described above, Theorem~\ref{thm:lipschitz} follows easily from Theorem~\ref{thm:loops}. 

\begin{proof}[Theorem~\ref{thm:lipschitz} (assuming Theorem~\ref{thm:loops})]
	\textbf{\textit{(i)}} By Proposition~\ref{prop:lip-to-loop}~\textit{(ii)},
	a random Lipschitz function~$\Phi_\calD$ distributed according to~$\pi_\calD$  
	can be sampled from a random loop configuration~$\omega$ distributed according to~$\bbP_\calD$ 
	by orienting each loop of~$\omega$ uniformly and independently. 
	Then~$\Phi_\calD(0)$ has the distribution of a simple random walk on~$\bbZ$ with~$N_\calD$ steps,  
	where~$N_\calD$ is the number of loops in~$\omega$ surrounding~$0$. 
	Thus,~$\textrm{Var}(\Phi_\calD(0)) = \bbE_\calD(N_\calD)$.
	The conclusion follows from~\eqref{eq:E_surrounding_loops}.
        \smallskip
	
	\noindent\textbf{\textit{(ii)}}
	Using the coupling from Proposition~\ref{prop:lip-to-loop}, we get that for any~$u\in F(\calD)$, the value of~$\Phi_\calD(u) - \Phi_\calD(0)$ is a function of number of loops separating~$u$ from~$0$ and their orientations. By items~$(i)$ and~$(ii)$ of Theorem~\ref{thm:loops} the infinite-volume limit of~$\bbP_\calD$ exists and consists only of finite loops. Thus, the infinite-volume limit of~$\Phi_\calD - \Phi_\calD(0)$ also exists.\smallskip
	
	\noindent\textbf{\textit{(iii)}}
	We will prove the statement for~$\Phi_\bbH$; that for~$\Phi_\calD$ is proved in the same way. 
	Similarly to the previous items, we have
	\begin{align}\label{eq:Var_NxNy}
		\mathrm{Var}(\Phi_\bbH(x) - \Phi_\bbH(y)) = \bbE_\bbH(N_{x \setminus y} + N_{y \setminus x}),
	\end{align}
	where~$N_{x\setminus y}$ stands for the number of loops surrounding~$x$ but not~$y$ 
	and~$N_{y \setminus x}$ for those surrounding~$y$ but not~$x$.

	For the lower bound, notice that~$N_{x\setminus y}$ is larger than the number of loops surrounding~$x$ and contained in~$\La_{|x-y|}$. 
	Thus, by~\eqref{eq:E_surrounding_loops},~$\bbE_\bbH(N_{x \setminus y}) \geq c \log |x-y|$ for some universal constant~$c > 0$.
	The desired lower bound on~$\mathrm{Var}(\Phi_\bbH(x) - \Phi_\bbH(y))$ follows. 
	
	For the upper bound, define~$\Gamma$ to be the outermost loop surrounding~$x$ but not~$y$, provided such a loop exists. 
	Let~$\gamma$ be a possible realisation of~$\Gamma$ and let~$\Int(\gamma)$ be the interior of the domain delimited by~$\gamma$. 
	Notice that the event~$\{\Gamma = \gamma\}$ is measurable in terms of the configuration on and outside~$\gamma$.
	Therefore, conditionally on~$\Gamma = \gamma$, the restriction of~$\bbP_\bbH$ to~$\Int(\gamma)$ is
	the uniform measure among all loop configuration in~$\Int(\gamma)$, which is to say it is equal to~$\bbP_{\Int(\gamma)}$. 
	Thus 
	\begin{align*}
		\bbE_\bbH(N_{x\setminus y}) = \bbP( N_{x\setminus y} > 0) + \sum_\gamma \bbE_{\Int(\gamma)}(N_{\Int(\gamma)}(x))\cdot \bbP_\bbH(\Gamma = \gamma),
	\end{align*}
	where the sum is over all possible realisations~$\gamma$ of~$\Gamma$ and~$N_{\Int(\gamma)}(x)$ in the right hand side 
	stands for the number of loops surrounding~$x$ and contained in~$\Int(\gamma)$. 
	Now, since~$y \notin \Int(\Gamma)$,~$\mathrm{dist}(x,\Int(\gamma)^c) \leq |x-y|$ for any path~$\gamma$ appearing in the sum. 
	Thus,~\eqref{eq:E_surrounding_loops} proves that~$\bbE_\bbH(N_{\Int(\gamma)}(x)) \leq 1 + C \log |x-y|$ for some universal constant~$C$. 
	
	The same holds for~$\bbE_\bbH(N_{y\setminus x})$. Using this and~\eqref{eq:Var_NxNy}, we obtain the desired upper bound on~$\mathrm{Var}(\Phi_\bbH(x) - \Phi_\bbH(y))$.
\end{proof}

Let us briefly comment on the uniqueness of infinite-volume measures for Lipschitz functions. 
One may think that, due to Theorem~\ref{thm:Gibbs}, 
$\pi_\bbH$ should be the only infinite-volume measure 
with the property that its restriction to any finite domain is uniform among Lipschitz functions that take the value~$0$ at the origin. 
This is not the case. Indeed, the correspondence between the loop and Lipschitz functions models is not perfect, and does not allow us to deduce this. 

Moreover the claim is false, as an infinite family of infinite-volume measures for uniform Lipschitz functions is expected to exist, one for each global ``slope''.
The loop representation of any of these contains infinite paths and is not Gibbs in the sense of~\eqref{eq:DLR}.

\paragraph{Structure of the paper} 
The rest of the paper is entirely dedicated to the loop~$O(2)$ model with~$x = 1$. 
In Section~\ref{sec:red_blue} we derive a representation of the loop model in terms of two loop~$O(1)$ configurations conditioned not to intersect. 
These are in turn represented in terms of spin configurations that are shown to satisfy the FKG inequality and a certain form of Spatial Markov property. 

In Section~\ref{sec:infinite_volume} this spin representation is used to construct an infinite-volume, ergodic loop~$O(2)$ measure. 
The infinite-volume measure is then shown to be unique (in some sense that will be made precise later). 
In doing so, we show that~$0$ is surrounded by infinitely many loops. 
For height functions, this corresponds to the delocalisation of~$\Phi(0)$ or equivalently to the divergence of covariances. 
At this stage, the delocalisation/divergence is not quantitative.

Section~\ref{sec:dichotomy} contains a dichotomy theorem.
In the language of uniform Lipschitz functions, the dichotomy theorem roughly states that the covariance between two points 
either is bounded or diverges logarithmically in the distance between the  points. 

Finally, in Section~\ref{sec:macro}, 
the non-quantitative delocalisation result and the dichotomy theorem are used to prove Theorem~\ref{thm:loops}.
Theorem~\ref{thm:Gibbs} is also proved here. Moreover, we provide an RSW result for height functions in Section~\ref{sec:RSW_heights}. 

The paper is structured so as to isolate the different ingredients of our argument;
some of them may be useful for the analysis of the loop~$O(n)$ model with other values of~$n$ and~$x$, or other similar models. 
We further discuss in Section \ref{sec:spin_representation} the various properties of the loop~$O(2)$ model that are necessary for our proof. 

\paragraph{Notation} 
Below is a list of notation used throughout the paper. Some of it was already mentioned, some is new. 

Recall that~$\bbH$ denotes the hexagonal lattice; its dual is the triangular lattice, written~$\bbH^* = \bbT$. 
We will call edge-path any finite or infinite sequence of adjacent edges of~$\bbH$ with no repetitions. 
A face-path is a sequence of adjacent faces of~$\bbH$ with no repetitions, 
or equivalently it is a path on~$\bbT$ that does not visit the same vertex twice. 

Domains~$\calD = (V(\calD),E(\calD))$ are interior of edge-polygons of~$\bbH$.
The edges of the polygon form the edge-boundary of~$\calD$, written~$\partial_E\calD$.
The faces of~$\bbH$ adjacent to~$\partial_E\calD$ and inside~$\partial_E\calD$ (outside, respectively) 
form the inner face boundary of~$\calD$, written~$\partial_\int \calD$ (and the outer face-boundary written~$\partial_\out \calD$, respectively). 
The set of faces of~$\bbH$ inside~$\partial_E\calD$ is written~$F(\calD)$; 
those not adjacent to~$\partial_E\calD$ form the interior of~$\calD$, denoted by~$\Int(\calD) = F(\calD) \setminus \partial_\int\calD$. 
The dual~$\calD^* = (V(\calD^*), E(\calD^*))$ of~$\calD$ is the induced subgraph of~$\bbT$ with vertex set~$F(\calD)$.  

An edge configuration on~$\calD$ is an element~$\omega \in \{0,1\}^{E(\calD)}$; 
it is identified to the graph with vertex set~$V(\calD)$ and edge-set~$\{e \in E(\calD) :\, \omega(e) = 1\}$. 
Write~$u \xlra{\omega} v$ to indicate that two vertices~$u,v$ of~$V(\calD)$ are connected in~$\omega$. 
The same notation applies to~$\bbH$ and~$\calD^*$. 

A spin configuration on~$\calD$ is an element~$\sigma \in \{-,+\}^{F(\calD)}$; the notation extends to~$\bbH$. 
Below we will use two superposing spin configurations. 
We identify one as red, the other as blue and denote the relevant spins by~$\{\rm,\rp\}$ and~$\{\bm,\bp\}$ for legibility. 

For a red-spin configuration~$\sigma \in \{\rm,\rp\}^{F(\calD)}$ and two faces~$u,v \in F(\calD)$, 
write~$u \xlra{\srp} v$ (or~$u \xlra{\srp \text{ in~$\calD$}} v$ when the choice of~$\calD$ is unclear) 
to indicate that there exists a face-path in~$\calD$ starting at~$u$ and ending at~$v$, formed entirely of faces with~$\sigma$-spin~$\rp$. 
Such a path will be called a~$\rp$-path or simple-$\rp$ path. Connected components for this notion of connectivity are called~$\rp$-clusters. 

A double-$\rp$ path will be an edge-path for which all adjacent faces have spin~$\rp$;
connections by double-$\rp$ paths will be denoted by~$\xlra{\srp\srp}$.
The same applies to spins~$\rm, \bp$ and~$\bm$. 

Write~$\nxlra{}$ for the negation of~$\lra$.

\paragraph{Acknowledgements} 
The authors would like to thank Hugo Duminil-Copin for numerous discussions and tips, 
especially concerning the dichotomy theorem of Section \ref{sec:dichotomy}, and Ron Peled for suggesting to develop the loop-weight~$2 = 1 + 1$ following Chayes and Machta.
Our conversations with Matan Harel, Marcelo Hilario, and Nick Crawford also were very helpful. 
We acknowledge the hospitality of IMPA (Rio de Janeiro), where this project started. 

The first author is supported by the Swiss NSF grant P300P2\_177848, and partially supported by the European Research Council starting grant 678520 (LocalOrder).
The second author is a member of the NCCR SwissMAP.

\section{1+1 = 2}\label{sec:red_blue}

Fix a domain~$\calD$. Choose a loop configuration~$\omega$ according to~$\bbP_{\calD}$ and colour each loop of~$\omega$ in either red or blue, with equal probability, independently for each loop. Extend~$\bbP_\calD$ to include this additional randomness. 
Write~$\omega_r$ and~$\omega_b$ for the configurations of blue and red loops. Then, for any two disjoint loop configurations~$\omega_r,\omega_b$, 
\begin{align*}
   	\bbP_\calD(\omega_r,\omega_b) 
	&=\frac{1}{Z(\calD)}2^{\ell(\omega)} \big(\tfrac12\big)^{\ell(\omega_r)} \big(\tfrac12\big)^{\ell(\omega_b)}= \frac{1}{Z(\calD)}.
\end{align*}
In other words,~$\bbP_\calD$ is the uniform distribution on pairs of loop configurations~$(\omega_r,\omega_b)$ that do not to intersect. 

In the context of Lipschitz functions, one may think of~$\omega_r$ as the level lines with higher value on the inside and~$\omega_b$ as those with higher value on the outside 
(that is the clockwise and counter-clockwise, respectively, oriented loops in the language of~\eqref{eq:ol-to-lip}).
 While accurate, this interpretation is not relevant below.

Keeping the idea of colouring loops as the intuition, in the next section we introduce a measure on pairs of red and blue~$\pm 1$ spin configurations. Though this measure is tightly linked to the loop~$O(2)$ measure on pairs of red and blue loops and under certain boundary conditions these two measures will be shown to coincide, 
we emphasise that this is not always the case.

To shorten notation, we will use the symbols~$\rm,\rp$ to denote the values of red spins and~$\bm,\bp$ for blue spins. 

\subsection{Spin representation}\label{sec:spin_representation}

Define~$\mu_\calD$ to be the uniform measure on all pairs of spin configurations~$\sigma_r\in\{\rm,\rp\}^{F(\calD)}$ and~$\sigma_b\in\{\bm,\bp\}^{F(\calD)}$ such that for every two adjacent faces~$u,v\in F(\calD)$ at least one of the equalities~$\sigma_r(u)=\sigma_r(v)$ and~$\sigma_b(u)=\sigma_b(v)$ holds. We call such configurations~$\sigma_r$,~$\sigma_b$ \emph{coherent} and denote this relation by~$\sigma_r\perp\sigma_b$.

Given a spin configuration~$\sigma\in\{\pm 1\}^\calD$, define~$\omega(\sigma)$ to be set of edges of~$\calD$ separating adjacent faces bearing different spin in~$\sigma$. Then~$\omega(\sigma)$ consists of disjoint loops and paths linking boundary vertices in~$\calD$.

The correspondence~$\sigma\mapsto \omega(\sigma)$ is a classical tool in the study of the Ising model, called the high temperature representation (see for instance \cite[Sec.~3.10.1]{FriVel17}). If~$\sigma$ is chosen according to a Ising distribution, then~$\omega(\sigma)$ has the law of a loop~$O(1)$ model, with parameter~$x$ depending on the temperature of the Ising measure. 
For the loop~$O(n)$ model with general values of~$n$, this correspondence was used in~\cite{DumGla17} with the name cluster representation.

The following proposition describes the relation between~$\mu_\calD$ and~$\bbP_{\calD}$. Define~
\begin{align}\label{eq:srpsrm_def}
	\mu^{\srp\srm}_\calD:=
	\mu_\calD(\cdot\,|\,\sigma_r \equiv \rp \text{ on } \partial_\int\calD \text{ and } \sigma_b \equiv \bp \text{ or } \sigma_b \equiv \bm \text{ on } \partial_\int\calD),
\end{align}
where~$\equiv$ should be understood as ``equal everywhere to''. The notation~$\mu^{\srp\srm}_\calD$ comes from Theorem~\ref{thm:DMP}, where these boundary conditions are shown to be equivalent to setting~$\rp$ on the interior boundary of~$\calD$ and~$\rm$ on its exterior boundary.

\begin{prop}\label{prop:loop-to-spin}
	If the couple~$(\sigma_r,\sigma_b)$ has law~$\mu^{\srp\srm}_\calD$, then the couple~$(\omega(\sigma_r), \omega(\sigma_b))$ has law~$\bbP_\calD$. In particular~$\omega(\sigma_r)\cup \omega(\sigma_b)$ has the law of the loop~$O(2)$ model on~$\calD$. 
\end{prop}

\begin{proof}
	The map~$\sigma_r\mapsto \omega(\sigma_r)$ is a bijection between spin configurations~
	$\sigma_r\in\{\rm,\rp\}^{F(\calD)}$ that are equal to~$\rp$ on~$\partial_\int\calD$ and all loop configurations on~$\calD$.
	Indeed, due to the constant spin of~$\sigma_r$ on~$\partial_\int\calD$,~$\omega(\sigma_r)$ is indeed a loop configuration. 
	Moreover, the reverse mapping is the following: a loop configuration~$\omega$ on~$\calD$ is mapped to the spin configuration~$\sigma_r\in\{\rm,\rp\}^{F(\calD)}$ that is equal to~$\rp$ (resp.~$\rm$) at all faces of~$\calD$ that are surrounded by an even (resp. odd) number of loops of~$\omega$. 
	
	Similarly, the map~$\sigma_b \mapsto \omega(\sigma_b)$ 
	defined on the set of spin configurations~$\sigma_b\in\{\bm,\bp\}^{F(\calD)}$ that are constant on~$\partial_\int\calD$
	and taking values in~$\calL(\calD)$ is two to one, due to its invariance under global spin flip.
	
	The condition~$\sigma_r\perp\sigma_b$ corresponds to~$\omega(\sigma_r)\cap \omega(\sigma_b) = \emptyset$. Thus,~$\mu^{\srp\srm}_\calD$  induces a uniform measure on all pairs~$(\omega(\sigma_r), \omega(\sigma_b))$ of non-intersecting red and blue loop configurations on~$\calD$, that is~$\bbP_\calD$. As described above, the marginal of this measure on the non-coloured loop configuration~$\omega(\sigma_r)\cup \omega(\sigma_b)$ is the loop~$O(2)$ measure on~$\calD$. 
\end{proof}

\begin{rem}\label{rem:x_neq_1}
	Extensions of the statement to all~$x \neq 1$ are possible and result in non-uniform measures on pairs of spin configurations. 
	As already mentioned, the correspondence between double-spins and loops 
	does not extend to general boundary conditions for the loop~$O(2)$ model. 
\end{rem}

We will show below that, under~$\mu_\calD$, the marginals~$\sigma_r$ and~$\sigma_b$ satisfy the FKG inequality. 
Moreover the spin measures of the type~$\mu_\calD$ satisfy the Spatial Markov property in the following sense. 
If~$\calD'$ is a domain contained in some larger domain~$\calD$, 
then the restriction of~$\mu_\calD$ to~$\calD'$, conditionally on the spins~$\sigma_r, \sigma_b$ outside~$\calD'$, 
is entirely determined by the values of~$\sigma_r$ and~$\sigma_b$ on~$\partial_\out\calD$.

It may be tempting to think that these two observations suffice to apply the techniques developed for the random-cluster model to our setting (such as those of \cite{DumSidTas17,DumRaoTas18}). Unfortunately this is easier said than done.  
Indeed, many of these techniques use a form of monotonicity of boundary conditions. 
In our case, it is unclear how to compare boundary conditions consisting of pairs of spins, 
as the FKG inequality applies only individually to the single-spin marginals of~$\mu_\calD$. 

To circumvent this difficulty, we will focus our study on one of the single-spin marginals of~$\mu_\calD$; 
we arbitrarily choose the red-spin marginal, and call it~$\nu_\calD$. 
As already stated, this measure satisfies the FKG inequality, but fails to have a general spatial Markov property. 
However, we show in Theorem~\ref{thm:DMP} and Corollary~\ref{cor:DMP} that a limited version of the spatial Markov property applies to~$\nu_\calD$, under certain restrictions.
\smallskip 

One may attempt to apply the same strategy to other values of~$n$ and~$x$. 
Our argument is quite intricate, and different parts of it use different properties of the double spin representation described above.
The paper is organised to separate the different arguments, so as to facilitate the identification of blocks that may be applied to other models. 
Below is brief list of the essential properties of the double spin representation and their uses.
\begin{itemize}
\item The FKG inequality for the red-spin marginal is crucial and is used extensively throughout the proof. 
As mentioned in Remark~\ref{rem:FKG}~\textit{(iii)}, the FKG inequality extends to the red-spin marginal of a certain double spin representation of the loop~$O(n)$ model with parameters~$n \geq 2$  and~$x \leq \frac{1}{\sqrt{n-1}}$. 
\item That~$x = 1$ is essentially only used for the spatial Markov property. In its current form, the property does not apply to~$x \neq 1$. 
\item The symmetry between the red and blue spin marginals (which, in light of Remark~\ref{rem:FKG}~\textit{(iii)} boils down to~$n =2$)
is akin to a self-duality property, and is used  to prove RSW type estimates (see Lemma~\ref{lem:square_crossed}).
\end{itemize} 
Finally, let us mention that it is expected that the loop~$O(2)$ model for~$n =2$ and~$x \geq 1/\sqrt2$ has a similar behaviour to the case~$x = 1$, that is macroscopic loops exist at every scale. However, for all~$ n>2$ and any~$x > 0$ or~$n =2$ and~$x < 1/\sqrt2$, loops are expected to  exhibit exponential decay. 
Thus, parts of our proof need to fail for more general values of~$n$ and~$x$. 
The dichotomy theorem of Section~\ref{sec:dichotomy} (or similar statements) may be expected 
to hold for all values of~$n$ and~$x$, but no proof is generally available.

\subsection{Spatial Markov property}

In general, the measures~$\nu_\calD$, that is the red-spin marginals of~$\mu_\calD$, do not have the spatial Markov property.
However, a version of this property holds in certain cases. 
Recall the definition \eqref{eq:srpsrm_def} of~$\mu^{\srp\srm}_\calD$ 
and set~
\begin{align*}
	\mu^{\srp\srp}_\calD:=\mu_\calD(\cdot\,|\,\sigma_r \equiv \rp \text{ on } \partial_\int\calD).
\end{align*}
Let~$\nu^{\srp\srm}_\calD$ and~$\nu^{\srp\srp}_\calD$, respectively, be the marginals on~$\sigma_r$ of the above two measures. 
Define the measures~$\mu^{\srm\srp}_\calD$,~$\mu^{\srm\srm}_\calD$,~$\mu^{\sbp\sbm}_\calD$ etc. in a similar ways, and write~$\nu^{\srm\srp}$ etc. their red-spin marginals. 

\begin{thm}[Spatial Markov property]\label{thm:DMP}
	Let~$\calD,\calD'\subset\bbH$ be two domains such that~$\partial_E\calD \subset E(\calD')$. Let~$\tau_r\in\{\rp,\rm\}^{F(\calD')}$ and~$\tau_b\in\{\bp,\bm\}^{F(\calD')}$ be two coherent spin configurations on~$\calD'$.
	\begin{enumerate}[label=(\roman*)]
		\item if~$\tau_r = \rp$ on~$\partial_\int\calD\cup\partial_\out\calD$, then
		\begin{align}\label{eq:domain-markov}
			\mu_{\calD'}\big[\sigma_r,\sigma_b  \,\big|\, 
				\sigma_r = \tau_r\text{ on } F(\calD') \setminus \Int(\calD), \,
				\sigma_b = \tau_b \text{ on } F(\calD') \setminus F(\calD)
			\big] \eqonD\mu_\calD^{\srp\srp};
		\end{align}
		\item if~$\tau_r = \rp$ on~$\partial_\int\calD$, 
		$\tau_r =\rm$ on~$\partial_\out\calD$ and~$s:=\tau_b(u)$ for some~$u\in\partial_\out\calD$ , then\\
		\begin{align}\label{eq:domain-markov-rm}
			& \mu_{\calD'}\big[\sigma_r,\sigma_b  \,\big|\, 
				\sigma_r = \tau_r\text{ on } F(\calD') \setminus \Int(\calD), \,
				\sigma_b = \tau_b \text{ on } F(\calD') \setminus F(\calD)
			\big] \\
			&\quad \eqonD \mu_\calD^{\srp\srm}(\cdot\,|\, \sigma_b(v) = s\text{ for some } v\in\partial_\int\calD), \nonumber
		\end{align}
	\end{enumerate}	
	where symbol~$\eqonD$ means that the two measures are equal when~$\sigma_r$ and~$\sigma_b$ are restricted to~$\calD$.
\end{thm}

\begin{proof}
	All measures under consideration are uniform over sets of coherent pairs~$(\sigma_r,\sigma_b)$ that agree with the corresponding boundary conditions. Thus, it is enough to show that the two sets corresponding to the two sides of~\eqref{eq:domain-markov}, and of~\eqref{eq:domain-markov-rm}, respectively, are equal.
	\begin{itemize}
		\item[\textit{(i)}] Consider a pair of coherent configurations~$\sigma_r\in\{\rp,\rm\}^{F(\calD)}$ and~$\sigma_b\in\{\bp,\bm\}^{F(\calD)}$ contributing to the RHS of~\eqref{eq:domain-markov}; let us show that they also contribute to the LHS. By definition,~$\sigma_r \equiv \rp$ on~$\partial_\int\calD$, which is to say that~$\sigma_r = \tau_r$ on~$\partial_\int\calD$. 
		It remains to check that, if~$\sigma_r$ and~$\sigma_b$ are completed by~$\tau_r$ and~$\tau_b$, respectively, on~$F(\calD')\setminus F(\calD)$, they are coherent on~$\calD'$.
		For edges of~$E(\calD)$ and~$E(\calD') \setminus (E(\calD) \cup \partial_E\calD)$, the coherence condition follow from the coherence of~$\sigma_r$ with~$\sigma_b$ and that of~$\tau_r$ with~$\tau_b$, respectively. 
		 For edges of~$\partial_E\calD$ the statement holds because both faces adjacent to each such edge are~$\rp$ in~$\sigma_r$.
		
		The reverse direction is straighforward since each pair of configurations~$\sigma_r\in\{\rp,\rm\}^{F(\calD)}$ and~$\sigma_b\in\{\bp,\bm\}^{F(\calD)}$ contributing to the LHS of~\eqref{eq:domain-markov} is coherent and satisfies~$\sigma_r \equiv \rp$ on~$\partial_\int\calD$.
		
		\item[\textit{(ii)}] The values of~$\tau_r$ imply that~$\tau_b$ is constant on~$\partial_\int\calD\cup\partial_\out\calD$. 
		Similarly, the definition of~$\mu^{\srp\srm}$ requires that~$\sigma_b$ be constant on~$\partial_\int\calD$ in the RHS of~\eqref{eq:domain-markov-rm}.	
		The values of~$\tau_b$ and~$\sigma_b$ on~$\partial_\int\calD$ are the same because of the condition~$\tau_b(u)= \sigma_b(v)=s$ for some~$u\in\partial_\int\calD$ and~$v\in\partial_\out\calD$. 
		Thus, the pairs~$(\tau_r,\tau_b)$ and~$(\sigma_r,\sigma_b)$ agree on~$\partial_\int\calD$ 
		and as a consequence these boundary values impose the same distribution on the LHS and the RHS of~\eqref{eq:domain-markov-rm}.
	\end{itemize}
\end{proof}

Summing equalities of Theorem~\ref{thm:DMP} over all possibilities for~$\sigma_b$, we get the following corollary for the red-spin marginals of the measures.

\begin{cor}[Spatial Markov property for~$\nu$]\label{cor:DMP}
	Let~$\calD, \calD'$ be two domains such that~$\partial_E\calD \subset E(\calD')$. 
	Let~$\tau_r\in\{\rp,\rm\}^{F(\calD')}$. Then the following statements hold:
	\begin{enumerate}[label=(\roman*)]
		\item if~$\tau_r \equiv \rp$ on~$\partial_\int\calD\cup\partial_\out\calD$, then 
			\[\nu_{\calD'}\big(\sigma_r \,\big|\, 
				\sigma_r = \tau_r \text{ on } F(\calD') \setminus \Int(\calD)
			\big)
		\eqonD \nu_\calD^{\srp\srp}(\sigma_r);
		\]
		\item if~$\tau_r \equiv \rp$ on~$\partial_\int\calD$ and~$\tau_r \equiv \rm$ on~$\partial_\out\calD$, then 
		\[
			\nu_{\calD'}\big(\sigma_r  \,\big|\, 
				\sigma_r = \tau_r \text{ on } F(\calD') \setminus  \Int(\calD)
			\big)
		 \eqonD \nu_\calD^{\srp\srm}(\sigma_r),
		\]
	\end{enumerate}
	where by symbol~$\eqonD$ we mean that the two measures are equal when~$\sigma_r$ is restricted to~$\calD$.
\end{cor}

\begin{rem}\label{rem:4arcs}
	It is tempting to think that the above Spatial Markov property holds for any boundary conditions on~$\partial_\int\calD\cup\partial_\out \calD$. 
	This is not the case. One significant example is that of the boundary conditions consisting of four arc of alternating spins~$\rp\rp$,~$\rm\rm$,~$\rp\rp$,~$\rm\rm$. Indeed, these boundary conditions are coherent with non-intersecting loop configurations\footnote{Here, due to the atypical boundary conditions, the configurations~$\omega_r$ and~$\omega_b$ are allowed to induce odd degrees for vertices on the boundary of~$\calD$. Then they are formed of disjoint loops and paths with endpoints on the boundary of~$\calD$. }
	$(\omega_r, \omega_b)$ where
	~$\omega_b$ contains 
	 \begin{itemize} 
	 \item paths between the arcs~$\rp\rp$,
	 \item paths between the arcs~$\rm\rm$ or
	 \item none of the above.
	 \end{itemize}
	 The three cases above are mutually exclusive. Depending on the red configuration outside~$\calD$ 
	 one or both of the first two cases may be excluded. 
\end{rem}

\subsection{FKG inequality}

In this section we show that the red-spin marginals~$\nu$ of the measures~$\mu$ satisfiy the FKG inequality.
This property is crucial to all our proofs. 
Similar properties were found in~\cite{DumGla17} for the single-spin representation of the loop~$O(n)$ for a certain range of parameters 
and in \cite{GlaPel18} for a  spin representation of height functions on~$\bbZ^2$ arising from the six-vertex model.

Fix some domain~$\calD$.
We start by introducing a partial order on~$\{\rp,\rm\}^{F(\calD)}$. Given two elements~$\sigma,\tau\in\{\rp,\rm\}^{F(\calD)}$ we say that~$\sigma \leq \tau$ if~$\sigma(u)\leq \tau(u)$ for every~$u\in F(\calD)$, where by convention~$\rm\leq \rp$. 
An event~$A\subset \{\rp,\rm\}^{F(\calD)}$ is called increasing if for any~$\sigma\in A$ and~$\tau\in\{\rp,\rm\}^{F(\calD)}$ 
such that~$\sigma \leq \tau$, we have~$\tau\in A$.

A probability measure~$\bbP$ on~$\{\rp,\rm\}^{F(\calD)}$ is said to satisfy the FKG inequality (or called positively associated) if for any two increasing events~$A,B\subset \{\rp,\rm\}^{F(\calD)}$, we have
\begin{equation}\label{eq:fkg-def}
	\bbP(A\cap B) \geq \bbP(A)\cdot \bbP(B).
\end{equation}

Recall that the marginal of~$\mu_\calD$ on the red spin configurations is denoted by~$\nu_\calD$.

\begin{thm}\label{thm:FKG}
	The measure~$\nu_\calD$ satisfies the FKG inequality~\eqref{eq:fkg-def}.
\end{thm}

Before proving the FKG inequality, let us compute~$\nu_\calD$. For a spin configuration~$\sigma$ on~$\calD$, let~$\theta(\sigma) \in \{0,1\}^{E(\calD^*)}$ be the set of all edges~$e=uv\in E(\calD^*)$ such that~$\sigma(u) \neq \sigma(v)$.
If~$\sigma$ is associated to a loop configuration~$\omega$, then~$e^*\in \theta(\sigma)$ if and only if~$e$ is present in~$\omega$. 
For readers familiar with the notion of duality in percolation (where the dual configuration is written~$\omega^*$), we mention that~$\theta(\sigma) =  (\omega^*)^c$. See Figure~\ref{fig:some_loops} for an example.  
Denote by~$k(\theta(\sigma))$ the number of connected components of~$\theta(\sigma)$;
note that isolated vertices of~$\calD^*$ (that is faces of~$\calD$) are also counted as connected components. 

\begin{prop}\label{prop:marginal-ditribution}
	\begin{enumerate}[label=(\roman*)]
		\item The law of~$\sigma_r$ under~$\mu_\calD$ is given by 
	\begin{align}
		\label{eq:red_spin_law_free}
		\nu_\calD(\sigma_r) =\frac1{Z_\calD}\, 2^{k(\theta(\sigma_r))},
	\end{align}
	where~$Z_\calD$ is a normalising constant.
		\item The law of~$\sigma_b$ on~$\calD$ under the conditional measure~$\mu_\calD(.\,|\, \sigma_r)$ 
		is obtained by colouring independently and uniformly the clusters of~$\theta(\sigma_r)$ in either~$\bm$ or~$\bp$.
	\end{enumerate}
\end{prop}

\begin{figure}
    \begin{center}
        \includegraphics[width=0.53\textwidth, page=5]{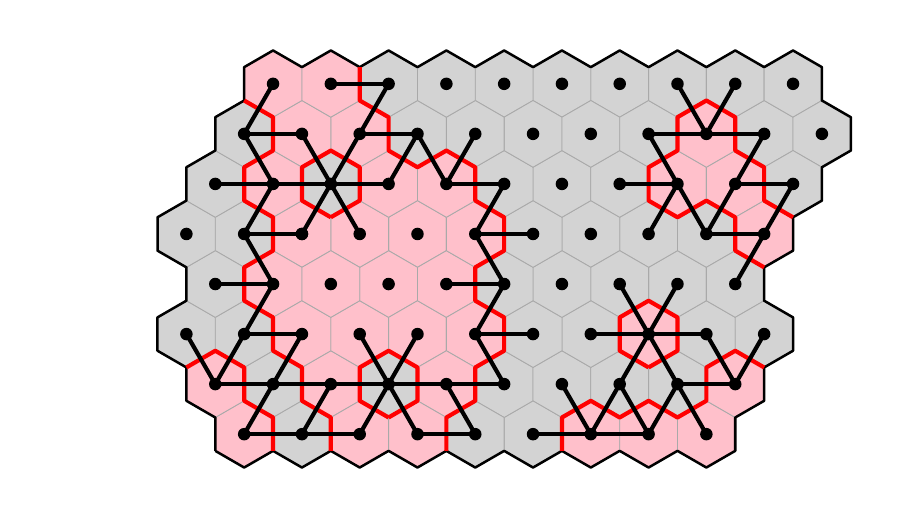} \hspace{-1.8cm}
       \includegraphics[width=0.53\textwidth, page=4]{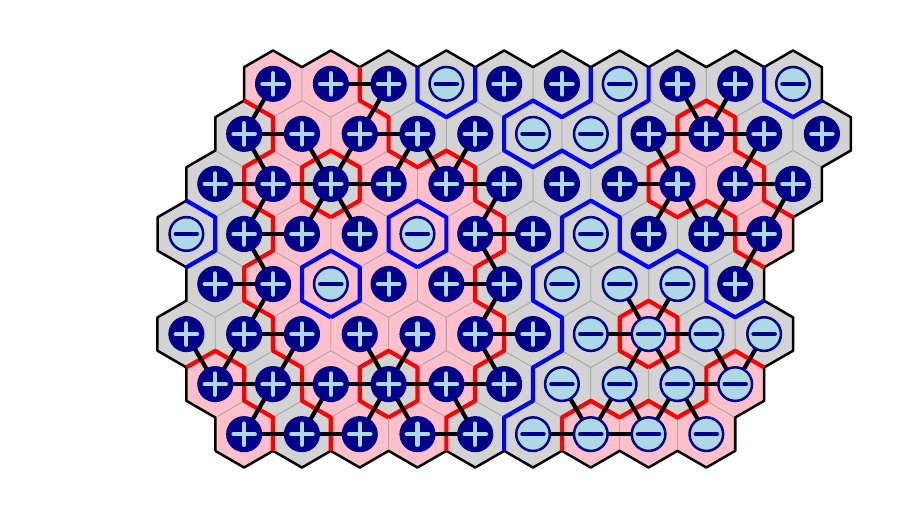}
        \caption{{\em Left:} A red-spin configuration on a domain~$\calD$ and the associated loops. 
        The pink faces correspond to red spin~$\frp$, while the gray ones to red spin~$\frm$. 
        The graph~$\theta(\sigma_r)$ is drawn in black.
        {\em Right:} A blue-spin configuration coherent with the red one.}
        \label{fig:some_loops}
    \end{center}
\end{figure}

\begin{proof}
	Let~$\sigma_r\in\{\rp,\rm\}^{F(\calD)}$ and consider any~$\sigma_b\in\{\bp,\bm\}^{F(\calD)}$ that is coherent with~$\sigma_r$. For any two faces~$u,v\in F(\calD)$ corresponding to vertices in~$V(\calD^*)$ that are connected by an edge in~$\theta(\sigma_r)$, we have~$\sigma_b(u) = \sigma_b(v)$. Thus,~$\sigma_b$ has a constant value on each connected component of~$\theta(\sigma_r)$. 
	Moreover, there is no restriction on values of~$\sigma_b$ on different connected components of~$\theta(\sigma_r)$.
	Thus, there are exactly~$2^{k(\theta(\sigma_r))}$ blue spin configurations coherent with~$\sigma_r$, and~\textit{(i)} follows readily. 
	In addition, when conditioned on~$\sigma_r$, the measure on these blue spin configurations is uniform, thus asserting~\textit{(ii)}.
\end{proof}	

\begin{rem}\label{rem:blue+}
	A straightforward adaptation of the proof above shows that, for~$\calA \subset F(\calD)$, 
	the law of~$\sigma_r$ under~$\mu_\calD(.\,|\, \sigma_b \equiv \bp \text{ on~$\calA$})$ is given by 
	$\frac1{Z}\, 2^{k_\calA(\theta(\sigma_r))}$, 
	where~$k_\calA(\theta(\sigma_r))$ is the number of connected components of~$\theta(\sigma_r)$ 
	when all components intersecting~$\calA$ are counted as a single one. 
	When~$\calA$ is connected,~$k_\calA(\theta(\sigma_r))$ may be viewed as the number of connected components of the configuration
	obtained by adding to~$\theta(\sigma_r)$ all edges between pairs of adjacent faces of~$\calA$. 
	
	As a consequence 
	\begin{align*}
    	\nu_\calD^{\srp\srp}(\sigma_r) =\frac1{Z_\calD^{\srp\srp}}\, 2^{k(\theta(\sigma_r))}
    	\ind_{\{\sigma_r \equiv \srp \text{ on~$\partial_\int \calD$}\}} 
    	\quad \text{ and }\quad 
    	\nu_\calD^{\srp\srm}(\sigma_r) =\frac1{Z_\calD^{\srp\srm}}\, 2^{k_{\partial\calD}(\theta(\sigma_r))}
    	\ind_{\{\sigma_r \equiv \srp \text{ on~$\partial_\int \calD$}\}},
	\end{align*}
	where~$k_{\partial\calD}(\theta(\sigma_r))$ is the number of connected components of~$\theta(\sigma_r)$, where all components intersecting~$\partial_\int\calD$ are counted as a single one. 
\end{rem}

We are in a position to prove Theorem~\ref{thm:FKG}.

\begin{proof}[Theorem~\ref{thm:FKG}]
	By~\cite[Thm.~4.11]{Gri10}, it is enough to show the FKG lattice condition, which states that, for any two spin configurations~$\sigma$ and~$\tilde\sigma$, 
	\begin{align}\label{eq:FKG1}
		\nu_\calD(\sigma \vee \tilde \sigma) \nu_\calD(\sigma \wedge \tilde \sigma) 
		\geq\nu_\calD(\sigma) \nu_\calD(\tilde \sigma),
	\end{align}
	where~$\sigma \vee \tilde \sigma, \sigma \wedge \tilde\sigma\in \{\rp, \rm\}^\bbT$ are defined by~$\sigma \vee \tilde\sigma(u) = \max (\sigma(u),\tilde\sigma(u))$ and~$\sigma \wedge \tilde\sigma(u) = \min (\sigma(u),\tilde\sigma(u))$ for every~$u\in\bbT$.
	Moreover, by \cite[Thm.~(2.22)]{Gri06}, it is enough to show~\eqref{eq:FKG1} for any two configurations which differ for exactly two faces. 	That is, that for any~$\sigma\in\{\rp, \rm\}^{F(\calD)}$ and~$u, v\in F(\calD)$ two distinct faces, 
	\begin{align}\label{eq:FKG-u-v}
		\nu_\calD(\sigma^{\srp\srp}) \cdot \nu_\calD(\sigma^{\srm\srm})\ge \nu_\calD(\sigma^{\srp\srm}) \cdot \nu_\calD(\sigma^{\srm\srp}),
	\end{align}
	where~$\sigma^{ab}$ is the configuration coinciding with~$\sigma$ except (possibly) at~$u$ and~$v$, 
	and such that~$\sigma^{ab}(u)= a$ and~$\sigma^{ab}(v)= b$.
	By Proposition~\ref{prop:marginal-ditribution}, the ratio of the LHS and RHS of~\eqref{eq:FKG-u-v} is written
	\begin{align}\label{eq:FKG2}
		\frac{\nu_\calD(\sigma^{\srp\srp}) \nu_\calD(\sigma^{\srm\srm})}
		{\nu_\calD(\sigma^{\srp\srm}) \nu_\calD(\sigma^{\srm\srp})}
		= 2^{k(\theta(\sigma^{\srp\srp})) + k(\theta(\sigma^{\srm\srm})) -k(\theta(\sigma^{\srp\srm})) - k(\theta(\sigma^{\srm\srp})) }.
	\end{align}
	Our goal is thus to show that
	\begin{equation}\label{eq:fkg-k}
		k(\theta(\sigma^{\srp\srp})) + k(\theta(\sigma^{\srm\srm})) -k(\theta(\sigma^{\srp\srm})) - k(\theta(\sigma^{\srm\srp})) \geq 0.
	\end{equation}
	
	
	First we will treat the simple case where~$\sigma$ is such that~$u \nxlra{\srp} v$ in~$\sigma^{\srp\srp}$ and~$u \nxlra{\srm} v$ in~$\sigma^{\srm\srm}$. Then~$u$ and~$v$ are not adjacent and there exist two paths or circuits, one of~$\rm$ the other of~$\rp$, that separate~$u$ from~$v$ in~$\calD$. Hence, there exists a path or loop~$\gamma$ in~$\omega(\sigma)$ that separates~$u$ from~$v$ and does not contain any edges of the faces~$u$ or~$v$. For any choice of~$a,b\in\{\rp,\rm\}$, edges in~$\bbT$ that cross~$\gamma$ belong to~$\theta(\sigma^{ab})$, thus forming a path or a circuit of edges in~$\theta(\sigma^{ab})$ that separates~$u$ from~$v$. The effect on~$k(\theta(\sigma))$ of switching the spin at~$v$ from~$\rp$ to~$\rm$ is then independent of the value of the spin at~$u$: 
	\begin{align*}
		k(\theta(\sigma^{\srp\srp})) -k(\theta(\sigma^{\srp\srm})) 
		= k(\theta(\sigma^{\srm\srp})) - k(\theta(\sigma^{\srm\srm})).			
	\end{align*}
	As a consequence, the LHS of~\eqref{eq:fkg-k} is zero. 
	\smallskip
	
	We move on to the case where~$u$ and~$v$ are connected by a path of~$\rp$ or by a path of~$\rm$. 
	Before diving into the core of the proof, we need to eliminate a degenerate case: 
	when~$u$ and~$v$ are neighbouring faces and no face of~$\calD$ is adjacent to both~$u$ and~$v$. 
	Then~$\calD$ may be split into two domains~$\calD_u$ and~$\calD_v$ containing 
	all faces connected to~$u$ in~$\calD \setminus \{v\}$ and those connected to~$v$ in~$\calD\setminus \{u\}$, respectively. 
	It is then immediate to see that the number of connected components of~$\theta(\sigma^{\srp\srp})$ 
	intersecting~$\calD_u$ is the same as that for~$\theta(\sigma^{\srp\srm})$.
	The same statement applies to~$\theta(\sigma^{\srm\srp})$ and~$\theta(\sigma^{\srm\srm})$.
	A similar statement may be formulated for~$\calD_v$, by pairing~$\theta(\sigma^{\srp\srp})$ with~$\theta(\sigma^{\srm\srp})$
	and~$\theta(\sigma^{\srp\srm})$ with~$\theta(\sigma^{\srm\srm})$. 
	Finally, in~$\theta(\sigma^{\srp\srm})$ and~$\theta(\sigma^{\srm\srp})$, faces~$u$ and~$v$ are in the same connected component, 
	while in~$\theta(\sigma^{\srp\srp})$ and~$\theta(\sigma^{\srm\srm})$ they are in different components. 
	Thus, we find  
	\begin{align*}
		k(\theta(\sigma^{\srp\srp})) + k(\theta(\sigma^{\srm\srm})) -k(\theta(\sigma^{\srp\srm})) - k(\theta(\sigma^{\srm\srp})) =2\geq 0.
	\end{align*}
	
	Henceforth we may assume that, if~$u$ and~$v$ are neighbours, then there exists at least one face of~$\calD$ adjacent to both~$u$ and~$v$. 
	Moreover, we will suppose that~$u$ and~$v$ are connected by a path of~$\rp$ in~$\sigma^{\srp\srp}$ or by a path of~$\rm$ in~$\sigma^{\srm\srm}$.
	By symmetry, we may limit our study to the case where~$u$ is connected to~$v$ in~$\sigma^{\srp\srp}$ by a~$\rp$-path;
	when~$u$ and~$v$ are neighbours, we may choose the path to contains at least one vertex other than~$u$ and~$v$. 
	
	Denote by~$P$ the~$\rp$-cluster of~$u$ (and implicitly of~$v$ as well) in~$\sigma^{\srp\srp}$; 
	denote by~$M$ the union of all~$\rm$-clusters in~$\sigma^{\srp\srp}$ that are adjacent to~$u$ or~$v$. 
	Both~$P$ and~$M$ are fixed sets of faces of~$\calD$. 
	Then all the connected components of~$\theta(\sigma^{\srp\srp})$,~$\theta(\sigma^{\srp\srm})$,~$\theta(\sigma^{\srm\srp})$, and~$\theta(\sigma^{\srm\srm})$ that do not intersect~$P\cup M$ are the same in these four configurations,
	and thus cancel out in~\eqref{eq:fkg-k}. 
	It remains to study the contribution of connected components of~$\theta(.)$ that do intersect~$P\cup M$.
	
	 For a spanning subgraph~$\Theta$ of~$\calD^*$, define~$k_P(\Theta)$ to be the number of connected components of~$\Theta$ that intersect~$P$, and ~$k_M(\Theta)$ as number of connected components that intersect~$M$ and do not intersect~$P$. 
	Clearly,~$k_P(\Theta)+ k_M(\Theta)$ is equal to the number of connected components in~$\Theta$ that intersect~$P\cup M$. 
	Thus, is suffices to prove the following two inequalities:
	\begin{align}\label{eq:fkg-p}
		k_P(\theta(\sigma^{\srp\srp})) + k_P(\theta(\sigma^{\srm\srm})) 
		- k_P(\theta(\sigma^{\srp\srm})) - k_P(\theta(\sigma^{\srm\srp})) &\geq 0, \\
		\label{eq:fkg-m}
		k_M(\theta(\sigma^{\srp\srp})) + k_M(\theta(\sigma^{\srm\srm})) 
		- k_M(\theta(\sigma^{\srp\srm})) - k_M(\theta(\sigma^{\srm\srp})) &\geq 0.
	\end{align}
	
We start by proving the easier inequality~\eqref{eq:fkg-m}. 
Four types of components contribute to~$k_M(\theta(\sigma^{\srm\srm}))$:
those who contain faces adjacent to both~$u$ and~$v$, those who contain faces adjacent to~$u$ but not~$v$, 
those who contain faces adjacent to~$v$ but not~$u$, and those containing no faces adjacent to~$u$ or~$v$. 
Write~$K_{\{u,v\}}$,~$K_{\{u\}}$,~$K_{\{v\}}$ and~$K_{\emptyset}$ for the number of components in each category above. 
By the definition of~$k_M$ and the fact that~$u,v \in P$, any connected component contributing to~$k_M(\theta(\sigma^{\srm\srm}))$ 
is such that all its faces that are adjacent to~$u$ or~$v$ have spin~$\rm$ in~$\sigma^{\srm\srm}$. 
When turning the spin of~$u$ from~$\rm$ to~$\rp$, all components of the type~$K_{\{u,v\}}$,~$K_{\{u\}}$ become connected to~$u$, and thus cease to contribute to~$k_M$. 
The same holds for~$v$, and we find:
\begin{align*}
	k_M(\theta(\sigma^{\srp\srm})) = K_{\{v\}} + K_{\emptyset}, \qquad	k_M(\theta(\sigma^{\srm\srp})) = K_{\{u\}} + K_{\emptyset}, \quad \text{ and } \quad 
	k_M(\theta(\sigma^{\srp\srp})) = K_{\emptyset}.
\end{align*}
Using that~$k_M(\theta(\sigma^{\srm\srm})) = K_{\{u,v\}} + K_{\{u\}} + K_{\{v\}} + K_{\emptyset}$, we find that the LHS of~\eqref{eq:fkg-m} is equal to~$K_{\{u,v\}}$, hence is non-negative. 
\smallskip

Let us now prove~\eqref{eq:fkg-p}.  
Denote by~$E_u,E_v\subset E(\calD^*)$ the sets of all edges linking~$u$ (resp.~$v$) to adjacent vertices in~$V(\calD^*)$.
The next claim constitutes the core of the proof and, as we will see below, implies readily~\eqref{eq:fkg-p}.
	\begin{claim}\label{cl:number-of-clusters}
		The following equalities hold:
		\begin{align}
			k_P(\theta(\sigma^{\srm\srp})) &= k_P(\theta(\sigma^{\srp\srp})\cup E_u),\label{eq:fkg-claim-u} \\
			k_P(\theta(\sigma^{\srp\srm})) &= k_P(\theta(\sigma^{\srp\srp})\cup E_v),\label{eq:fkg-claim-v} \\
			k_P(\theta(\sigma^{\srm\srm})) &= k_P(\theta(\sigma^{\srp\srp})\cup E_u\cup E_v).\label{eq:fkg-claim-uv} 
		\end{align}
	\end{claim}
	
	\begin{proof}
		We start by showing~\eqref{eq:fkg-claim-u}. Note that
		\[
			\theta(\sigma^{\srm\srp}) \subset \theta(\sigma^{\srp\srp})\cup E_u \quad\quad \text{and} 
			\quad\quad [\theta(\sigma^{\srp\srp})\cup E_u] \setminus \theta(\sigma^{\srm\srp}) 
			 = \{uw \in E(\calD^*)\, \colon \, \sigma^{{\srm\srp}}(w) = \rm \}.
		\]
		Thus, it remains to show that for any face~$w\sim u$, 
		\begin{align}\label{eq:fkg-claim-u2}
			\big(\sigma^{\srm\srp}(w) = \rm \text{ and } w \xlra{\theta(\sigma^{\srm\srp})} P\big) 
			\Rightarrow w \xlra{\theta(\sigma^{\srm\srp})} u.
		\end{align} 

		Figure~\ref{fig:k_P}, left diagram, helps illustrate the construction below.
		Consider a face~$w$ neighbouring~$u$, such that~$\sigma^{\srm\srp}(w) = \rm$ and~$w \xlra{\theta(\sigma^{\srm\srp})} P$.
		Let~$\gamma = (\gamma_0,\dots, \gamma_n)$ be a simple path of~$\theta(\sigma^{\srm\srp})$ with~$\gamma_0 = w$,~$\gamma_n \in P$ 
		and such that~$\gamma_0,\dots, \gamma_{n-1} \notin P$. 
		By our assumption~$\sigma^{{\srm\srp}}(w) = \rm$, we have~$w \notin P$, so~$n\geq1$. 
		Continue~$\gamma$ by a face-path~$\gamma_{n},\gamma_{n+1}, \dots, \gamma_{m}$ contained in~$P$ and with~$\gamma_m =u$. 
		(Note that we do not require that the path~$\gamma_n \dots, \gamma_m$ be contained in~$\theta(\sigma^{\srm\srp})$.)
		Then it is necessary that~$\sigma^{\srm\srp}(\gamma_{m-1}) = \rp$, hence~$\gamma_{m-1} \in P$, which is to say~$n < m$.
		Finally set~$\gamma_{m+1} = w$. 	
		
		\begin{figure}
		\begin{center}
		\includegraphics[height=0.27 \textwidth]{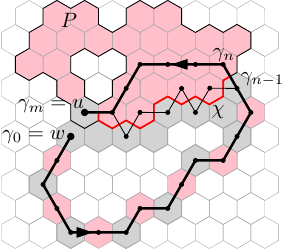}\quad
		\includegraphics[height=0.27 \textwidth]{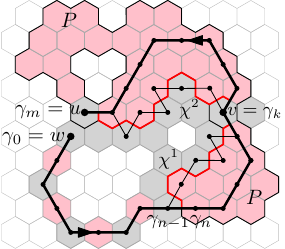}\quad
		\includegraphics[height=0.27 \textwidth]{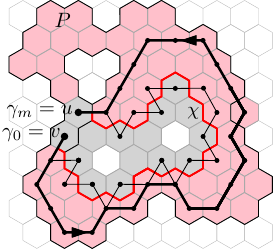}
		\caption{The constructions used in the proofs of~\eqref{eq:fkg-claim-u} (left) and~\eqref{eq:fkg-claim-uv} (centre and right).
		Spins~$\frp$ are pink and~$\frm$ are gray; only spins of interest are depicted.
		The path~$\gamma$ (black bold) uses faces of alternating spins until it enters~$P$, 
		then it continues on~$P$, whose faces (except for~$v$ in the central and right diagrams) are of spin~$\frp$.  
		The paths~$\chi$,~$\chi^1$ and~$\chi^2$ (in red) are part of the boundary of~$P$ and separate faces of distinct spins. 
		Their dual edges contain paths linking~$u$ to~$\gamma_n$,~$u$ to~$v$ and~$v$ to~$\gamma_n$, respectively. }
		\label{fig:k_P}
		\end{center}
		\end{figure}	 
		
		Then~$\gamma$ is a non-trivial simple cycle on~$\bbT$. 
		Since the domain~$\calD$ is simply connected,~$\gamma$ delimits a simply connected domain which we denote by~$D_\gamma$. 
		The boundary of~$P \setminus \{u\}$ intersects~$\gamma$ at two places: the midpoint of the edge~$\gamma_{m-1}\gamma_m$ 
		and the midpoint of the edge~$\gamma_{n-1}\gamma_n$.
		Thus the boundary of~$P\setminus\{u\}$ contains a path~$\chi$ that is contained in~$D_\gamma$ 
		and that connects these two midpoints of edges.
		
		Finally notice that, for any two adjacent faces~$a,b$ with~$a \in P \setminus \{u\}$ and~$b \notin P\setminus \{u\}$, 
		we have~$\sigma^{\srm\srp}(a) = \rp$ and~$\sigma^{\srm\srp}(b) = \rm$, hence~$ab \in \theta(\sigma^{\srm\srp})$.
		Applying this to faces on either side of~$\chi$, 
		we find that all edges of~$\bbT$ crossing~$\chi$ are contained in~$\theta(\sigma^{\srm\srp})$.
		In particular, we deduce that~$u$ is connected in~$\theta(\sigma^{\srm\srp})$ to~$\gamma_n$, hence also to~$w$.
		This completes the proof of~\eqref{eq:fkg-claim-u}.
		The same argument proves~\eqref{eq:fkg-claim-v}. 
		
		We turn to the proof of~\eqref{eq:fkg-claim-uv}.
		We will prove this in two steps:
		\begin{align}\label{eq:fkg-claim-uv2}
			k_P(\theta(\sigma^{\srm\srm})) 
			=k_P(\theta(\sigma^{\srp\srm})\cup E_u)   
			= k_P(\theta(\sigma^{\srp\srp})\cup E_u\cup E_v)
		\end{align}
		The second equality above is implied by~\eqref{eq:fkg-claim-v}. 
		Indeed, we have proved that no edge of ~$\{vw \in E(\calD^*)\, \colon \, \sigma^{{\srm\srp}}(w) = \rm \}$
		may connect two distinct clusters contributing to~$k_P(\theta(\sigma^{\srp\srm}))$. That is also true for clusters contributing to 
		$k_P(\theta(\sigma^{\srp\srm}) \cup E_u)$, since the latter configuration dominates the former. 
		 
		The first equality of~\eqref{eq:fkg-claim-uv2} is similar to~\eqref{eq:fkg-claim-u}, 
		with the only difference that it applies to~$\sigma^{\srm\srm}$ rather than~$\sigma^{\srm\srp}$. 
		This apparent detail complicates the proof slightly as~$u$ is not necessarily connected to all points of~$P$ 
		by paths of~$\rp$ in~$\sigma^{\srp\srm}$.
		The middle and right diagram of Figure~\ref{fig:k_P} helps illustrate the argument below.
		
		As for \eqref{eq:fkg-claim-u}, the proof goes through the equivalent of \eqref{eq:fkg-claim-u2}. 
		Fix a face~$w$ neighbouring~$u$ with~$\sigma^{\srm\srm}(w) = \rm$ 
		and which belongs to a connected component of~$\theta(\sigma^{\srm\srm})$ that intersects~$P$.
		Our goal is to prove that~$w$ is connected to~$u$ in~$\theta(\sigma^{\srm\srm})$
			
		In a first instance let us suppose that~$w \neq v$. 
		Then, as in the proof of~\eqref{eq:fkg-claim-u},
		we may produce a path~$w = \gamma_0,\dots, \gamma_n,\dots, \gamma_m = u$ 
		such that~$\gamma_0,\dots, \gamma_{n-1} \notin P$,~$\gamma_n,\dots, \gamma_m \in P$ 
		and~$\gamma_{0},\dots,\gamma_{n}$ uses only edges of~$\theta(\sigma^{\srm\srm})$. 
		If such a path may be constructed to not include~$v$, then we choose~$\gamma$ such, 
		and the same reasoning as in~\eqref{eq:fkg-claim-u} (applied with~$P\setminus \{v\}$ instead of~$P$) allows us to conclude that 
		$ u \xlra{\theta(\sigma^{\srm\srm})} \gamma_{n}\xlra{\theta(\sigma^{\srm\srm}) } w$. 
		
		Suppose now that no path~$\gamma$ with the properties above and which avoids~$v$ exists. 
		Then pick~$\gamma$ to visit~$v$ at some index~$k \geq n$ and with~$k < m -1$ (see Figure~\ref{fig:k_P}, center).
		We have~$k \geq n$ since~$v \in P$; 
		we may pick~$k < m-1$ since, even when~$u$ and~$v$ are adjacent,~$u$ is connected to~$v$ by a non-trivial path of~$\rp$, 
		and we include this path in~$\gamma$. 
		It is also true that~$k > n$, since~$\sigma^{\srm\srm}(\gamma_{k-1}) = \rp$ necessarily. 
		Let~$D_\gamma$ be the domain delimited by~$\gamma$.

		Consider the boundary of~$P \setminus \{u,v\}$ inside the domain~$D_\gamma$;
		it intersects the boundary of~$D_\gamma$ at four points: the midpoint of the edges
		$\gamma_{n-1}\gamma_n$,~$\gamma_{k-1}\gamma_k$,~$\gamma_{k}\gamma_{k+1}$ and~$\gamma_{m-1}\gamma_m$. 
		Since no path~$\gamma$ avoiding~$v$ exists, the boundary of~$P \setminus \{u,v\}$ contains two non-empty segments~$\chi^1$ and~$\chi^2$
		which connect~$\gamma_{n-1}\gamma_n$ to ~$\gamma_{k-1}\gamma_k$ and~$\gamma_{k}\gamma_{k+1}$ to~$\gamma_{m-1}\gamma_m$, respectively.
		By the choice of~$\chi^1$ and~$\chi^2$ as parts of the boundary of~$P \setminus \{u,v\}$, 
		all edges of~$\bbT$ that intersect~$\chi^1$ and~$\chi^2$ are present in~$\theta(\sigma^{\srm\srm})$.
		In particular, we find~$ u \xlra{\theta(\sigma^{\srm\srm})} v$ and~$v\xlra{\theta(\sigma^{\srm\srm}) } w$, 
		which implies that~$u$ is connected to~$w$ in~$\theta(\sigma^{\srm\srm})$.
		\smallskip
		
		Finally let us study the case when~$w = v$ and hence~$u$ and~$v$ are adjacent  (see Figure~\ref{fig:k_P}, right).
		Then, due to our assumption that~$u$ and~$v$ are connected by a non-trivial path of~$\rp$ in~$\sigma^{\srp\srp}$, 
		we may choose a face-path~$\gamma = \gamma_0,\dots,\gamma_m$ with~$m\geq 2$,~$\gamma_0 = v$,~$\gamma_m = u$ 
		and~$\sigma^{\srm\srm}(\gamma_k) = \rp$ for all~$1 \leq k < m$. 
		The cycle~$\gamma \cup \{uv\}$ delimits a simply connected domain~$D_\gamma$. 
		By considering the interface between~$P$ and the~$\rm$ cluster of~$u$ in~$\sigma^{\srm\srm}$, 
		we deduce the existence of an edge-path~$\chi$ on~$E(D_\gamma)$ with~$\rp$ on one side and~$\rm$ on the other, 
		that starts on an edge adjacent to~$v$ and ends on one adjacent to~$u$. 
		This implies that~$u \xlra{\theta(\sigma^{\srm\srm}) }v$, and the proof is complete. 
	\end{proof}
	
	Using Claim~\ref{cl:number-of-clusters},~\eqref{eq:fkg-p} becomes
	\begin{equation}\label{eq:fkg-ineq-fk-ising}
		k_P(\theta(\sigma^{\srp\srp})) - k_P(\theta(\sigma^{\srp\srp})\cup E_u) 
		\geq  k_P(\theta(\sigma^{\srp\srp})\cup E_v)  - k_P(\theta(\sigma^{\srp\srp})\cup E_u\cup E_v).
	\end{equation}
	The LHS above is the number of distinct connected components in~$\theta(\sigma^{\srp\srp})$ 
	that contain at least one endpoint of an edge of~$E_u$ minus one. 
	The RHS is the same number for~$\theta(\sigma^{\srp\srp})\cup E_v$ instead of~$\theta(\sigma^{\srp\srp})$. 
	Clearly, the former is greater or equal than the latter, and the proof of~\eqref{eq:fkg-p} is finished.
\end{proof}
	
Below we formulate several corollaries about the FKG inequality under various boundary conditions that we are going to use in the proofs.
	
\begin{cor} \label{cor:fkg}
	The FKG inequality~\eqref{eq:fkg-def} holds also in the following cases:
	\begin{enumerate}[label=(\roman*)]
			\item for the red-spin marginal of~$\mu_\calD$, 
			when the red spins are conditioned to take given values on a set of faces of~$\calD$
			and the blue spins are conditioned to be~$\bp$ on a connected set of faces of~$\calD$.
			More precisely, for~$\sigma_r$ chosen according to
			~$\mu_{\calD}(.|\, \sigma_r = \sigma_0 \text{ on~$\calA$} \text{ and } \sigma_b\equiv \bp  \text{ on~$\calB$})$, 
			where~$\calA$ is any set of faces of~$\calD$,~$\sigma_0$ is any red spin configuration on~$\calA$, 
			and~$\calB$  is a connected set of faces of~$\calD$;
			\item for the measures~$\nu^{\srm\srm}_\calD$,~$\nu^{\srm\srp}_\calD$,~$\nu^{\srp\srm}_\calD$ and~$\nu^{\srp\srp}_\calD$.
	\end{enumerate}
\end{cor}
	
\begin{proof}
	\noindent\textit{(i)}
		First let us show the FKG inequality when~$\calA$ is empty. 
		As described in Remark~\ref{rem:blue+}, conditioning the blue spins to be~$\bp$ on~$\calB$ boils down to 
		counting all connected components of~$\theta(\sigma_r)$ that intersect $\calB$ as a single one. 
		When~$\calB$ is connected, this may be achieved by adding to~$\theta(\sigma_r)$ all edges linking pairs of neighbouring vertices in~$\calB$. 
		The proofs of~\eqref{eq:fkg-claim-u},~\eqref{eq:fkg-claim-v} and~\eqref{eq:fkg-claim-uv} adapt directly to this situation.
		Indeed, as already discussed in the proof above, 
		adding edges to~$\theta(\sigma_r)$ only helps in proving~\eqref{eq:fkg-claim-u},~\eqref{eq:fkg-claim-v} and~\eqref{eq:fkg-claim-uv}.
		The rest of the proof of Theorem~\ref{thm:FKG} applies directly.

		Next assume that~$\calA$ is non-empty and~$\sigma_0$ is given. 
		The FKG lattice condition for~$\mu_{\calD}(.|\, \sigma_r = \sigma_0 \text{ on~$\calA$}\text{ and } \sigma_b\equiv \bp  \text{ on~$\calB$})$ 
		is a subset of the inequalities that constitute the FKG lattice condition for~$\mu_{\calD}(.|\, \sigma_b\equiv \bp  \text{ on~$\calB$})$.
		Since the latter were proved to hold, so do the former. 
	\smallskip 
		
	\noindent\textit{(ii)} Let~$\calD'$ be a domain containing~$F(\calD) \cup \partial_\out \calD$. 
	Then~$\nu_{\calD'}$ satisfies the FKG inequality. 
	By point \textit{(i)} and the Spatial Markov property (Corollary~\ref{cor:DMP}), the FKG inequality also applies to 
	$\nu^{\srm\srm}_\calD$,~$\nu^{\srm\srp}_\calD$,~$\nu^{\srp\srm}_\calD$ and~$\nu^{\srp\srp}_\calD$.
%
\end{proof}

\begin{rem}\label{rem:FKG}
	\begin{itemize}
	\item[(i)] The FKG inequality does not apply to~$\sigma_r$ 
	under~$\mu_{\calD}(.|\, \sigma_b \equiv \bp \text{ on~$\calB$})$
	when~$\calB$ is not connected. 
	A counter-example is provided by a domain formed of six faces in a line, 
	with~$\calB$ being formed of the first and last face,~$u$ and~$v$ being the second and fifth face, respectively, 
	and~$\sigma$ being the red spin configuration formed of alternating~$\rp$ and~$\rm$ spins.
		
	Nor does the FKG inequality apply to the red spin marginal of 
	$\mu_{\calD}(.|\, \sigma_b \equiv \bp \text{ on~$\calB_+$ and } \sigma_b \equiv \bm \text{ on~$\calB_-$})$, 
	where~$\calB_+$ and~$\calB_-$ are disjoint sets of faces of~$\calD$.  
	
	\item[(ii)] The proof of the FKG inequality only uses limited features of the hexagonal lattice. 
	Indeed, it adapts to any planar trivalent graph whose set of faces forms a simply connected domain. 
	It is however worth mentioning that the condition of simply connectedness is essential.
	Indeed, counter-examples may be given for sets of faces of~$\bbH$ which are not simply connected: 
	the counter-example of point (i) above may easily be adapted. 

	\item[(iii)] A similar instance of the FKG inequality extends to the loop~$O(n)$ model with~$n\geq 2$ and~$x\leq 1/\sqrt{n-1}$, 
	when the red spin configuration is obtained by colouring loops in red with probability~$1/n$ and in blue otherwise, independently. 
	The only difference in the proof is that the term~$2^{k(\theta(\sigma))}$ in~\eqref{eq:red_spin_law_free} 
	should be replaced by the partition function of the Ising model on the graph obtained by collapsing 
	each cluster of~$\theta(\sigma_r)$ into a single vertex. 
	The FKG property of the FK-Ising representation then leads to the analogue of~\eqref{eq:fkg-ineq-fk-ising}. 
	We do not give further details of this generalisation as it is irrelevant here;
	the reader is referred to~\cite{GlaPel18}, where similar ideas are used 
	to prove a FKG statement for the spin representation of a six-vertex model.
	\end{itemize}
\end{rem}

\subsection{Comparison between boundary conditions}\label{subsec:comparison-b-c}

Above he have introduced a number of boundary conditions for the positively associated measure~$\nu_\calD$. 
As for the random cluster model or other positively associated models, the boundary conditions may have an increasing or decreasing effect on the measure. 

For two measures~$\nu_1,\nu_2$ on~$\{\rm,\rp\}^{F}$ (where~$F$ is some non-empty set), say that~$\nu_1$ stochastically dominates~$\nu_2$, written~$\nu_1 \geq_{\text{st}} \nu_2$, 
if for any increasing event~$A \subset \{\rm,\rp\}^{F}$,~$\nu_1(A) \geq \nu_2(A)$. 

\begin{cor}[Comparison between boundary conditions]\label{cor:monotonicity_bc}\hfill
	\begin{itemize}
		\item[(i)] Let~$\calD$ be a domain and let~$\calA \subset F(\calD)$. Let~$\sigma^1 \leq \sigma^2$ be two (red) spin configurations on~$\calA$. Then~$\nu_\calD( . \,|\, \sigma_r = \sigma^1 \text{ on~$\calA$})\leq_{\text{st}} \nu_\calD( . \,|\, \sigma_r = \sigma^2 \text{ on~$\calA$})$.
		\item[(ii)] For any domain~$\calD$ the following comparison inequalities hold:
		\begin{equation}\label{eq:comparison}
			\nu^{\srm\srm}_\calD \leq_{\text{st}}\nu^{\srm\srp}_\calD \leq_{\text{st}}\nu^{\srp\srm}_\calD \leq_{\text{st}}\nu^{\srp\srp}_\calD.
		\end{equation}
	\end{itemize}	
\end{cor}

\begin{proof}
		\textit{(i)} Write~$\calA = \calA_= \sqcup \calA_\neq$, where~$\calA_=$ is the set of faces where~$\sigma_1$ and~$\sigma_2$ agree 
		and~$\calA_{\neq}$ that where they disagree. 
		By the ordering~$\sigma_1 \leq \sigma_2$, we deduce that~$\sigma_1$ is constantly~$\rm$ on~$\calA_\neq$ 
		while~$\sigma_2$ is constantly~$\rp$ on this set. 
		Due to the positive association of~$\nu_{\calD}( . \,|\, \sigma_r = \sigma_1 \text{ on~$\calA_=$})$ shown in Corollary~\ref{cor:fkg},
		\begin{align*}
			\nu_{\calD}( . \,|\, \sigma_r = \sigma_1 \text{ on~$\calA$})
			&=\nu_{\calD}( . \,|\, \sigma_r = \sigma_1 \text{ on~$\calA_=$ and }\sigma_r \equiv \rm \text{ on~$\calA_\neq$})\\
			&\leq_{\text{st}}
			\nu_{\calD}( . \,|\, \sigma_r = \sigma_1 \text{ on~$\calA_=$ and }\sigma_r \equiv \rp \text{ on~$\calA_\neq$})\\
			&=\nu_{\calD}( . \,|\, \sigma_r = \sigma_2 \text{ on~$\calA$}).
		\end{align*}
		
		\noindent
		\textit{(ii)} Let us begin with the first and last inequalities of~\eqref{eq:comparison}.
		Let~$\calD'$ be a domain containing~$F(\calD) \cup \partial_\out \calD$. By point \textit{(i)} above, 
		\begin{align*}
    		&\nu_{\calD'} \big[. \,\big|\, \sigma_r \equiv \rp \text{ on~$\partial_{\int} \calD$ and } 
    			\sigma_r \equiv \rm \text{ on~$F(\calD')\setminus F(\calD)$}\big]\\
    		&\qquad \leq_{\text{st}}	
    		\nu_{\calD'} \big[. \,\big|\, \sigma_r \equiv \rp \text{ on~$\partial_\int \calD \cup (F(\calD')\setminus F(\calD)$})\big].
 		\end{align*}
		The Spatial Markov property (Corollary~\ref{cor:DMP}) translates the above to 
		$\nu^{\srp\srm}_\calD \leq_{\text{st}}\nu^{\srp\srp}_\calD$. 
		The first inequality of~\eqref{eq:comparison} is proved in the same way. 
		
		We move on to the middle inequality of~\eqref{eq:comparison}.
		Considering~\eqref{eq:domain-markov-rm} and the symmetry of blue spins, this inequality may be written as
		\[
			\mu_\calD (\cdot \, | \, \sigma_r = \rm \text{ on } \partial_\int\calD, \,  \sigma_b = \bp \text{ on } \partial_\int\calD) 
			\leq_{\text{st}} 
			\mu_\calD (\cdot \, | \, \sigma_r = \rp \text{ on } \partial_\int\calD, \,  \sigma_b = \bp \text{ on } \partial_\int\calD),
		\]
		where the stochastic ordering refers only to the red-spin marginal. 
		Clearly, the set~$\partial_\int\calD$ is connected in~$\bbH^*$, thus the inequality follows from Corollary~\ref{cor:fkg} \textit{(i)}.		
\end{proof}

Let~$\calD$ be a domain with vertices~$a,b,c,d$ on its boundary~$\partial_E \calD$, arranged in counter-clockwise order, and
such that the edges incident to~$a$,~$b$,~$c$ and~$d$ all belong to~$\partial_E \calD$ or to~$E(\calD)$. 
Call~$(ab)$ the segment of~$\partial_E \calD$ between~$a$ and~$b$, when going around~$\calD$ in the counter-clockwise direction.
Define~$(bc)$,~$(cd)$ and~$(da)$ similarly. 

Let~$\mu_\calD^{ a\srp b \srm c \srp d \srm a}$ be the uniform measure on pairs of coherent red and blue spin configurations on~$\calD$ 
with the property that~$\sigma_r$ is equal to~$\rp$ on all faces of~$\partial_\int\calD$ adjacent to~$(ab)$ or~$(cd)$ 
and~$\rm$ on all other faces of~$\partial_\int\calD$. 
The condition above also imposes that the blue spins of the two faces of~$\partial_\int\calD$ that are adjacent to~$a$ are equal, 
and the same for the pairs of faces adjacent to~$b$,~$c$ and~$d$. 
Other than this, there is no restriction for the blue spins on~$\partial_\int\calD$.
The marginal on red spins of the above is denoted by~$\nu_\calD^{a\srp b \srm c \srp d \srm a}$.

Fix now a larger domain~$\calD'$ with~$\calD \subset \calD'$ (possibly with~$\partial_E\calD\cap \partial_E\calD' \neq \emptyset$). 
We say that a configuration~$\tau_r$ on~$\calD' \setminus \Int(\calD)$ 
imposes boundary conditions~${a\rp b \rm c \rp d \rm a}$ on~$\calD$
if all the faces adjacent to~$(ab) \cup (cd)$ but not to~$(bc) \cup (da)$ have spin~$\rp$ in~$\tau_r$
and all those adjacent to~$(bc) \cup (da)$ but not to~$(ab) \cup (cd)$ have spin~$\rm$.
The faces of~$\partial_\int \calD \cup \partial_\out \calD$ that are adjacent to both~$(bc) \cup (da)$ and~$(ab) \cup (cd)$
may have spins~$\rp$ or~$\rm$ in~$\tau_r$%
\footnote{Due to the choice of~$a,b,c$ and~$d$, all such faces are on~$\partial_\out\calD$; their spins have no influence on the measure induced in~$\calD$.}
(see Figure~\ref{fig:pmpm}).

\begin{figure}
\begin{center}
	\includegraphics[width=0.35\textwidth]{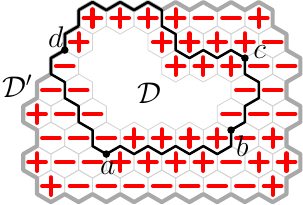}
	\caption{Two domains~$\calD \subset \calD'$ bounded by the black and grey contours, respectively. 
	The configuration in~$\calD' \setminus \Int(\calD)$ imposes boundary conditions~${a\frp b \frm c \frp d \frm a}$ on~$\calD$.}
\label{fig:pmpm}
\end{center}
\end{figure}

As already discussed in Remark~\ref{rem:4arcs},
the spatial Markov property does not apply to the boundary conditions~${a\rp b \rm c \rp d \rm a}$.
Indeed, the connectivity of the edges of~$\theta(\tau_r)$ that are adjacent to~$a$,~$b$,~$c$ and~$d$ influences the measure in~$\calD$ 
and is not determined by the spins on~$\partial_\int \calD \cup \partial_\out \calD$.
It may be that certain configuration~$\sigma_b$ are awarded positive probability in~$\mu_\calD^{a\srp b \srm c \srp d \srm a}$, 
but null probability in~$\mu_{\calD'}(.\,|\, \sigma_r = \tau_r \text{ on~$\calD' \setminus \Int(\calD)$})$.
However, if we limit ourselves to the red-spin marginal, the Radon-Nikodim derivative of the second measure with respect to the first may be shown to be uniformly bounded. 

\begin{lem}\label{lem:4-arcs}
	\textit{(i)} Let~$\calD$ be a domain and~$a,b,c,d$ be points on~$\partial_E\calD$ as above. 
	Let~$\calD'$ be a domain containing~$\calD$ and~$\tau_r$ be a red spin configuration on~$\calD'\setminus \Int(\calD)$, 
	that imposes boundary conditions~${a\rp b \rm c \rp d \rm a}$ on~$\calD$.
	Then, for any configuration~$\varsigma_r \in \{\rm,\rp\}^{F(\calD)}$, 
	\begin{align*}
		\tfrac{1}{2^3}\, \nu_\calD^{a\srp b \srm c \srp d \srm a}(\varsigma_r) \leq 
		\nu_{\calD'}[\sigma_r = \varsigma_r \text{ on~$\calD$} \, | \, \sigma_r = \tau_r \text{ on } F(\calD')\setminus \mathrm{Int}(\calD)]
		\leq 2^3 \,\nu_\calD^{a\srp b \srm c \srp d \srm a}(\varsigma_r).
	\end{align*}
	\noindent\textit{(ii)} Moreover, if~$a,b,c,d \in \partial_E \calD \cap \partial_E \calD'$ and 
	the arcs~$(bc)$ and~$(da)$ of ~$\partial_E \calD$ and~$\partial_E \calD'$ coincide, then 
	\begin{align*}
		\nu_\calD^{a\srp b \srm c \srp d \srm a} \geq_{\text{st}}{\nu_{\calD'}^{a\srp b \srm c \srp d \srm a}},
	\end{align*}
	by which we mean that the former stochastically dominates the restriction to~$\calD$ of the latter. 
	
	\noindent\textit{(iii)} Finally, if~$a',b',c',d' \in \partial_E\calD$ are another set of four points with the same properties 
	as~$a,b,c,d$ and such that~$(bc)\subset (b'c')$ and~$(da)\subset (d'a')$, then
	\begin{align*}
	\nu_\calD^{a\srp b \srm c \srp d \srm a} \geq_{\text{st}}\nu_{\calD}^{a'\srp b' \srm c' \srp d' \srm a'}.
	\end{align*}
\end{lem}

\begin{proof}
	\textbf{\textit{(i)}}
	Both measures in the statement are supported on configuration~$\varsigma_r \in \{\rm,\rp\}^{F(\calD)}$ that agree with~$\tau_r$ on~$\partial_\int\calD$.
	By Proposition~\ref{prop:marginal-ditribution}, for any such configuration~$\varsigma_r$, 		
	\begin{align}\label{eq:marginal-ditribution1}
    	&\nu_\calD^{a\srp b \srm c \srp d \srm a}(\varsigma_r) = \frac{1}{Z_\calD^{a\srp b \srm c \srp d \srm a}} 2^{k(\theta(\varsigma_r))} \text{ and }\\
    	&\nu_{\calD'}[\sigma_r = \varsigma_r \text{ on~$\calD$} \, | \, \sigma_r = \tau_r \text{ on } F(\calD')\setminus \mathrm{Int}(\calD)]
	=  \frac{1}{\tilde Z} 2^{k_\calD(\theta(\varsigma_r \cup \tau_r))},
	\end{align}
	where~$Z_\calD^{a\srp b \srm c \srp d \srm a}, \tilde Z$ are normalising constants and~$k_\calD(\theta(\varsigma_r \cup \tau_r))$ is the number of connected components of 
	$\theta(\varsigma_r \cup \tau_r)$ that intersect~$\calD$. 
	Indeed, the number of connected components that do not intersect~$\calD$ does not depend on~$\varsigma_r$, hence cancel out.
	
	Observe that~$\theta(\varsigma_r \cup \tau_r)$ contains more connections than~$\theta(\varsigma_r)$, hence fewer connected components. 
	However, in~$\theta(\varsigma_r)$, there are at most four distinct connected components that may be connected in ~$\theta(\varsigma_r \cup \tau_r)$. 
	Thus, 
	$$k(\theta(\varsigma_r)) - 3 \leq k_\calD(\theta(\varsigma_r \cup \tau_r)) \leq k(\theta(\varsigma_r)).$$
	By summing the above over all configuration~$\sigma_r$ in the support of ~$\nu_\calD^{a\srp b \srm c \srp d \srm a}$, 
	we find~$\frac{1}{8}Z_\calD^{a\srp b \srm c \srp d \srm a} \leq \tilde Z \leq Z_\calD^{a\srp b \srm c \srp d \srm a}$.
	Inserting the last two inequalities in \eqref{eq:marginal-ditribution1} provides the desired bound.

	\noindent\textbf{\textit{(ii)}}
	We have 
	$$\nu_\calD^{a\srp b \srm c \srp d \srm a} 
	= \nu_{\calD'}^{a\srp b \srm c \srp d \srm a}(.\,|\, \sigma_r \equiv \rp \text{ on~$\calD'\setminus \Int(\calD)$}).$$
	The FKG inequality implies the desired stochastic domination. 
	\medskip
	
	\noindent\textbf{\textit{(iii)}}
	Let~$A_+$ be the faces of~$\partial_\int\calD$ that are adjacent to~$(a'b')$ or~$(c'd')$ and
	$A_-$ be the faces of~$\partial_\int\calD$ that are adjacent to~$(bc)$ or~$(da)$.
	Set~$A_\neq = \partial_\int\calD \setminus (A_+\cup A_-)$.
	Then
	\begin{align*}
		&\nu_\calD^{a\srp b \srm c \srp d \srm a} = \nu_{\calD}(.\,|\, \sigma_r \equiv \rp \text{ on~$A_+ \cup A_\neq$ and }\sigma_r \equiv \rm \text{ on~$A_-$})	
		\qquad \text{ and }\\
		&\nu_\calD^{a'\srp b' \srm c' \srp d' \srm a'} = \nu_{\calD}(.\,|\, \sigma_r \equiv \rp \text{ on~$A_+$ and }\sigma_r \equiv \rm \text{ on~$A_- \cup A_\neq$}).
	\end{align*}
	Corollary \ref{cor:monotonicity_bc} \textit{(i)} implies that the first measure dominates the second. 
\end{proof}

\section{Infinite-volume measure: existence and uniqueness}\label{sec:infinite_volume}

In this section we construct an infinite-volume Gibbs measure for the loop~$O(2)$ model. 
As for the Ising, Potts or FK models, the infinite-volume limit will be created as a limit of finite-volume measures. 
The existence of the limit rests on the monotonicity of the measures~$\nu_\calD$ in their boundary conditions.  

\begin{thm}[Existence of limiting measure]\label{thm:limit_joint}
	For any increasing sequence of domains ~$(\calD_n)_{n \geq 0}$ with~$\bigcup_{n\geq0} \calD_n = \bbH$, 
	the sequence of measures~$\mu_{\calD_n}^{\srp\srp}$ converges to a measure~$\mu^{\srp\srp}_\bbH$ on~$\{\rm,\rp\}^{F(\bbH)}$.
	Moreover, the ~$\mu^{\srp\srp}_\bbH$ is invariant under translations and rotations by multiples of~$\pi/3$, 
	ergodic with respect to translations, 
	and has positively associated blue- and red-spin marginals. 
\end{thm}

The same argument may be used to construct measures~$\mu^{\srm\srm}_\bbH$,~$\mu^{\sbp\sbp}_\bbH$ and~$\mu^{\sbm\sbm}_\bbH$
as limits of finite-volume measures with the proper boundary conditions. 
Below we prove that these measures are all equal to a single measure~$\mu_\bbH$. 
We also show that the measures~$\mu_{\calD}^{\srp\srm}$ converge to~$\mu_\bbH$ as~$\calD$ increases to~$\bbH$. 

For the double-spin representation, the theorem below may be understood as a partial uniqueness theorem; 
see Remark \ref{rem:frozen-states} for more on why it is not a complete uniqueness theorem. 
For Lipschitz functions, the theorem below amounts to non-quantitative delocalisation: 
it proves that the value at~$0$ is not tight as the domain increases to~$\bbH$, 
but does not offer the speed at which its variance increases. 

For~$n \geq 1$, let~$\Circ_{\srp\srp}(n)$ be the event that there exists a simple closed path of edges of~$\bbH$
surrounding~$\La_n$ with the property that the red spin of all faces adjacent to any of its edges is~$\rp$. 
The events~$\Circ_{\srm\srm}(n)$,~$\Circ_{\sbp\sbp}(n)$ and~$\Circ_{\sbm\sbm}(n)$ are defined similarly. 

\begin{thm}[Uniqueness of infinite-volume measure / Delocalisation]\label{thm:uniqueness_nu}
	For any~$n \geq 1$, 
	\begin{align}\label{eq:uniqueness_nu}
		\mu^{\srp\srp}_\bbH(\Circ_{\srp\srp}(n)) = \mu^{\srp\srp}_\bbH(\Circ_{\srm\srm}(n)) = 
		\mu^{\srp\srp}_\bbH(\Circ_{\sbp\sbp}(n)) = \mu^{\srp\srp}_\bbH(\Circ_{\sbm\sbm}(n)) = 1.
	\end{align}
	In particular~$\mu^{\srp\srp}_\bbH = \mu^{\srm\srm}_\bbH = \mu^{\sbp\sbp}_\bbH = \mu^{\sbm\sbm}_\bbH$, and we will simply write~$\mu_\bbH$. 
	
	Also, for any sequence of finite domains~$\calD_n$ that increases to~$\bbH$, the measures 
	$\mu_{\calD_n}^{\srp\srm}$,~$\mu_{\calD_n}^{\srm\srp}$,~$\mu_{\calD_n}^{\sbp\sbm}$ and~$\mu_{\calD_n}^{\sbm\sbp}$
	all converge to~$\mu_\bbH$.
\end{thm}

\begin{rem}\label{rem:frozen-states}
	The theorem above states that any finite-volume measure with any red boundary condition converges to~$\mu_\bbH$. 
	However, we do not claim this for mixed red and blue boundary conditions. 
	
	It may be tempting to believe that, for any assignment of red and blue spins~$\xi_n^r,\xi_n^b$ on~$\partial_\int\calD_n$,
	the measure~$\mu_{\calD_n}$ conditioned to have spins~$\xi_n^r$ and~$\xi_n^b$ on~$\partial_\int\calD_n$ also converge to~$\mu_\bbH$.
	Unfortunately this is not the case: 
	counter examples may be created where the boundary conditions~$\xi_n^r,\xi_n^b$ force one single configuration inside the domain.
\end{rem}

The rest of the section is dedicated to the proofs of the two theorems above. 
The RSW theorem developed in Section~\ref{sec:RSW_weak} will be of great use also in Section~\ref{sec:dichotomy}. 

\subsection{Infinite-volume measure for red marginal}

We will work here only with the red-spin marginals~$\nu$ of the measures~$\mu$. 

\begin{thm}[Limiting measure for red spins]
\label{thm:limit-for-red}
	For any increasing sequence of domains ~$(\calD_n)_{n \geq 0}$ with~$\bigcup_{n\geq0} \calD_n = \bbH$, the sequence of measures~$\nu_{\calD_n}^{\srp\srp}$ converges to a measure~$\nu^{\srp\srp}_\bbH$ on~$\{\rm,\rp\}^{F(\bbH)}$.
	Moreover ~$\nu^{\srp\srp}_\bbH$ is translation-invariant, ergodic with respect to translations and positively associated. 
\end{thm}

\begin{proof}
	Let~$\calD$ and~$\tilde\calD$ be two finite domains, with~$\calD \subset \tilde\calD$. 
	Due to the FKG inequality (Theorem~\ref{thm:FKG}) and to the Spatial Markov property (Theorem~\ref{thm:DMP}) 
	for the boundary conditions~$\rp\rp$, 
	$$ \nu_\calD^{\srp\srp}\geq_\textrm{st} \nu_{\calD'}^{\srp\srp},$$
	where the above only refers to the restrictions of the measures to~$\calD$. 
	Thus, the sequence of measures~$(\nu^{\srp\srp}_{\calD_n})_{n\geq 0}$ is decreasing,
	hence converges to a measure on~$\{\rm,\rp\}^{F(\bbH)}$, which we denote by~$\nu^{\srp\srp}_\bbH$.
	
	Since the limit exists for any sequence of domains, it necessarily is the same for any sequence of domains~$(\calD_n)_{n\geq 0}$. 
	In particular, the same limit is obtained for any sequence~$(\La_n + z)_{n\geq 1}$ with~$z \in V(\bbH)$, 
	which implies that ~$\nu^{\srp\srp}_\bbH$ is invariant under translations. 
	
	That~$\nu^{\srp\srp}_\bbH$ is positively associated for increasing events depending only on the state of finitely many faces
	follows by passing to the limit. The property extends to arbitrary increasing events by the monotone class theorem 
	(see \cite[Prop.~4.10]{Gri06}). 
	
	In order to prove that~$\nu^{\srp\srp}_\bbH$ is ergodic, we will show that it has the following mixing property. 
	The measure~$\nu^{\srp\srp}_\bbH$ is said to be {\em mixing} if, 
	for any events~$A, B$, if~$\tau_x(B)$ denotes the translation of~$B$ by some~$x \in V(\bbH)$, 
	\begin{align}\label{eq:mixing}
		\lim_{|x|\to\infty}\nu^{\srp\srp}_\bbH(A\cap \tau_x(B)) = \nu^{\srp\srp}_\bbH(A)\nu^{\srp\srp}_\bbH(B).
	\end{align}
	The above implies that~$\nu^{\srp\srp}_\bbH$ is ergodic with respect to translations, as explained in \cite[Cor.~4.23]{Gri06}.
	By the monotone class theorem, it suffices to prove~\eqref{eq:mixing} for events~$A$ and~$B$ that are increasing and only depend on finitely many faces. 
	We do this below. 
	
	Let~$A,B$ be increasing events depending only on the states of faces in some finite domain of~$\bbH$.
	Fix~$\epsilon > 0$. Then there exists~$N \geq 1$ such that 
	$\nu^{\srp\srp}_\bbH(A) \geq \nu^{\srp\srp}_{\La_N}(A) - \epsilon$ and~$\nu^{\srp\srp}_\bbH(B) \geq \nu^{\srp\srp}_{\La_N}(B) - \epsilon$. 
	Then, for any~$x \in V(\bbH)$ with~$|x| > 2N+2$, by positive association of~$\nu^{\srp\srp}_\bbH$ and \eqref{eq:domain-markov},
	\begin{align*}
		\nu^{\srp\srp}_\bbH(A\cap \tau_x(B)) 
		&\leq \nu^{\srp\srp}_\bbH[A\cap \tau_x(B) \,|\, \sigma_r \equiv \rp \text{ outside~$\Int(\La_N) \cup \tau_x(\Int(\La_N))$}] \\
		&=\nu^{\srp\srp}_{\La_N}(A) \, \nu^{\srp\srp}_{\La_N}(B)\\
		&\leq  \nu^{\srp\srp}_\bbH(A)	\, \nu^{\srp\srp}_\bbH(B) + 2\epsilon. 
	\end{align*}
	Conversely, positive association implies that 
	\begin{align*}
		\nu^{\srp\srp}_\bbH(A\cap \tau_x(B))
		\geq  \nu^{\srp\srp}_\bbH(A)\nu^{\srp\srp}_\bbH(B).
	\end{align*}	
	These two inequalities and the fact that~$\epsilon > 0$ is arbitrary imply~\eqref{eq:mixing}, and hence the ergodicity of~$\nu^{\srp\srp}_\bbH$.
\end{proof}

\subsection{Crossing estimates for double-plus percolation (weak version)}\label{sec:RSW_weak}

We will work in the rest of the paper with two percolation models derived from spin configurations.
Let us describe them for a red spin configuration~$\sigma_r$ on~$\bbH$; 
the definitions adapt readily to blue spins, to~$-$ instead of~$+$, and to domains of~$\bbH$.
  
The first corresponds to connections via face-paths of spins~$\rp$. 
This percolation, along with its paths, clusters etc. will be referred to as {\em simple-$\rp$}; 
connections between two sets of faces~$A$ and~$B$ are denoted by~$A \xlra{\srp} B$.  
This notion was implicitly used in the proof of Theorem~\ref{thm:FKG}.

The second is termed {\em double-$\rp$} percolation. 
The double-plus configuration~$\dbp(\sigma_r) \in\{0,1\}^{E(\bbH)}$ associated to~$\sigma_r$ 
is formed of the edges of~$\bbH$ whose two adjacent faces have spins~$\rp$. 
We regard~$\dbp(\sigma_r)$ as a bond percolation on~$\bbH$ and use the ensuing notion of connectivity. 
In particular, for sets~$A$ and~$B$ of vertices, we write~$A \xlra{\srp\srp}B$
for the event that there exists an edge-path in~$\dbp(\sigma_r)$ with one endpoint in~$A$ and the other in~$B$. 
More generally, we call a double-$\rp$ path, or a double-path of spin~$\rp$, a path of edges in~$\dbp(\sigma_r)$.
All notions related to this percolation (clusters, crossings, circuits etc.) will be referred to as double-$\rp$. 
Thus, the event~$\Circ_{\srp\srp}(n)$ of Theorem~\ref{thm:uniqueness_nu} may be described as
the existence of a double-$\rp$ circuit surrounding~$\La_n$. 
The appeal of this second, more restrictive percolation model is that double-$\rp$ circuits isolate the inside from the outside 
in the sense of the Spatial Markov property \eqref{eq:domain-markov}.

Since~$\sigma_r$ is positively correlated under~$\mu_\calD$, so is~$\dbp(\sigma_r)$. 
Indeed any event~$A$ which is increasing for~$\dbp(\sigma_r)$ is also increasing for~$\sigma_r$. 
A {double-minus} configuration~$\dbm(\sigma_r) \in\{0,1\}^{E(\bbH)}$ is defined in a similar way and is also positively correlated under~$\mu_\calD$. We want to stress however that the union~$\dbp(\sigma_r) \cup \dbm(\sigma_r)$ is not necessarily positively correlated.

\newcommand{\R}{\mathsf{R}}

Write~$\Par_{m,n}$ for the set of faces of~$\bbH$ with centres at~$k + \ell e^{i\pi/3}$ with~$0 \leq k \leq m$ and~$0 \leq \ell \leq n$
(see Figure~\ref{fig:simple_cross}). 
They form a domain approximately shaped as a parallelogram. 
Its boundary~$\partial_E \Par_{m,n}$ may be partitioned into four sides called 
$\Bottom(m,n)$, $\Right(m,n)$, $\Top(m,n)$ and~$\Left(m,n)$, defined as their name indicates. 
To be precise,~$\Right(m,n)$ and~$\Left(m,n)$ start and end with vertical edges. 
Below we will also use the notation~$\Bottom(m,n)$, $\Right(m,n)$, $\Top(m,n)$ and~$\Left(m,n)$ 
to refer to the faces of~$\partial_\int \Par_{m,n} \cup \partial_\out \Par_{m,n}$ 
that are adjacent to these sections of~$\partial_E\Par_{m,n}$.
Faces in the corners of~$\Par_{m,n}$ belong to two such sets. 

Write~$\calC^h_{\srp}(m,n)$ (and~$\calC_{\srp}^v(m,n)$) for the event that there exits a face-path in~$\Par_{m,n}$ 
formed only of faces with spin~$\rp$,
with the first face adjacent to the left side of~$\Par_{m,n}$ and the last face adjacent to the right side 
(and top and bottom sides, respectively). 
Call such face-paths horizontal (respectively, vertical) \emph{simple-$\rp$ crossings} of~$\Par_{m,n}$. 

Write~$\calC^h_{\srp\srp}(m,n)$ (and~$\calC_{\srp\srp}^v(m,n)$) for the event that there exits a path of edges of~$\dbp(\sigma_r)$ 
contained in~$\Par_{m,n}$ with one endpoint on the left side of~$\Par_{m,n}$ and the other on the right side (and top and bottom sides, respectively); for technical reasons, we ask that the endpoints of the paths not be corners of~$\Par_{m,n}$. 
We call such paths horizontal (and vertical, respectively) {\em double-$\rp$ crossings} of~$\Par_{m,n}$. 

The Russo-Seymour-Welsh (or RSW for short) theory first appeared in the simultaneous works of Russo and Seymour and Welsh for Bernoulli percolation \cite{Rus78,SeyWel78}. 
Its ultimate conclusion is that rectangles are crossed with probability bounded by constants that only depend on the rectangles' aspect-ratios, not their sizes. Such crossing probability bounds were obtained for Bernoulli percolation using two separate arguments:
\begin{itemize}
    \item a self-duality argument proves that the probability of crossing a square of any size is~$1/2$ 
    (or more generally bounded uniformly away from~$0$);
    \item the so-called RSW lemma proves that crossing a rectangle of aspect ratio 2 in the long direction 
    is bounded by a function of the probability of crossing a square of (roughly) the same size. 
\end{itemize}
The same two step procedure will be used below for the double-$\rp$ percolation. 
While for bond percolation on~$\bbZ^2$ with parameter~$1/2$ the first point is immediate due to self-duality, 
in our context a more complex argument is needed. 
The second point also requires special attention, due to the lack of independence and even of a general Spatial Markov property. 
A weak version of the RSW lemma is obtained easily using a general argument due to Tassion \cite{Tas16} 
(see Proposition~\ref{prop:RSW_Vincent} below).
A more elaborate statement is proved later on (see Proposition~\ref{prop:RSW_Hugo}); it requires considerable work.

\subsubsection{Crossings of symmetric domains}\label{sec:sym_domains}

This part contains results on crossing of symmetric domains; they are akin to the consequences of self-duality for site percolation on the triangular lattice or bond percolation on~$\bbZ^2$. 
Two type of crossings will be treated: simple-$\rp$ crossings and double-$\rp$ crossings. 
We start with the former, where self-duality applies as for percolation. 

\begin{lem}\label{lem:simple_self_duality}
	Let~$\calD$ be a domain containing~$\Par_{n,n} \cup \partial_\out \Par_{n,n}$ for some~$n$.
	Let~$\zeta \in\{\rp,\rm\}^{F(\calD)}$ be such that~$\zeta \equiv \rp$ on~$\Bottom(n,n)\cup\Top(n,n)$. 
	Then
	\begin{align}\label{eq:simple_self_duality}
		\mu_\calD\big[ \calC^v_{\srp}(n,n) \,\big|\, \sigma_r = \zeta \text{ outside~$\Int(\Par_{n,n})$}\big] \geq \tfrac13.
	\end{align}
\end{lem}

\begin{proof}
	Fix~$\calD$,~$\zeta$ and~$n$. Drop~$n$ from the notation~$\Par$,~$\Bottom$ and~$\Top$. 
	First observe that by the monotonicity in boundary conditions (Corollary~\ref{cor:monotonicity_bc} \textit{(i)}),
	the LHS of \eqref{eq:simple_self_duality} is minimal when~$\zeta \equiv \rm$ on~$\calD \setminus (\Bottom(n,n)\cup\Top(n,n))$.
	We will assume this to be the case. 
	All faces of~$\Bottom(n,n)\cup\Top(n,n)$ have spin~$\rp$ in~$\zeta$ as required by the proposition; 
	we will switch the sign of the two left-most faces of~$\Top(n,n)$ to~$\rm$ 
	-- this only decreases further the LHS of  \eqref{eq:simple_self_duality}.
	
	For~$\sigma_r$ a red spin configuration on~$\Int(\Par)$, write~$\sigma_r \cup \zeta$ 
	for the configuration on~$\calD$ obtained by completing~$\sigma_r$ with~$\zeta$ on~$\calD \setminus \Int(\Par)$.
	Write~$\tau(\sigma_r)$ for the configuration obtained by applying the symmetry with respect to the line~$e^{i\pi/6}\bbR$ to~$-\sigma_r$. 
	It is a known fact (see duality of site-percolation on~$\bbT$ \cite[Sec. 1.2]{Kes82}) 
	that either~$\sigma_r \in \calC^v_{\srp}(n,n)~$ or~$\tau(\sigma_r) \in \calC^v_{\srp}(n,n)$. 

	Recall from Proposition~\ref{prop:marginal-ditribution}  
	that~$\mu_\calD^{\srm\srm}(\sigma_r \cup \zeta)$ is proportional to~$2^{k(\theta(\sigma_r\cup \zeta))}$. 
	Also notice that~$\theta(\tau(\sigma_r) \cup \zeta)$ is easily determined in function of~$\theta(\sigma_r\cup \zeta)$.
	Indeed, the configuration~$\zeta$ restricted to~$\partial_\int\Par \cup \partial_\out\Par$ is (almost) invariant under~$\tau$. 
	Thus,~$\theta(\tau(\sigma_r) \cup \zeta)$ restricted to~$\Par \cup \partial_\out\Par$ 
	is the reflection with respect to the line~$e^{i\pi/6}\bbR$ of~$\theta(\sigma_r \cup \zeta)$. 
	See Figure~\ref{fig:simple_cross}.
	
	\begin{figure}
	\begin{center}
	\includegraphics[page=2, width=0.46\textwidth]{square_cross5.pdf}\hspace{0.01\textwidth}
	\includegraphics[page=1, width=0.46\textwidth]{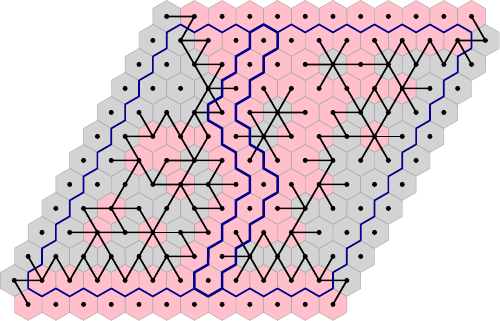}
	\caption{A configuration~$\sigma_r$ of red spins ($\frp$ is depicted in red,~$\frm$ in gray) on~$\Int(\Par_{n,n})$ with the configuration~$\zeta$ on~$\partial_\int\Par \cup\partial_\out \Par$. 
	In the left image no simple-$\frp$ crossing exists between the red arcs~$\Top$ and~$\Bottom$. 
	The right image is~$\tau(\sigma_r) \cup\zeta$; it contains a simple-$\frp$ crossing between these arcs. 
	Observe that~$\theta(\tau(\sigma_r) \cup \zeta)$ is the reflection of~$\theta(\sigma_r \cup \zeta)$ with respect to the diagonal of the rhombus.
	In this concrete example~$k[\theta(\sigma_r\cup \zeta)] - k[\theta(\tau(\sigma_r)\cup \zeta)]=1$ because the clusters of~$\Top$ and~$\Bottom$ are linked together after the reflection but not before.
	 }
	\label{fig:simple_cross}	
	\end{center}
	\end{figure}

	All other edges have same state in~$\theta(\tau(\sigma_r) \cup \zeta)$ and~$\theta(\sigma_r \cup \zeta)$: 
	they are determined by~$\zeta$ and are quite simple.
	Indeed, in~$\theta(\zeta)$, all faces except those in the corners of~$\partial_\out\Par$ are isolated points.
	Moreover, the top corners of~$\partial_\out\Par$ are connected, as are the bottom ones. 
	
	It follows that, for any~$\sigma_r$,
	$| k[\theta(\sigma_r\cup \zeta)] - k[\theta(\tau(\sigma_r)\cup \zeta)] | \leq 1.$ 
	Thus
	\begin{align}\label{eq:abs_cont}
		\mu_\calD( \sigma_r\cup \zeta)  
		\geq \tfrac{1}2\,
		\mu_\calD( \tau(\sigma_r)\cup \zeta). 
	\end{align}
	By summing the above over~$\sigma_r \in \calC^v_{\srp}(n,n)$ we find
	\begin{align*}
		\mu_\calD[\calC^v_{\srp}(n,n) \,|\, \sigma_r = \zeta \text{ outside~$\Int(\Par)$}]
		\geq \tfrac12 	\mu_\calD[ \calC^v_{\srp}(n,n)^c \,|\, \sigma_r = \zeta \text{ outside~$\Int(\Par)$}].
	\end{align*}
	This proves the desired bound. 
\end{proof}

Next we turn to double-$\rp$ crossings. The absence of such a crossing does not induce the existence of a double-$\rm$ crossing, 
and we may not apply the same argument as above. We do however have a similar statement. 

\begin{lem}\label{lem:monochrom}
	For any~$m,n \geq 1$, and any pair of coherent configurations~$\sigma_r\in \{\rp,\rm\}^{F(\Par_{m,n})}$ 
	and~$\sigma_b\in \{\bp,\bm\}^{ F(\Par_{m,n})}$,
	either~$\Par_{m,n}$ is crossed horizontally by a double path of constant red spin, 
	or it is crossed vertically be a double path of constant blue spin. 
	That is
	\begin{align}\label{eq:monochrom}
		\big[\calC^{v}_{\sbp\sbp}(m,n) \cup \calC^{v}_{\sbm\sbm}(m,n)\big] 
		\subset \big[\calC^{h}_{\srp\srp}(m,n) \cup \calC^{h}_{\srm\srm}(m,n)\big]^c.
	\end{align}
\end{lem}

\begin{proof}
	Recall that~$\dbp(\sigma_r),\dbm(\sigma_r),\dbp(\sigma_b),\dbm(\sigma_b)\subset E(\Par_{m,n})$ 
	denote the sets of double plus and double minus edges in~$\sigma_r$ and~$\sigma_b$, respectively. 
	Also, recall that to each edge~$e\in E(\bbH)$ we associate its dual~$e^*\in E(\bbT)$ that is defined as the unique edge on~$\bbT$ that intersects~$e$. For a set~$S\subset E(\bbH)$ we denote by~$S^*\subset\bbT$ the set of edges dual to the edges in~$S$.
	
	By duality between~$\bbH$ and~$\bbT$, either~$\dbp(\sigma_r)\cup\dbm(\sigma_r)$  contains a left-right crossing of~$\Par_{m,n}$, or~$\left[E(\Par_{m,n})\setminus(\dbp(\sigma_r)\cup\dbm(\sigma_r))\right]^*$ contains a top-bottom crossing of~$\Par_{m,n}$.
	
	First consider the case when~$\dbp(\sigma_r)\cup\dbm(\sigma_r)$  contains a left-right crossing of~$\Par_{m,n}$. 
	Any such crossing consists either entirely of edges of~$\dbp(\sigma_r)$ or entirely of edges of~$\dbm(\sigma_r)$.
	Indeed, edges of~$\dbp(\sigma_r)$ and~$\dbm(\sigma_r)$ can never share a vertex. 
	In conclusion, in this case at least one of~$\calC^{h}_{\srp\srp}(m,n)$ and~$\calC^{h}_{\srm\srm}(m,n)$ occurs.
	
	It remains to consider the case when~$\left[E(\Par_{m,n})\setminus(\dbp(\sigma_r)\cup\dbm(\sigma_r))\right]^*$ contains a top-bottom crossing of~$\Par_{m,n}$. Let~$\gamma^*$ be such a crossing. 
	
	For each edge~$e\in E(\calD)$, let~$N(e)\subset E(\Par_{m,n})$ denote the set of edges consisting of~$e$ and all edges in~$E(\Par_{m,n})$ that share a vertex with~$e$. Then, if~$e^*\in \gamma^*$, we claim that~$N(e)\subset \dbp(\sigma_b)\cup \dbm(\sigma_b)$. 
	Indeed, the two faces of ~$\Par_{m,n}$ separated by~$e$ have opposite red spin, hence same blue spin. 
	Moreover, the blue spins of the two faces adjacent the endpoints of~$e$ but not containing~$e$ in their boundary 
	must also coincide with the spins on either side of~$e$. 
	
	It remains to observe that the union of~$N(e)$ taken over all~$e$ such that~$e^*\in \gamma^*$ 
	contains a top-bottom crossing of~$\Par_{m,n}$. 
	Thus~$\dbp(\sigma_b)\cup \dbm(\sigma_b)$ contains a top-bottom crossing of~$\Par_{m,n}$, 
	and thus either~$\calC_{\sbp\sbp}^v(m,n)$ or~$\calC_{\sbm\sbm}^v(m,n)$ occurs.
%
%
%
%
\end{proof}

\begin{rem}\label{rem:monochrom}
	It is obvious from the proof that Lemma~\ref{lem:monochrom} may be generalised to other domains with four arcs marked on the boundary. 
	
	Later we will also use the fact that, if an annulus~$\La_N \setminus \La_n$ does not contain a circuit around~$\La_n$ of either double-$\rp$ or double-$\rm$, then~$\La_n$ is connected to~$\La_N^c$ by a double-path of constant blue spin.
\end{rem}

\begin{lem}\label{lem:square_crossed}
	For any~$n \geq 1$ and any domain~$\calD$ containing~$\Par_{n,n}$ and symmetric with respect to one of the diagonals of~$\Par_{n,n}$, 
	\begin{align*}
		\mu^{\srp\srm}_\calD\big[ \calC^h_{\srp\srp}(n,n)\big] \geq \tfrac14.
	\end{align*}
\end{lem}

\begin{proof}
	By Lemma~\ref{lem:monochrom} we have 
	\begin{align*}
		1 
		&\leq \mu^{\srp\srm}_{\calD}\big[ \calC^h_{\srp\srp}(n,n)\big]
		+\mu^{\srp\srm}_{\calD}\big[ \calC^h_{\srm\srm}(n,n)\big]
		+\mu^{\srp\srm}_{\calD}\big[ \calC^h_{\sbp\sbp}(n,n)\big]
		+\mu^{\srp\srm}_{\calD}\big[ \calC^h_{\sbm\sbm}(n,n)\big]\\
		& = 2\,\mu^{\srp\srm}_{\calD}\big[ \calC^h_{\srp\srp}(n,n)\big]
		+2\, \mu^{\srm\srp}_{\calD}\big[ \calC^h_{\srp\srp}(n,n)\big]\\
		&\leq 4\, \mu^{\srp\srm}_{\calD}\big[ \calC^h_{\srp\srp}(n,n)\big].
	\end{align*}
	In the equality we used the fact that the blue spin marginal of~$\mu^{\srp\srm}_\calD$ is~$\frac12(\nu^{\sbp\sbm}_{\calD} +\nu^{\sbm\sbp}_{\calD})$ and that~$\mu^{\srp\srm}_{\calD}[ \calC^h_{\srm\srm}(n,n)]= \mu^{\srm\srp}_{\calD}[ \calC^h_{\srp\srp}(n,n)]$ ; in the last line we used that~$\nu^{\srp\srm}_{\calD}\geq_{\text{st}}\nu^{\srm\srp}_{\calD}$. 
	This provides the desired result. 
\end{proof}

\begin{cor}\label{cor:square_crossed}
	For any~$n \geq 1$ 
	\begin{align*}
		\nu^{\srp\srp}_\bbH\big[ \calC^h_{\srp\srp}(n,n)\big] \geq \tfrac14.
	\end{align*}
\end{cor}

\begin{proof}
	For any domain~$\calD$ as in Lemma~\ref{lem:square_crossed}, by the monotonicity of boundary conditions, 
	\begin{align*}
		\nu^{\srp\srp}_\calD\big[ \calC^h_{\srp\srp}(n,n)\big] \geq \tfrac14.
	\end{align*}
	By taking the limit of the above as~$\calD$ grows to~$\bbH$, we obtain the desired bound. 
\end{proof}



\subsubsection{Sub-sequential RSW}\label{sec:RSW_Vincent}

\begin{prop}[RSW]\label{prop:RSW_Vincent}
	We have
	\begin{align}\label{eq:RSW}
		\limsup_{n\to\infty}\nu^{\srp\srp}_\bbH \big[\calC^h_{\srp\srp}(2n,n)\big] >0.
	\end{align}
	As a consequence~$\limsup_{n\to\infty}\nu^{\srp\srp}_\bbH \big[\Circ_{\srp\srp}(n,2n)\big] > 0$.
\end{prop}


The proof of Proposition~\ref{prop:RSW_Vincent} uses a technique introduced by Tassion in~\cite{Tas16}. 
Indeed, the main argument of \cite[Thm.~1]{Tas16} shows that~$\liminf_{n\to\infty} \nu^{\srp\srp}_\bbH(\calC^h_{\srp}(n,n)) > 0$ 
(which is the result of Corollary~\ref{cor:square_crossed}) implies \eqref{eq:RSW}. 
This technique applies to general percolation measures with the FKG property and sufficient symmetry. 
Our model fits in this framework and the relevant part of the proof of \cite[Thm.~1]{Tas16} applies readily. 
We simply point out that, in order to harness the symmetries of the hexagonal lattice, 
one should apply the argument using crossings of hexagonal domains between opposite sides, rather than crossings of squares or lozenges. 

Note that \cite[Thm.~1]{Tas16} actually claims a stronger result than~\eqref{eq:RSW}, where~$\limsup$ is replaced by~$\liminf$. 
This improvement requires an additional ingredient which is lacking here. For now we are content with the above sub-sequential form of RSW. 

A stronger statement (with the lower bound valid for all~$n$) will be proved in Section~\ref{sec:dichotomy} -- see Proposition~\ref{prop:RSW_Hugo}. 
All the ingredients for it are already available, however the proof is tedious and is not necessary at this point. 
The argument of \cite{Tas16} is elegant, short and quite robust, and suffices to prove Theorem~\ref{thm:uniqueness_nu}; we prefer it for now. 
%

\begin{proof}
	The argument of \cite[Thm.~1]{Tas16} requires minor modifications 
	because the hexagonal lattice in invariant under rotations of~$\pi/3$, unlike the square one, which is invariant under rotations of~$\pi/2$.
	We briefly sketch the adapted argument below. 
	
	Write~$\sfT$ and~$\sfB$ for the top and bottom horizontal sections of~$\partial \La_n$, and 
	let~$\calC_{\srp\srp}^v(\La_n)$ be the event that~$\sfT$ and~$\sfB$
	are connected to each other by a double-$\rp$ path contained in~$\La_n$. 
	From Corollary \ref{cor:square_crossed}, 
	using standard applications of the FKG inequality and the invariance of~$\La_n$ under rotations by multiples of~$\pi/3$, 
	we deduce that~$\nu^{\srp\srp}_\bbH \big[\calC^h_{\srp\srp}(2n,n)\big]$ is bounded away from~$0$ uniformly in~$n$. 
	
	Following \cite{Tas16}, define~$2\alpha_n$ as the maximal width of a centred interval~$I$ on~$\sfB$ 
	such that 
	$$ \nu^{\srp\srp}_\bbH \big[I \xlra{\srp\srp \text{ in }\La_n} \sfT \big] \leq 	
	\nu^{\srp\srp}_\bbH \big[\sfB \setminus I \xlra{\srp\srp \text{ in }\La_n} \sfT \big].$$
	
	\begin{figure}
    \begin{center}
    \includegraphics[width = 0.39\textwidth]{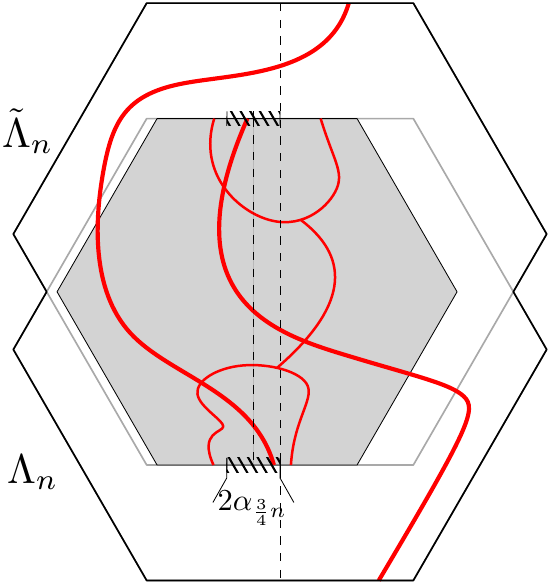}
    \caption{The construction that proves that if~$\alpha_{n} \leq \frac12 \alpha_{3n/4}$, then~$\La_n \cup \tilde\La_{n}$\
    is crossed vertically with uniformly positive probability.}
    \label{fig:tassion}
    \end{center}
\end{figure}

	Call~$\tilde \La_n$ the vertical translate of~$\La_n$ by~$n/2$ and~$\widetilde{\sfT}$ the corresponding translate of~$\sfT$. 
	Then, using the same argument as in \cite[Lemma 2.2]{Tas16} (see also Figure \ref{fig:tassion}), 
	$\alpha_{n} \leq 2 \alpha_{3n/4} \leq n/4$ 	implies that 
	\begin{align*}
		 \nu^{\srp\srp}_\bbH \big[\sfB \xlra{\srp\srp \text{ in }\La_n \cup \tilde \La_n} \sfT \big] >c,
	\end{align*}
	for some constant~$c> 0$ independent of~$n$. 
	Through additional standard applications of the FKG inequality, 
	the above implies in turn that 
	$\nu^{\srp\srp}_\bbH [\calC^h_{\srp\srp}(4n,2n)]$ 
	and~$\nu^{\srp\srp}_\bbH [\Circ_{\srp\srp}(2n,4n)]$ 
	are bounded below by strictly positive quantities depending only on~$c$. 
	Moreover, if~$\alpha_n > n/4$, then lower bounds on 
	$\nu^{\srp\srp}_\bbH [\calC^h_{\srp\srp}(4n,2n)]$ 
	and~$\nu^{\srp\srp}_\bbH [\Circ_{\srp\srp}(2n,4n)]$ 
	follow by simple considerations, similar to those of the start of the proof of \cite[Lemma 2.2]{Tas16}. 
	
	Finally, since~$0\leq \alpha_n \leq n$ for all~$n$, there exist infinitely many values of~$n$ such that~$\alpha_{n} \leq 2 \alpha_{3n/4}$, and the proof is complete. 
\end{proof}

\begin{cor}\label{cor:circ_dp_dp}
	Under~$\nu^{\srp\srp}_\bbH$,~$0$ is surrounded a.s. by an infinite number of disjoint circuits of double-$\rp$. 
\end{cor}

\begin{proof}
	Suppose the opposite, that is that with positive~$\nu^{\srp\srp}_\bbH$-probability, 
	$0$ is surrounded by a finite number of disjoint double-$\rp$ circuits.
	Since~$\nu^{\srp\srp}_\bbH$ is ergodic and the above event is translation invariant, it occurs with probability~$1$. 
	
	Set~$N = \min\{n \geq 1: \, \Circ_{\srp\srp}(n) \text{ does not occur}\}$; observe that~$N$ is a random variable that is, by our assumption, 
	$\nu^{\srp\srp}_\bbH$-a.s. finite. 
	Then there exists~$n_0$ such that 
	\begin{align*}
		\nu^{\srp\srp}_\bbH(N \geq n_0) < \limsup_{n\to\infty}\nu^{\srp\srp}_\bbH \big[\Circ_{\srp\srp}(n)\big],
	\end{align*}
	since the right-hand side is strictly positive by Proposition~\ref{prop:RSW_Vincent}. Using that, for all~$n>n_0$,
	\[
		\nu^{\srp\srp}_\bbH(N \geq n_0) \geq \nu^{\srp\srp}_\bbH \big[\Circ_{\srp\srp}(n)\big],
	\]
	we obtain a contradiction.
\end{proof}

\begin{cor}\label{cor:no_inf_cluster_theta}
	The graph~$\theta(\sigma_r)$ contains~$\nu^{\srp\srp}_\bbH$-a.s. no infinite cluster. 
\end{cor}

\begin{proof}
	Observe that a double-$\rp$-circuit in~$\sigma_r$ blocks connections in~$\theta(\sigma_r)$. 
	More precisely, if~$\sigma_r \in \Circ_{\srp\srp}(n)$, then all clusters of~$\theta(\sigma_r)$ that intersect~$\La_n$ are finite. 
	Now Corollary~\ref{cor:circ_dp_dp} states that~$\nu^{\srp\srp}_\bbH(\Circ_{\srp\srp}(n)) = 1$ for all~$n$, 
	which implies by the observation above that~$\theta(\sigma_r)$ contains no infinite cluster a.s. 
\end{proof}

\subsection{Joint infinite-volume measure}

We now turn to the existence of limiting measures for the joint law of the red and blue spins, that is Theorem~\ref{thm:limit_joint}. 

The crucial property here is that by Lemma~\ref{prop:marginal-ditribution}, conditionally on~$\sigma_r$,
$\sigma_b$ is obtained by colouring independently and uniformly the clusters of~$\theta(\sigma_r)$ in either~$\bm$ or~$\bp$. 
This procedure may also be applied in infinite-volume for red-spin configurations sampled according to~$\nu^{\srp\srp}_\bbH$. 
The absence of infinite clusters in~$\theta(\sigma_r)$ is used to show that 
the result of this procedure in a finite but large volume is close to that in infinite-volume. 

\begin{proof}[Theorem~\ref{thm:limit_joint}]
	Let ~$(\calD_n)_{n \geq 0}$ be an increasing sequence of domains with~$\bigcup_{n\geq0} \calD_n = \bbH$.
	Recall from Theorem~\ref{thm:limit-for-red} that the red-spin marginals of~$\mu_{\calD_n}^{\srp\srp}$ converge
	to an ergodic translation-invariant limiting measure denoted by~$\nu_\bbH^{\srp\srp}$. 
	Let~$\mu_\bbH^{\srp\srp}$ be the measure obtained by sampling~$\sigma_r$ according to~$\nu_\bbH^{\srp\srp}$, 
	then awarding to all faces of each cluster of~$\theta(\sigma_r)$ a blue spin uniformly chosen in~$\{\bm,\bp\}$, 
	independently for each cluster. 
	Let us prove that~$\mu_{\calD_n}^{\srp\srp}$ converges to~$\mu_\bbH^{\srp\srp}$.
	
	Fix~$k\in \bbN$ and~$\eps>0$. 
	We will show that the total-variation distance between the restrictions of 
	$\mu_{\calD_n}^{\srp\srp}$ and~$\mu_\bbH^{\srp\srp}$ to~$\La_k$ 
	is smaller than~$2\eps$, provided that~$n$ is large enough. 
	
	Let~$K \geq k$ be such that~$\nu_\bbH^{\srp\srp}(\La_k \xlra{\theta(\sigma_r)} \La_K^c) <\eps$.
	Due to Corollary~\ref{cor:no_inf_cluster_theta}, it is always possible to choose~$K$ with this property.
	
	Now, let~$N = N(\eps,K)$ be such that, for any~$n \geq N$, the distance in total variation between the restrictions of 
	$\nu_{\calD_n}^{\srp\srp}$ and~$\nu_\bbH^{\srp\srp}$ to~$\La_K$ is smaller than~$\eps$.
	Thus, one may couple~$\nu_{\calD_n}^{\srp\srp}$ and~$\nu_\bbH^{\srp\srp}$ to produce configurations~$\sigma_r, \sigma_r'$
	in such a way~$\sigma_r = \sigma_r'$ on~$\La_K$ with probability at least~$1- \eps$. 
	Moreover, by choice of~$K$, with probability at least~$1 - 2\eps$,~$\sigma_r = \sigma_r'$ on~$\La_K$ 
	and there is no connected component of~$\theta(\sigma_r)$ that intersects both~$\La_k$ and~$\La_K^c$. 
	On this event, the connected components of~$\theta(\sigma_r)$ and~$\theta(\sigma_r')$ that intersect~$\La_k$ are identical. 
	Using the same blue spin assignment for these components, we have produced a coupling of~$\mu_{\calD_n}^{\srp\srp}$ and~$\mu_\bbH^{\srp\srp}$
	that is equal inside~$\La_k$ with probability at least~$1- 2\eps$, which was our goal.
	
	Since~$\eps > 0$ and~$k$ are arbitrary, we conclude that ~$\mu_{\calD_n}^{\srp\srp}$ converges to~$\mu_\bbH^{\srp\srp}$.
		 
	The translation invariance of~$\mu_\bbH^{\srp\srp}$ follows from that of~$\nu_\bbH^{\srp\srp}$. 
	Since~$\mu_{\La_n}^{\srp\srp}$ is invariant under rotations by multiples of~$\pi/3$ and converges to~$\mu_\bbH^{\srp\srp}$, 
	the latter is also invariant under such rotations. 
	The ergodicity of~$\mu_\bbH^{\srp\srp}$ follows from that of~$\nu_\bbH^{\srp\srp}$
	and from the absence of infinite clusters in~$\theta(\sigma_r)$. 
\end{proof}

\begin{prop}\label{prop:Burton_Keane_blue}
	Under~$\mu_{\bbH}^{\srp\srp}$,~$\sigma_b$ contains a.s. no infinite~$\bp$-cluster and no infinite~$\bm$-cluster.
	As a consequence,~$\omega_b$ is formed entirely of finite loops~$\mu^{\srp\srp}_\bbH$-a.s. 
\end{prop}

The proof below is a straightforward application of the uniqueness argument of Burton and Keane \cite{BurKea89} 
and of Zhang's argument for non-coexistence of clusters 
(see \cite[Lem 11.12]{Gri99a} for an illustration of this argument which was never published by Zhang himself).

\begin{proof}
	To start, observe that under~$\mu_{\bbH}^{\srp\srp}$ the number~$N_{\sbp}$ of infinite~$\bp$-clusters is a.s. constant. 
	This is a direct consequence of the ergodicity of~$\sigma_b$ under this measure. The same applies to infinite~$\bm$-clusters. 
	
	The technique introduced by Burton--Keane in \cite{BurKea89} applies readily to the blue-spin marginal under~$\mu_{\bbH}^{\srp\srp}$. 
	Indeed, this marginal satisfies the finite-energy property required by  \cite{BurKea89}. 
	As a consequence we obtain that either~$N_{\sbp} = 0$~$\mu_{\bbH}^{\srp\srp}$-a.s. or~$N_{\sbp} = 1$~$\mu_{\bbH}^{\srp\srp}$-a.s.
	
	Finally, let us prove that ~$N_{\sbp} = 0$~$\mu_{\bbH}^{\srp\srp}$-a.s. by contradiction. 
	Assume that~$N_{\sbp} = 1$~$\mu_{\bbH}^{\srp\srp}$-a.s. Then, by the symmetry of the blue-spin marginal, 
	the number of~$\bm$-infinite clusters is also equal to~$1$ a.s.
	Thus, there exists some~$n \geq 1$ such that~$\mu_{\bbH}^{\srp\srp}(\La_n \xlra{\sbp} \infty) > 1 - 1/4^{6}$. 
	Write~$\partial_1\La_n,\dots, \partial_6 \La_n$ for the six sides of~$\partial \La_n$ in counter-clockwise order. 
	Then 
	\begin{align*}
		\mu_{\bbH}^{\srp\srp}(\La_n \nxlra{\sbp}\infty)
		 = \mu_{\bbH}^{\srp\srp}\big(\bigcap_{j=1}^6 \{\partial_j \La_n \xlra{\sbp \text{ in~$\La_n^c$}}\infty \}^c\big)
		\geq \mu_{\bbH}^{\srp\srp}\big(\{\partial_1 \La_n \xlra{\sbp \text{ in~$\La_n^c$}}\infty \}^c\big)^6.
	\end{align*}
	The inequality is due to the FKG property for~$\sigma_b$ and to the invariance of the measure under rotations by~$\pi/3$. 
	Thus we find 
	\begin{align*}
		\mu_{\bbH}^{\srp\srp}\big(\partial_1 \La_n \xlra{\sbp \text{ in~$\La_n^c$}}\infty \big)
		\geq 1 - \big(1 - \mu_{\bbH}^{\srp\srp}(\La_n \nxlra{\sbp}\infty)\big)^{1/6}
		> 3/4.
	\end{align*}
	The same holds for any side of~$\La_n$ and also for~$\bm$-connections instead of~$\bp$ ones. 
	
	Define the event 
	$\calA_{\sbp} := \{\partial_1 \La_n \xlra{\sbp \text{ in~$\La_n^c$}}\infty\} \cap \{\partial_3 \La_n \xlra{\sbp \text{ in~$\La_n^c$}}\infty\}$ and 
	$\calA_{\sbm} := \{\partial_2 \La_n \xlra{\sbm \text{ in~$\La_n^c$}}\infty\} \cap \{\partial_4 \La_n \xlra{\sbm \text{ in~$\La_n^c$}}\infty\}$. 
	Using the union bound, we find
	\begin{align*}
		\mu_{\bbH}^{\srp\srp}\big(\calA_{\sbp} \cap\calA_{\sbm}\big) >0.
	\end{align*}
	Now notice that when the above event occurs, then necessarily either there exist two infinite~$\bp$-clusters or two infinite~$\bm$-clusters. 
	This contradicts the uniqueness of the infinite cluster proved above. 
	
	Finally, the existence of an infinite path in~$\omega_b$ implies the existence of both infinite~$\bp$ and~$\bm$ clusters, 
	which was excluded above. 
\end{proof}

\subsection{Uniqueness of infinite-volume measure: proof of Theorem~\ref{thm:uniqueness_nu}}\label{sec:infinite_vol_uniqueness}

\begin{proof}[Theorem~\ref{thm:uniqueness_nu}]
	Let us first prove that~$\mu^{\srp\srp}_\bbH$-a.s., 
	there exist infinitely many loops surrounding the origin in the loop representation of~$(\sigma_r,\sigma_b)$.	
	To that end, it is enough to show that for any~$n$,~$\mu^{\srp\srp}_\bbH$-a.s. there exists at least one loop surrounding~$\Lambda_n$.
	Fix~$n$ and consider the union of all~$\bp$- and~$\bm$-clusters that intersect~$\Lambda_n$. 
	Due to Proposition~\ref{prop:Burton_Keane_blue}, all these clusters are finite.  
	The outer boundary of their union is then a finite blue loop surrounding~$\La_n$.  Hence,~$0$ is surrounded a.s. by infinitely many loops. 

	Let us now prove \eqref{eq:uniqueness_nu}.
	Fix~$n$. By the above,~$\mu^{\srp\srp}_\bbH$-a.s. there exist infinitely many loops surrounding~$\Lambda_n$ 
	which may be ordered starting from the inner most.
	Since each loop is blue or red with probability~$1/2$ independently, 
	there exist a.s. four consecutive loops~$\gamma_1,\dots, \gamma_4$ surrounding~$\Lambda_n$ that have colours red, blue, red, blue, in this order, from inside out. 
	Then both~$\gamma_1$ and~$\gamma_3$ have constant blue spins on all faces adjacent to them, but that for~$\gamma_1$ is opposite to that for~$\gamma_3$. That is, either~$\gamma_1$ is double-$\bp$ and~$\gamma_3$ is double-$\bm$ or~$\gamma_1$ is double-$\bm$ and~$\gamma_3$ is double-$\bp$. Similarly, of~$\gamma_2$ and~$\gamma_4$, one is double-$\rp$ and the other is double-$\rm$. 
	This proves \eqref{eq:uniqueness_nu}.
	
	A direct consequence of~$\mu^{\srp\srp}_\bbH(\Circ_{\srm\srm}(n))=1$ is that the restriction of~$\nu^{\srp\srp}_\bbH$ to~$\La_n$ is 
	dominated by~$\nu^{\srm\srm}_\bbH$. Thus~$\nu^{\srp\srp}_\bbH=\nu^{\srm\srm}_\bbH$. 
	Moreover, due to the monotonicity of boundary conditions, for any sequence of finite domains~$\calD_n$ that increases to~$\bbH$, 
	the measures 
	$\nu_{\calD_n}^{\srp\srm}$ and~$\nu_{\calD_n}^{\srm\srp}$, as well as the red-spin marginals 
	of~$\mu_{\calD_n}^{\sbp\sbp}$,~$\mu_{\calD_n}^{\sbp\sbm}$,~$\mu_{\calD_n}^{\sbm\sbp}$ and~$\mu_{\calD_n}^{\sbm\sbm}$
	all converge to~$\nu_\bbH^{\srp\srp}$.
	
	Finally, due to the procedure that selects blue spins knowing the red spins, we conclude that 
	$\mu_{\calD_n}^\xi \xrightarrow[n\to\infty]{} \mu^{\srp\srp}_\bbH$ for all boundary conditions~$\xi \in \{\rp\rp,\rp\rm,\rm\rp,\rm\rm,\bp\bp,\bp\bm,\bm\bp,\bm\bm\}$.
%
%
%
	\end{proof}

\section{A dichotomy theorem}\label{sec:dichotomy}

\newcommand{\Mid}{\mathsf{Mid}}
\newcommand{\Line}{\mathsf{L_h}}
\newcommand{\Linev}{\mathsf{L_v}}
\newcommand{\Ext}{\mathrm{Ext}}
\newcommand{\centre}{\mathsf{mid}}
\newcommand{\wig}{\mathsf{wig}}
\newcommand{\loc}{\mathsf{loc}}
\newcommand{\wb}{\mathsf{wb}}
\newcommand{\hor}{\mathsf{hor}}
\newcommand{\connect}{\mathsf{cnt}}
\renewcommand{\Int}{\mathsf{Int}}
\renewcommand{\Ext}{\mathsf{Ext}}

Below we state a dichotomy result similar to those of \cite{DumSidTas17} and \cite{DumRaoTas18}. 
The result states that the model is in one of two states: co-existence of phases (see case \textit{(i)} of Corollary~\ref{cor:dicho}) or
(stretched)-exponential decay of diameters for clusters of one phase inside the other (see case \textit{(ii)}). 
In Section~\ref{sec:macro}, we show that the latter case contradicts Theorem~\ref{thm:uniqueness_nu}. 

Compared to the setting of \cite{DumSidTas17} and \cite{DumRaoTas18},  
the present model exhibits considerable additional difficulties 
due to the lack of a general Spatial Markov property and to the absence of monotonicity in the boundary conditions~$\rp\rm$.
These difficulties appear in several places, most notably in proving a crossing estimate inside mixed boundary conditions (Corollary~\ref{cor:RSW_mixed_bc})
and in eliminating case \textit{(ii)}.

\begin{thm}\label{thm:dicho}
	There exist constants~$\rho >  2$ and~$C>1$ such that, if for~$n\geq 1$ we set 
	$\alpha_n = \mu_{\La_{\rho n}}^{\srm\srm}[\Circ_{\srp\srp}(n,2n)]$, we have
	\begin{align}\label{eq:recurrence}
		\alpha_{(\rho+2)n} \leq C \alpha_n^2 \qquad \qquad \text{for all~$n\geq 1$}.
	\end{align}
\end{thm}

\begin{cor}\label{cor:dicho}
	For~$\rho > 2$ given by the above, one of the two following statements holds
	\begin{itemize}
	\item[(i)]~$\inf_n \alpha_n >0$ or
	\item[(ii)] there exist constants~$c,C>0$ and~$n_0 \geq 1$ such that~$\alpha_n \leq C e^{-n^c}$ 
	for all~$n = (\rho+2)^k n_0$ with~$k\in \bbN$. 
	\end{itemize} 
\end{cor}

The constant~$\rho$ in Theorem~\ref{thm:dicho} will be chosen large enough to accommodate certain geometric constructions used in the proof. 
Its choice only affects scenario (ii) of Corollary~\ref{cor:dicho}, which we will see is contradictory. 
Thus, any value of~$\rho$ suffices for our purposes.

\subsection{Preparation: measure in cylinder}\label{subsec:cylinder}

Write~$[a,b] \times [c,d]$ for the set of faces of~$\bbH$ with centres inside~$[a,b] \times [c,d]$. 
Any such rectangle is a domain of~$\bbH$ and we will treat it as such. 
Its boundary may be split into four segments: bottom, top, left and right.
We do not give precise definitions, but mention that the left and right sections start and end with vertical edges 
(see Figure~\ref{fig:R_Cyl} for an illustration). 

For~$n,m \in \bbN$, let~$\Rect_{m,n}$ be the rectangle~$[-m,m - \frac12] \times [0,n]$. 
Write~$\rp\rp/\rm\rm$ for the boundary conditions on~$\Rect_{m,n}$ 
where all faces adjacent to the bottom of~$\partial \Rect_{m,n}$ have red spin~$\rm$ 
and all other faces adjacent to ~$\partial \Rect_{m,n}$ have red spin~$\rp$. 
As for other boundary conditions, these may be defined only on~$\Rect_{m,n}$, with no reference to the outside faces. 
One may however check that, since there are only two arcs of different sign on the boundary, these boundary conditions do satisfy the Spatial Markov property. 

We will also consider the cylinder~$\Cyl_{m,n}$ obtained by identifying the left and right boundaries of~$\Rect_{m,n}$. 
Write~$\rp\rp/\rm\rm$ for the boundary conditions on~$\Cyl_{m,n}$ which are double-$\rp$ on the bottom and double-$\rm$ on the top. 
That is~$\mu_{\Cyl_{m,n}}^{\srp\srp/\srm\srm}$ is the uniform measure 
on pairs of coherent spin configurations~$(\sigma_b,\sigma_r)$ on~$\Cyl_{m,n}$
with the property that all faces adjacent to the top boundary of~$\Cyl_{m,n}$ have~$\sigma_r = \rp$
and all those adjacent to the bottom have~$\sigma_r = \rm$. 
No restriction on the blue spins of the boundary faces is imposed. 

It is immediate that the Spatial Markov property applies to~$\mu_{\Cyl_{m,n}}^{\srp\srp/\srm\srm}$ in the same way as for planar domains. 
In particular, ~$\mu_{\Rect_{m,n}}^{\srp\srp/\srm\srm}$ is related to~$\mu_{\Cyl_{m,n}}^{\srp\srp/\srm\srm}$ by the following. 

\begin{lem}\label{lem:cyl_to_rect}
	Fix~$m,n \geq 1$ and let~$\Linev$ be the right boundary of~$\Rect_{m,n}$. Then~$\Linev$ is also an edge-path of~$\Cyl_{m,n}$, and
	\begin{align*}
		\mu_{\Rect_{m,n}}^{\srp\srp/\srm\srm} 
		= \mu_{\Cyl_{m,n}}^{\srp\srp/\srm\srm}(.\,|\, \sigma_r \equiv \rp \text{ on faces adjacent to~$\Linev$}).
	\end{align*}
\end{lem}

\begin{figure}
    \begin{center}
    \includegraphics[width = 0.5\textwidth]{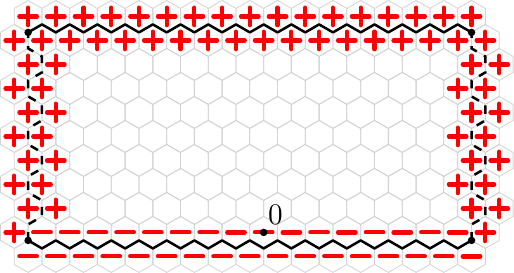}\qquad\qquad
    \includegraphics[width = 0.35\textwidth]{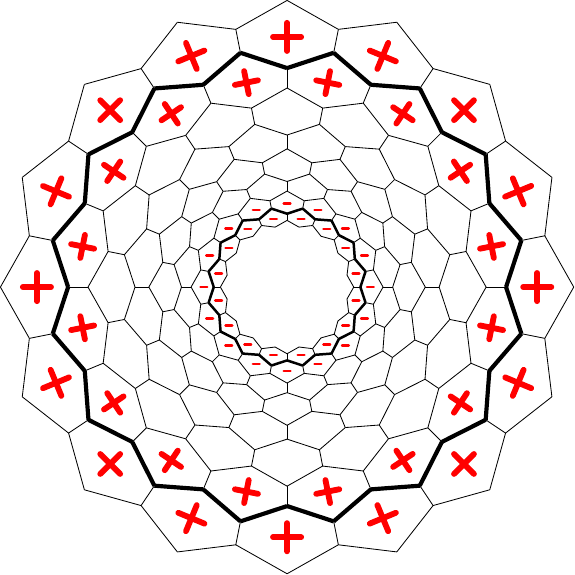}
    \caption{\textbf{Left:} The rectangle~$\Rect_{m,n}$ with boundary conditions~$\frp\frp/\frm\frm$.
    The dashed lines are the left and right boundaries; in~$\Cyl_{m,n}$ they are identified to each other.
    \textbf{Right:} The cylinder~$\Cyl_{m,n}$ may be drawn in the plane as depicted; its top and bottom are marked by bold lines. 
    The resulting graph is trivalent, all its faces except the interior one are hexagons, and their union is simply connected. 
    The FKG lattice condition may be proved for this graph in the same way as for domains of the hexagonal lattice. 
    The measure~$\mu_{\Cyl_{m,n}}^{\sfrp\sfrp/\sfrm\sfrm}$ 
    is obtained as a conditioning of the measure on the depicted graph:
    all faces adjacent to the bottom have red spin~$\frm$, 
    while all those adjacent to the top have red spin~$\frp$. }
    \label{fig:R_Cyl}
    \end{center}
\end{figure}

While~$\Cyl_{m,n}$ is not a planar domain, the FKG inequality applies to it. 

\begin{lem}
	For any~$m,n \geq 1$, the red-spin marginal of~$\mu_{\Cyl_{m,n}}^{\srp\srp/\srm\srm}$ 
	satisfies the FKG lattice condition and is positively associated. 
\end{lem}

\begin{proof}
	We will not give a full proof of this, only a sketch. 
	Notice that~$\Cyl_{m,n} \cup \partial_\out \Cyl_{m,n}$, may be embedded in the plane and rendered simply connected 
	by adding a face of degree~$4m$ below the bottom of~$\partial_\out \Cyl_{m,n}$, as drawn in Figure~\ref{fig:R_Cyl}, right diagram. 
	Write~$\calD$ for the planar graph obtained by this procedure. 
	Then a straightforward adaptation of Theorem~\ref{thm:FKG} (see also remark~\ref{rem:FKG} \textit{(ii)}) 
	shows that the FKG lattice condition also holds for~$\mu_\calD$. 

	As explained in Corollary~\ref{cor:fkg}, conditioning on the value of red spins on a given set 
	conserves the FKG lattice condition for the red spin marginal. 
	In particular, the red-spin marginal of~$\mu_\calD$ 
	conditioned on the event that all faces adjacent to the top of~$\Cyl_{m,n}$ have~$\sigma_r \equiv \rp$
	while all those adjacent to the bottom have ~$\sigma_r \equiv \rm$
	also satisfies the FKG lattice condition. 
	Finally, the Spatial Markov property states that the conditional measure above 
	is identical to~$\mu_{\Cyl_{m,n}}^{\srp\srp/\srm\srm}$ (when restricted to~$\Cyl_{m,n}$). 
\end{proof}

\subsection{Strong RSW theory}

As promised in Section~\ref{sec:RSW_Vincent}, we will now prove a stronger RSW result. 
Variations of it may be envisioned; we will state it in the form most useful to us. 
We start with a general lemma that allows to lengthen crossings of long rectangles. 
The main result of the section (Proposition~\ref{prop:RSW_Hugo}) is given afterwards. 
It will be stated and proved for the cylinder, but may also be deduced in other settings. 

\subsubsection{Lengthening crossings}

Recall the definition of~$\Par_{m,n}$ and its boundary segments~$\Top$ and~$\Bottom$ from Section~\ref{sec:RSW_weak}.

\begin{lem}\label{lem:double_cross}
	Let~$\calD$ be a domain and~$n$ be such that 
	$\Par_{3n,n} \cup \partial_\out\Par_{3n,n} \subset \calD$.
	Fix some red spin configuration~$\zeta$ on~$\calD$ with the property that 
	all faces of~$\Bottom(3n,n) \cup\Top(3n,n)$ are awarded spins~$\rp$. 
	Then 
	\begin{align}\label{eq:double_cross}
		\mu_{\calD}[ \calC^v_{\srp\srp}(3n,n) \, |\, \sigma_r = \zeta \text{ outside~$\Int(\Par_{3n,n})$}] \geq 1/36.
	\end{align}
\end{lem}

The previous lemma may be used to glue crossings of long rectangles as described below.

\begin{cor}\label{cor:lengthen_crossings}
 	Let~$\calD$ be a domain such that~$\Par_{5n,n}\subset \calD$.
	Then 
	\begin{align}\label{eq:lengthen_crossings}
		\mu_{\calD}( \calC^h_{\srp\srp}(5n,n)) \geq
		\tfrac{1}{288} \mu_{\calD}( \calC^h_{\srp\srp}(4n,n))\,
		\mu_{\calD}\big(\calC^h_{\srp\srp}\big[(n,0) + \Par_{4n,n}\big]\big),
	\end{align}
	where~$(n,0) + \Par_{4n,n}$ is the translate of~$\Par_{4n,n}$ by~$n$ units to the right.
\end{cor}

Due to the Spatial Markov property, the statements of Lemma~\ref{lem:double_cross} and Corollary~\ref{cor:lengthen_crossings} 
also apply to measures with boundary conditions such as~$\mu_\calD^{\srp\srp}$. 

\begin{figure}
	\begin{center}\hspace{-0.1\textwidth}
	\includegraphics[width = 0.57\textwidth, page = 1]{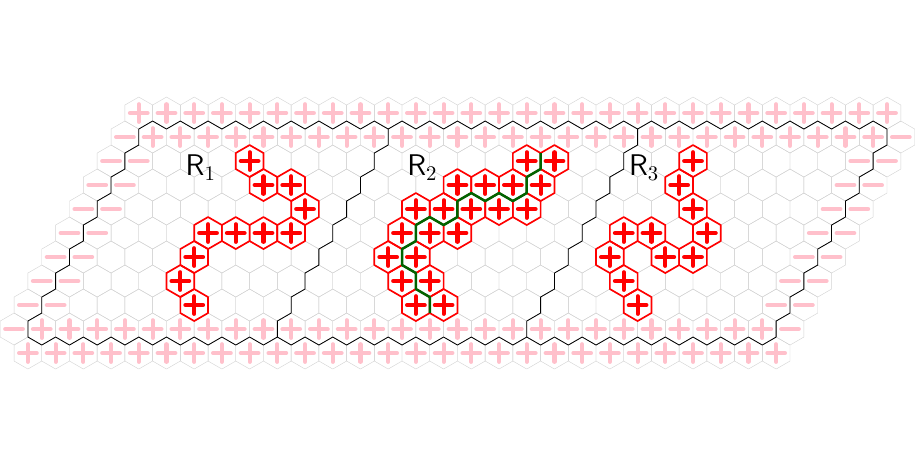}\hspace{-0.06\textwidth}
	\includegraphics[width = 0.57\textwidth, page = 2]{rhombi2.pdf}\hspace{-0.2\textwidth}
  	\caption{{\em Left:}
		The parallelogram ~$\Par_{3n,n}$ contains three disjoint translations~$\R_1$,~$\R_2$ and~$\R_3$ of ~$\Par_{3n,n}$. 
		Given the outside configuration~$\zeta$,~$\R_1$ and~$\R_3$ 
		contain simple-$\frp$ vertical crossings with probability at least~$1/2$. 
		Conditionally on the existence of such crossings,~$\R_2$ contains a vertical double-$\frp$
		crossing with probability at least~$1/4$. 
		{\em Right:} The latter statement is proved by working in the symmetric domain~$\sfD$, 
		which is crossed vertically by a double-$\frp$ path with probability~$1/4$ at least.}
		\label{fig:par}
	\end{center}
\end{figure}

\begin{proof}[Lemma~\ref{lem:double_cross}]
	By the FKG inequality, the LHS of \eqref{eq:double_cross} is minimal when~$\zeta \equiv \rm$ on 
	$\calD \setminus (\Bottom(3n,n) \cup\Top(3n,n))$. We may assume this below. 

	Observe that~$\Par_{3n,n}$ may be partitioned into three translations of~$\Par_{n,n}$; 
	call them~$\R_1$,~$\R_2$ and~$\R_3$ ordered from left to right (see Figure~\ref{fig:par}).
	Due to Lemma~\ref{lem:simple_self_duality} and the FKG inequality, 
	\begin{align}\label{eq:simple_paths}
		\mu_{\calD}[ \calC_{\srp}^v(\R_1) \cap \calC_{\srp}^v(\R_3)  \, |\, \sigma_r = \zeta \text{ outside~$\Int(\Par_{3n,n})$} ] 
		\geq \tfrac19.
	\end{align}
	
	If~$\calC_{\srp}^v(\R_1)$ occurs, let~$\Gamma_L$ be the left-most simple-$\rp$ vertical crossing of~$\R_1$. 
	Formally, let~$\Gamma_L$ be the edge-path running along the left side of the spin-crossing. 
	Also, let~$\Gamma_R$ be the right-most simple-$\rp$ vertical crossing of~$\R_3$ when~$\calC_{\srp}^v(\R_3)$ occurs.
	
	Fix two possible realisations~$\gamma_L,\gamma_R$ of~$\Gamma_L,\Gamma_R$.
	Let~$\sfD_0$ be the domain formed of the faces of~$\Rect_{3n,n}$ between~$\gamma_L$ and~$\gamma_R$. 
	The events~$ \Gamma_L = \gamma_L$ and~$\Gamma_R =\gamma_R$ are both measurable in terms of the configuration outside~$\sfD_0$. 

	Consider the line~$\ell$ running through the bottom left and upper right corner of~$\R_2$ and let~$\rho$ be the orthogonal symmetry with respect to~$\ell$; 
	note that~$\bbH$ is invariant under~$\rho$.
	Let~$\sfD$ be the domain formed of the faces of~$\sfD_0$ and those of~$\rho(\sfD_0)$. 
	
	Let~$a,b,c,d$ be the corners of~$\R_2$ ordered in counter-clockwise order, starting from the top-left corner. 
	Write~$(bc)$ and~$(da)$ for the arcs of~$\partial \sfD$ in counter-clockwise order.
	Then, 
	due to the monotonicity in boundary conditions and the FKG inequality,
	\begin{align*}
		&\mu_{\calD}[\calC^v_{\srp\srp}(\R_2)\, | \, \Gamma_L = \gamma_L\text{ and } \Gamma_R =\gamma_R 
		\text{ and } \sigma_r = \zeta \text{ outside~$\Int(\Par_{3n,n})$}] 
		\\
		&\qquad = \mu_{\sfD}^{\srp\srm}[\calC^v_{\srp\srp}(\R_2)\,|\, \sigma_r \equiv \rp \text{ on } \Bottom(3n,n) \cup\Top(3n,n)]
		\geq \mu_{\sfD}^{\srp\srm}[(bc)\xlra{\srp\srp}(ad)] 
		\geq 1/4.
	\end{align*}
	The equality is due to the specific to the boundary conditions induced by~$\Gamma_L$,~$\Gamma_R$ and~$\zeta$ on~$\sfD_0$
	(a brief analysis is needed to ensure that the is no multiplicative constant appearing between the two sides). 
	The first inequality is due to the FKG inequality and the inclusion of events; the last one is due to Lemma~\ref{lem:square_crossed}. 
	Averaging the above over all possible values of~$ \Gamma_L~$ and~$\Gamma_R$ and using~\eqref{eq:simple_paths}, we obtain the desired bound. 
\end{proof}

\begin{proof}[Corollary~\ref{cor:lengthen_crossings}]
	Let~$\Gamma_L^+$ and~$\Gamma_L^-$ be the top and bottom most, respectively, double-$\rp$ horizontal crossings of~$\Par_{4n,n}$. 
	Define~$\Gamma_R^+$ and~$\Gamma_R^-$ similarly for the rectangle~$(n,0) + \Par_{4n,n}$. 
	When both~$\Par_{4n,n}$ and~$(n,0) + \Par_{4n,n}$ 
	are crossed horizontally by double-$\rp$ paths,
	then either~$\Gamma_L^+$ intersects or is higher than~$\Gamma_R^-$ inside the middle parallelogram~$(0,n)+\Par_{3n,n}$,
	or~$\Gamma_L^-$ intersects or is lower than~$\Gamma_R^+$.
	As a consequence, 
	\begin{align}\label{eq:aaa}
		&\mu_{\calD}(\text{$\Gamma_L^+$ intersects or higher than~$\Gamma_R^-$}) 
		+\mu_{\calD}(\text{$\Gamma_L^-$ intersects or lower than~$\Gamma_R^+$}) \\
		&\quad\geq
		\mu_{\calD}\big( \calC^h_{\srp\srp}(4n,n)\cap \calC^h_{\srp\srp}\big[(n,0) + \Par_{4n,n}\big]\big) 
		\geq 
		\mu_{\calD}\big( \calC^h_{\srp\srp}(4n,n)\big)\,\mu_{\calD}\big(\calC^h_{\srp\srp}\big[(n,0) + \Par_{4n,n}\big]\big).\nonumber
	\end{align}
	In the last inequality we used the FKG property. 
	We focus next on the first term in the LHS above.
	
	For any realisation of~$\Gamma_L^+$ and~$\Gamma_R^-$ with the former intersecting or higher than the latter, 
	if~$\Gamma_L^+$ and~$\Gamma_R^-$ intersect or if they are connected to each other by a double-$\rp$ path, 
	then~$\mu_{\calD}(\calC^h_{\srp\srp}(5n,n))$ occurs. 
	Below we will show that, conditionally on~$\Gamma_L^+$ and~$\Gamma_R^-$, 
	the two paths intersect or are connected by a double-$\rp$ path with positive probability. 
	The case where the paths intersect is trivial; we assume henceforth that~$\Gamma_L^+$ and~$\Gamma_R^-$ are disjoint. 
	Notice that~$\Gamma_L^+$ is measurable in terms of the spins of the faces of~$\Par_{4n,n}$ above it,  
	and~$\Gamma_R^-$ is measurable in terms of the spins of the faces of~$(n,0) + \Par_{4n,n}$ below it.
	Let~$\sfU$ be the set of all faces of~$\Par_{5n,n}$ which are in neither of the two categories above. 
	
	Let~$a$ be the right endpoint of~$\Gamma_L^+$ and~$c$ be the left endpoint of ~$\Gamma_R^-$.
	Orient~$\Gamma_L^+$ from left to right and~$\Gamma_R^-$ from right to left.
	Let~$b$ be the last point of intersection of~$\Gamma_L^+$ with the left side of ~$(n,0) + \Par_{3n,n}$,
	and~$d$ be the last point of intersection of~$\Gamma_R^-$ with the right side of ~$(n,0) + \Par_{3n,n}$.
	Write~$\sfD$ for the domain contained in~$(n,0) + \Par_{3n,n}$, delimited at the top by the section of~$\Gamma_L^+$ between~$b$ and~$a$, 
	and at the bottom by the section of~$\Gamma_R^-$ between~$d$ and~$c$. 
	
	Let~$H$ be the event that there exists an edge-path~$\gamma$ contained in~$\sfD$, 
	connecting the arcs~$(ab)$ and~$(cd)$ of~$\partial_E\sfD$, such that all faces of~$\sfU \cap \sfD$ adjacent to it are of red spin~$\rp$. 
	Then, by the FKG inequality and the properties \textit{(i)} and \textit{(ii)} of Lemma~\ref{lem:4-arcs}, 
	\begin{align*}
		\mu_{\calD}(H \,|\, \Gamma_L^+,\Gamma_R^- )
		\geq \tfrac18\mu_{\sfD}^{a\srp b\srm c\srp d\srm}(H)	
		\geq \tfrac18\mu_{\sfD}^{a\srp b\srm c\srp d\srm}((ab)\xlra{\srp\srp \text{ in } \sfD}(cd)).
	\end{align*}
	The second inequality is due to the inclusion between the two events. 
	Notice also that when~$H$ occurs,~$\Gamma_L^+,\Gamma_R^-~$ are necessarily connected by a double-$\rp$ path.
	
	Now, applying again Lemma~\ref{lem:4-arcs}, we conclude that 
	\begin{align*}
		\mu_{\sfD}^{a\srp b\srm c\srp d\srm}((ab)\xlra{\srp\srp \text{ in } \sfD}(cd))
		\geq \mu_{\Par_{3n,n}}^{a'\srp b'\srm c'\srp d'\srm}[ \calC^v_{\srp\srp}(3n,n)],
	\end{align*}
	where~$a',b',c'$ and~$d'$ are the corners of the parallelogram, ordered in counter-clockwise order, starting from the top left. 
	The RHS of the above is bounded below by~$1/36$, as proved in Lemma~\ref{lem:double_cross}.
	Combining the last two displayed equations, we find 
	\begin{align*}
		\mu_{\calD}(\Gamma_L^+\xlra{\srp\srp \text{ in } \sfU}\Gamma_R^- \,|\, \Gamma_L^+,\Gamma_R^- )
		\geq \tfrac1{288}.
	\end{align*}
	Averaging over all values of~$\Gamma_L^+$ and~$\Gamma_R^-$ as above, we find 
	\begin{align*}
		\mu_{\calD}( \calC^h_{\srp\srp}(5n,n)) \geq	 \tfrac1{288} \mu_{\calD}(\text{$\Gamma_L^+$ intersects or higher than~$\Gamma_R^-$}).
	\end{align*}
	The same bound holds for the second term in \eqref{eq:aaa}, and~\eqref{eq:lengthen_crossings} follows. 
	
	\begin{figure}
	\begin{center}
	\includegraphics[width = 0.48\textwidth, page = 1]{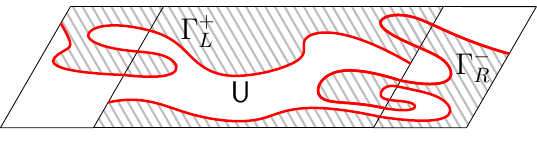}\quad
	\includegraphics[width = 0.48\textwidth, page = 2]{twoGammas2.pdf}\vspace{15pt}\\
	\end{center}
	\caption{{\em Left:} the paths~$\Gamma_L^+$ and~$\Gamma_R^-$ are measurable in terms of the hashed regions; 
	its complement is ~$\sfU$.
	{\em Right:} The domain~$\sfD$ is delimited by parts of~$\Gamma_L^+$ and~$\Gamma_R^-$. 
	Any crossing between the arcs~$(ab)$ and~$(cd)$ in~$\sfD$ induces a connection between~$\Gamma_L^+$ and~$\Gamma_R^-$.}
	\label{fig:twoGammas}
	\end{figure}
\end{proof}

\subsubsection{Statement of the strong RSW}

The notions of simple and double horizontal and vertical crossings adapt readily to rectangular domains~$[a,b]\times[c,d]$. 
Use the notations~$\calC^h_{.}([a,b]\times[c,d])$ and~$\calC^v_{.}([a,b]\times[c,d])$ for the existence of such crossings. 

\begin{prop}[strong RSW]\label{prop:RSW_Hugo}
	There exists a function~$\psi : (0,1] \to (0,1]$ such that, 
	for all~$N,M,n,k \geq 1$ with~$n+k \leq N$ and~$2n < M$,
	\begin{align}\label{eq:RSW_Hugo}
		\mu_{\Cyl_{M,N}}^{\srm\srm/\srp\srp}  \big[ \calC^h_{\srp\srp}([-2n,2n] \times [k,k+n])\big] 
		\geq \psi\Big(\mu_{\Cyl_{M,N}}^{\srm\srm/\srp\srp}  \big[ \calC_{\srp\srp}^v([-3n,3n] \times [k,k+n])\big] \Big).
	\end{align} 
\end{prop}

In other words, the above tells us that if wide rectangles are crossed vertically with positive probability (that is in the easy direction), 
then they are also crossed horizontally (\emph{i.e.} in the hard direction) with positive probability. 
This is a typical RSW result in that it relates probabilities of crossings in the easy direction to those of crossings in the hard direction. 
What is remarkable is that the measure to which it applies, namely~$\mu_{\Cyl_{M,N}}^{\srm\srm/\srp\srp}$, is not rotationally invariant. 
Thus, vertical crossings are ``orthogonal'' to horizontal ones. 
The exact aspect ratio of the two rectangles ($4$ and~$6$, respectively) is not essential; they have been chosen to simplify statements later on.

The proof of Proposition~\ref{prop:RSW_Hugo} is quite intricate. 
It is based on similar results from \cite{DumRaoTas18}, but with additional difficulties due to the two layers of spins necessary for the Spatial Markov property (see Theorem~\ref{thm:DMP}). 
A very similar version appears in \cite{DumHarLasRauRay18}.
The next section is dedicated to proving  Proposition~\ref{prop:RSW_Hugo}.

\subsubsection{Proof of Proposition~\ref{prop:RSW_Hugo}}

The structure is reminiscent of a proof by contradiction: 
assuming that~\eqref{eq:RSW_Hugo} fails 
(more precisely that~$\calC^h_{\srp\srp}([-2n,2n] \times [k,k+n])$ has probability below a certain threshold),
we prove that double-$\rp$ vertical crossings of~$[-3n,3n] \times [k,k+n]$ have some particular behaviour. 
This is done in a series of lemmas (Lemmas~\ref{lem:centre},~\ref{lem:wig},~\ref{lem:loc} and~\ref{lem:wb})
-- the title of each lemma indicates a constraint that vertical crossings need to satisfy {\em for~\eqref{eq:RSW_Hugo} to fail}. 
Then we prove that two typical vertical crossings of~$[-3n,3n] \times [k,k+n]$ with starting points sufficiently close to each other 
are connected with positive probability by a double-$\rp$ path (see Lemma~\ref{lem:connect}).
Proposition~\ref{prop:RSW_Hugo} follows from this last statement. 
Lemma~\ref{lem:connect} is the heart of the proof; it relies on the construction of a symmetric domain, 
similarly to what was done for Lemma~\ref{lem:double_cross}. \smallskip

Fix~$n,k,M,N$ as in the proposition. 
We will work with~$n$ large; the function~$\psi$ in the proposition may be adjusted to incorporate all small values of~$n$.
For ease of writing, translate the cylinder~$\Cyl_{M,N}$ vertically by~$-k$; 
write~$\mu$ for the measure with boundary conditions~$\rm\rm/\rp\rp$ on this translated cylinder.

Using this notation, our goal is to prove that, for any~$C > 0$ there exists~$\Delta > 0$ depending only on~$C$, 
not on~$n$,~$k$,~$M$ or~$N$, such that
\begin{align*}
	\mu[\calC_{\srp\srp}^v(\Rect_{3n,n})] > C \Rightarrow			
	\mu[\calC_{\srp\srp}^h(\Rect_{2n,n})] > \Delta.
\end{align*}
	
For~$m,h\geq 0$, write~$\Linev(m)$ for the right boundary of the rectangle~$\Rect_{m,N}$
and~$\Line(h)$ for the top boundary of~$\Rect_{M,h}$.
That is~$\Line(h)$ is approximately a horizontal line at height~$h$;~$\Linev(m)$ is a vertical line~$m$ units to the right of~$0$. 
Let~$\Strip_h$ be the set of faces of the cylinder contained between~$\Line(0)$ and~$\Line(h)$.
Define~$\Mid_h(m)$ as the segment of~$\Line(h)$ contained in~$[- m,m] \times \bbR$.
See Figure~\ref{fig:wig} for illustrations. 

Below we will talk about double-$\rp$ paths contained in~$\Strip_h$ with endpoints in~$\Mid_0(m)$ and~$\Mid_h(m)$, respectively. 
While not explicitly stating it each time, we will always ask that such a path have trivial winding around the cylinder. 

\begin{lem}[Endpoints of paths are centred]\label{lem:centre}
	For any~$\eps > 0$ and~$C_{\centre} >0$ there exists~$\Delta_{\centre} =\Delta_{\centre}(\eps,C_{\centre}) >0$ such that the following holds. 
	Fix~$m \leq n$ and write~$\calH_\centre(m)$ for the event that there exists a double-$\rp$ path in~$\Strip_m$ with one endpoint 
	in~$\Mid_0(m \eps)$ and one in~$\Line(m)\setminus \Mid_{m}(2 m\eps)$. 
	Then 
	\begin{align}
		\mu[\calH_\centre(m)] > C_{\centre} \Rightarrow			
		\mu[\calC_{\srp\srp}^h(\Rect_{2m,m})] > \Delta_{\centre}.
	\end{align}
\end{lem}

\begin{proof}
	Fix~$C_\centre$ and suppose~$\mu[\calH_\centre(m)] > C_{\centre}$. 
	Then, by symmetry, with probability at least~$C_\centre /2$ there exists a double-$\rp$ path in~$\Strip_m$ with lower endpoint in 
	$\Mid_0(m\eps)$ and upper end-point on~$\Line(m)$, to the right of~$\Linev(2\eps m)$. 
	By horizontal symmetry and translation invariance, the event that there exists a double-$\rp$ path in~$\Strip_m$ 
	with lower endpoint in~$(4\eps m,0) + \Mid_0(m\eps)$ and upper end-point on~$\Line(m)$, to the left of~$\Linev(2\eps m)$
	also has probability greater than~$C_\centre/2$. Notice that when these two events occur, then the segments~$\Mid_0(m\eps)$
	and~$(4\eps m,0) + \Mid_0(m\eps)$ are connected by a double-$\rp$ path contained in~$\Strip_m$. 
	Using the FKG inequality and the above considerations, we find
	\begin{align}\label{eq:arc}
		\mu\big[\Mid_0(m\eps)\xlra{\srp\srp \text{ in~$\Strip_m$}}(4\eps m,0) + \Mid_0(m\eps)\big] \geq \tfrac{1}4 \, C_\centre^2. 
	\end{align}
	
	For~$j \in \bbZ$, let~$M_j =  (2j\eps m,0) + \Mid_0(m\eps)$. Using again the invariance of~$\mu$ under horizontal shift, 
	we find that 
	\begin{align*}
		\mu\big(M_{j-1}\xlra{\srp\srp \text{ in~$\Strip_m$}}M_{j+1}\big) \geq \tfrac{1}4 \, C_\centre^2 
		\qquad \forall -\tfrac{1}\eps \leq j \leq \tfrac1\eps.
	\end{align*}
	Finally, if all the events above occur simultaneously, then~$\calC_{\srp\srp}^h(\Rect_{2m,m})$ also occurs. By the FKG inequality, we find 
	\begin{align*}
		\mu[\calC_{\srp\srp}^h(\Rect_{2m,m})] 
		\geq
		\mu\Big[ \bigcap_{ j =-1/\eps }^{1/\eps} \big\{M_{j-1}\xlra{\srp\srp \text{ in~$\Strip_m$}}M_{j+1}\big\}\Big]
		\geq \big(C_\centre/4\big)^{1 + 2/\eps  } =: \Delta_\centre.
	\end{align*}
\end{proof}

%

\begin{lem}[Vertical paths wiggle]\label{lem:wig}
	There exist explicit constants~$0 < \rho_\int < \rho_\out$ with the following property.
	For~$m \leq n$ define the events 
	\begin{itemize}
	\item~$\calH_\wig^\int(m)$ is the event that there exists a double-$\rp$ path contained in~$\Strip_m$, 
	from~$\Mid_{0}(\rho_\int m)$ to~$\Mid_{m}(\rho_\int m)$ that
	does not cross~$\Linev(\rho_\int m)$ or that does not cross~$\Linev(-\rho_\int m)$;
	\item~$\calH_\wig^\out(m)$ is the event that there exists a double-$\rp$ path contained in~$\Strip_m$, 
	from~$\Mid_0(\rho_\int m)$ to~$\Mid_m(\rho_\int m)$ that is not contained in~$\Rect(\rho_\out m , m)$.
	\end{itemize} 
	Then, for any~$C_{\wig} >0$ there exists~$\Delta_{\wig} >0$ such that, for any~$m \leq n$, 
	\begin{align}\label{eq:wig}
		\mu[\calH_\wig^\int(m)\cup \calH_\wig^\out(m)] > C_{\wig} \Rightarrow			
		\mu[\calC_{\srp\srp}^h(\Rect_{2m,m})] > \Delta_{\wig}.
	\end{align}
\end{lem}

In other words, the lemma tells us that, 
if vertical crossings of~$\Strip_m$ starting on~$\Mid_0(\rho_\int m)$ and ending on~$\Mid_h(\rho_\int m)$
wiggle either too little (when~$\calH_\wig^\int(m)$ occurs)  or too much (when~$\calH_\wig^\out(m)$ occurs), 
then wide rectangles may be crossed horizontally.

\begin{proof}
	Take~$\rho_\int = 1/72$ and~$\rho_\out = 4 + \rho_\int$. Fix some constant~$C_\wig$; 
	the value of~$\Delta_{\wig}$ will be determined below. 
	If~$\mu[\calH_\wig^\int(m)  \cup \calH_\wig^\out(m) ] > C_{\wig}$, 
	then either~$\mu[\calH_\wig^\int(m) ] > C_{\wig}/2$ or~$\mu[\calH_\wig^\out(m)] > C_{\wig}/2$. 
	
	Suppose the second inequality is valid. 
	If the event~$\calH_\wig^\out$ occurs, 
	one of the rectangles~$[\rho_\int,\rho_\out m]\times [0,m]$ 
	or~$[-\rho_\out m,-\rho_\int]\times [0,m]$ is crossed horizontally by a double-$\rp$ path. 
	Thus, due to the invariance of~$\mu$ under horizontal shift,~$\mu[\calC_{\srp\srp}^h(\Rect_{2m,m})] > C_{\wig}/4$.
	The implication is therefore proved in this case for any~$\Delta_\wig \leq  C_{\wig}/4$.
	
	It remains to consider the case when~$\mu[\calH_\wig^\int(m) ] > C_{\wig}/2$. 
	Then, by symmetry, the probability that there exists a double-$\rp$ path
	from~$\Mid_{0}(\rho_\int m)$ to~$\Mid_{h}(\rho_\int m)$  contained in~$\Strip_m$ 
	and to the right of~$\Linev(-\rho_\int m)$ is greater than~$C_{\wig}/4$.
	Moreover, due to the invariance of~$\mu$ under horizontal translation and symmetry, with probability at least~$C_\wig / 4$,
	the segment~$(2\rho_\int m,0 )+ \Mid_{0}(\rho_\int m)$ may be connected to~$(2\rho_\int m,0 ) + \Mid_{h}(\rho_\int m)$ 
	using a double-$\rp$ path contained in~$\Strip_m$ that stays to the left of~$\Linev(3\rho_\int m)$.
	The two events described above are increasing, and due to the FKG inequality, 
	they occur simultaneously with probability at least~$(C_\wig / 4)^2$.
	When both do occur, then the rectangle~$[-\rho_\int m,3\rho_\int m] \times [0,m]$ 
	is crossed vertically by a path of double-$\rp$. See Figure~\ref{fig:wig} - left diagram.
	
	\begin{figure}
	\begin{center}
	\includegraphics[scale = 1.2]{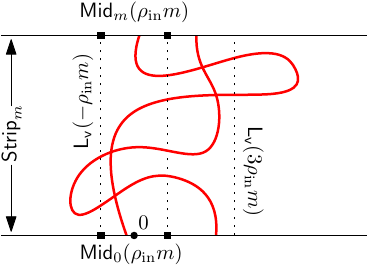}\qquad
	\includegraphics[scale = 1.2]{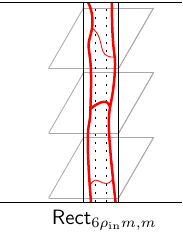}	
	\caption{{\em Left:} Using two symmetric instances of~$\calH_\wig^\int(m)$ we create a vertical crossing of the thin rectangle~$[-\rho_\int m,3\rho_\int m] \times [0,m]$.
	{\em Right:} When both 	$[-6\rho_\int,-2\rho_\int m]\times [0,m]$ and~$[2\rho_\int,6\rho_\int m]\times [0,m]$
	are crossed vertically by double-$\frp$ paths, then these paths may be connected as in Lemma~\ref{lem:double_cross}: 
	first by simple-$\frp$ paths (thin red lines), then by a double-$\frp$ path (bold red line).}
	\label{fig:wig}
	\end{center}
	\end{figure}
	
	Let~$\calA$ be the event that both rectangles~$[-6\rho_\int,-2\rho_\int m]\times [0,m]$
	and~$[2\rho_\int,6\rho_\int m]\times [0,m]$ contain double-$\rp$ vertical crossings. 
	Due to the FKG inequality, invariance under horizontal translation, and the above estimate, we find that~$\mu(\calA) \geq (C_{\wig}/4)^4$.
	
	A rotated version of Lemma~\ref{lem:double_cross} (and Corollary \ref{cor:lengthen_crossings}) 
	applies to the rectangle~$\Rect_{6\rho_\int m,m}$,
	and proves that the left-most vertical crossing of 	$[-6\rho_\int,-2\rho_\int m]\times [0,m]$
	connects via a double-$\rp$ path to the right-most vertical crossing of~$[2\rho_\int,6\rho_\int m]\times [0,m]$
	with probability at least~$1/36$. 
	(The choice of~$\rho_\int$ ensures that this rectangle is sufficiently thin 
	to be covered by three disjoint translates of~$\Par_{m/3,m/3}$ placed one on top of the other. 
	Slight adaptations to the proof need to be made; we leave this to the reader. See also Figure~\ref{fig:wig} - right diagram.)  
	Using this and the lower bound on the probability of~$\calA$, we find 
	\begin{align*}
		\mu\big[(-4\rho_\int m,0) +  \Mid_0(2\rho_\int m)\xlra{\srp\srp \text{ in~$\Strip_m$}}
		(4\rho_\int m,0) + \Mid_0(2\rho_\int m)\big] 
		\geq \tfrac{1}{ 36} \, (C_{\wig}/4)^4.
	\end{align*}
	The above is akin to~\eqref{eq:arc}. We conclude as in the proof of Lemma~\ref{lem:centre} that
	\begin{align}
		\mu[\calC_{\srp\srp}^h(\Rect_{2m,m})] > \Delta_{\wig},
	\end{align}
	for some sufficiently small constant~$\Delta_{\wig}>0$ depending only on~$C_{\wig}$ and~$\rho_\int$. 
\end{proof}

The last two lemmas will be used for different scales~$m \leq n$. 
The following will only be used at scale~$n$. To simplify notation, we only state it at this scale. 

\begin{lem}[Vertical paths have fixed width]\label{lem:loc}
	Let~$c_\loc = 5$ and fix any~$\eps \leq \rho_\int$. 
	For~$k \geq 1$, let  
	$\calG_{\loc}(\eps, k)$ be the event that any double-$\rp$ path contained in~$\Strip_n$ 
	with endpoints in~$\Mid_0(\eps n)$ and~$\Mid_n(\eps n)$ 
	intersects the vertical line 
	$\Linev(k\eps n)$ but not~$\Linev[(k+c_\loc)\eps n]$. 
	Then, for any constant~$ C_{\loc} > 0$ 
	there exists~$\Delta_\loc=\Delta_\loc(\eps,C_\loc) > 0$ such that
	\begin{align*}
		\Big(\mu [ \calG_{\loc}(\eps, k) ] < 1 - C_\loc \,\,\,\text{ for all } k \geq \lfloor \rho_\int / \eps\rfloor\Big)
		\Rightarrow \mu[\calC_{\srp\srp}^h(\Rect_{2n,n})] > \Delta_{\loc}.
	\end{align*}
\end{lem}

Notice that~$\calG_{\loc}(\eps, k)$ contains all configurations with no double-$\rp$ 
path connecting~$\Mid_0(\eps n)$ to~$\Mid_n(\eps n)$ inside~$\Strip_n$. 
Indeed, the condition is trivially satisfied.

\begin{proof}
	Fix~$\eps \leq \rho_\int$ and~$C_{\loc} > 0$; 
	to simplify notation we will consider ~$\rho_\int/\eps$ to be an integer. 
	Assume that~$\mu [ \calG_{\loc}(\eps, k) ] < 1 - C_\loc$ for all~$k  \geq {\rho_\int}/{\eps}$.
	
	

	For~$k \geq 1$, let~$\calE_k$ be the event that there exists a double-$\rp$ path in~$\Strip_n$ 
	with endpoints in~$\Mid_0(\eps n)$ and~$\Mid_n(\eps n)$ and which intersects~$\Linev(k\eps n)$.
	Also write~$\calE_0$ for the event that~$\Mid_0(\eps n)$ and~$\Mid_n(\eps n)$ are connected by a double-$\rp$ path contained in~$\Strip_n$,
	with no other restriction.
	The events~$\calE_k$ are increasing each, but form a decreasing sequence.
	Additionally, define~$\calA_k$ as the event that there exists a double-$\rp$ path in~$\Strip_n$ 
	that connects~$\Mid_0(\eps n)$ to~$\Mid_n(\eps n)$ and which does not intersect~$\Linev(k\eps n)$.
	Then each~$\calA_k$ is an increasing event and the sequence~$\calA_k$ is increasing in~$k$.
	Moreover~$\calG_{\loc}(\eps, k) = (\calA_k \cup \calE_{k+5})^c$.

	Notice that~$\calE_0^c$ is contained in all events ~$\calG_{\loc}(\eps, k)$, hence~$\mu(\calE_0) \geq C_\loc$.
	Set~$k$ to be the smallest index such that~$\mu(\calA_k) \geq C_\loc/2$.
	The existence of~$k$ is guaranteed by the fact that~$\lim_j \mu(\calA_j) = \mu(\calE_0) \geq C_\loc$.
	
	Suppose first that~$k \leq {\rho_\int}/{\eps}$. 
	Then,~$\calH_\wig^\int(n) \geq \mu(\calA_k) \geq C_\loc/2$, and Lemma \ref{lem:wig} shows that 
	$\mu[\calC_{\srp\srp}^h(\Rect_{2n,n})]$ is bounded below by some constant that only depends on~$C_\loc$.
	
	Henceforth we assume that ~$k > {\rho_\int}/{\eps}$. 
	Then, due to our initial assumption, 
	$$ 1 - C_\loc > \mu(\calG_{\loc}(\eps, k-1)) \geq 1 - \mu(\calA_{k-1}) -\mu(\calE_{k+4}) > 1 - C_\loc/2 -\mu(\calE_{k+4}),$$ 
	which implies~$\mu(\calE_{k+4}) > C_\loc/2$. 
	Write~$\tilde \calA_k$ for the horizontal shift of the event~$\calA_k$ by~$4\eps n$. 
	By the choice of~$k$, we have 
	$\mu(\tilde\calA_k ) = \mu(\calA_k )\geq C_\loc/2$.
	Using the FKG inequality, we find
	\begin{align*}
		\mu(\tilde \calA_k \cap \calE_{k+4})\geq \mu(\tilde \calA_k)\,\mu(\calE_{k+4}) \geq (C_\loc/2)^2. 
	\end{align*}
	Notice now that, if both~$\tilde \calA_k$ and~$\calE_{k+4}$ occur, 
	then the paths in the definition of these two events necessarily intersect. 
	In conclusion 
	\begin{align*}
		\mu\big[\Mid_0(\eps n)\xlra{\srp\srp \text{ in~$\Strip_n$}}(4\eps n ,0) + \Mid_0(\eps n)\big]  \geq (C_\loc/4)^2.
	\end{align*}
	As in the proof of Lemma~\ref{lem:centre}, this implies that 
	$\mu[\calC_{\srp\srp}^h(\Rect_{2n,n})]$ is larger than some threshold depending only on~$\eps$ and~$C_\loc$, and the lemma is proved. 
\end{proof}

In the  proof of Proposition~\ref{prop:RSW_Hugo} we will work with two scales: 
the scale~$n$ and a lower scale~$m$ chosen below. 
Moreover, the endpoints of the vertical paths will be fixed in some segment of length~$2\eps n$ where~$\eps > 0$ is also chosen below.

Fix~$m = \frac{\rho_\int}{2\rho_\out} n$.
Then, fix some~$\eps > 0$ so that
\begin{align}\label{eq:fix_eps}
	\eps \, n <\tfrac12 \rho_\int\,  m \qquad \text{ and } \qquad 
	\rho_\out \, m < \rho_\int \,  n - c_\loc\, \eps\, n. 
\end{align}
In conclusion, the scales~$\eps n$,~$m$ and~$n$ are fixed so that~$\eps n$ is much smaller than~$m$, which in turn is much smaller than~$n$. 
All constants below depend on the ratios between these scales. 

Write~$\Gamma_L$ and~$\Gamma_R$ for the left- and right-most, respectively, 
double-$\rp$ paths contained in~$\Strip_n$, with lower endpoint on~$\Mid_0(\eps n)$ 
and top endpoint in~$\Mid_n(\eps n)$ (recall that these are paths formed of edges of the hexagonal lattice).
If no such crossings exists, set~$\Gamma_L = \Gamma_R = \emptyset$. 
We will always orient such paths from their endpoint on~$\Line(0)$ towards that on~$\Line(n)$. 

When~$\Gamma_L$ and~$\Gamma_R$ exist and are disjoint, 
write~$\Int(\Gamma_L, \Gamma_R)$ for the domain with boundary formed of the concatenation of~$\Gamma_R$, 
the segment of~$\Line(n)$ between the top endpoints of~$\Gamma_R$ and~$\Gamma_L$ (from right to left),
$\Gamma_L$ (in reverse), and the segment of~$\Line(0)$ between 
the bottom endpoints of~$\Gamma_L$ and~$\Gamma_R$ (from left to right). 
Also let~$\Ext(\Gamma_L,\Gamma_R)$ be the set of faces of~$\Strip_n$ which are not strictly inside~$\Int(\Gamma_L, \Gamma_R)$; 
precisely,~$\Ext(\Gamma_L,\Gamma_R)$ contains all faces of~$\Strip_n\setminus \Int(\Gamma_L, \Gamma_R)$
as well as all faces adjacent to~$\Gamma_L$ or~$\Gamma_R$. 

It is standard that~$\Gamma_L$ and~$\Gamma_R$ may be explored from their left and right, respectively. 
That is, for~$\gamma_L$ and~$\gamma_R$ two possible realisations of~$\Gamma_L$ and~$\Gamma_R$, respectively, 
the event~$\{\Gamma_L = \gamma_L \text{ and } \Gamma_R =\gamma_R\}$ 
is measurable with respect to the state of faces in ~$\Ext(\gamma_L,\gamma_R)$. 

Our next goal is to show that, whenever~$\Gamma_L$ and~$\Gamma_R$ exist and behave reasonably well, 
they have a positive probability to be connected inside~$\Int(\Gamma_L, \Gamma_R)$ by a path of double-$\rp$. 
The notion of well-behaved vertical paths is defined below. 

For an edge-path~$\Gamma$ contained in~$\Strip_n$, with starting point in~$\Mid_0(\eps n)$ and endpoint in~$\Mid_n(\eps n)$,
let~$\Gamma^{b}$ be the segment of~$\Gamma$ contained between its starting-point and its first visit of~$\Line(m)$. 
Let~$\Gamma^t$  be the segment of~$\Gamma$  between its last visit of~$\Line(n-m)$ and its endpoint. 

Define~$\calG_\wb(k)$ as the event that 
any double-$\rp$ path~$\Gamma$ contained in~$\Strip_n$, 
with starting point in~$\Mid_0(\eps n)$ and endpoint in~$\Mid_n(\eps n)$ is well-behaved (for this value of~$k$), 
that is
\begin{itemize}
    	\item[(i)]~$\Gamma^b$ has one endpoint in~$\Mid_{m}(2\eps n)$ and~$\Gamma^t$ has one endpoint in~$\Mid_{n-m}(2\eps n)$;
    	\item[(ii)]~$\Gamma^b$ and~$\Gamma^t$ are both contained in~$\Rect(\rho_\out m,n)$ but 
    	each crosses~$\Linev(\rho_\int m)$,
    	\item[(iii)]~$\Gamma$ crosses~$\Linev(k\,\eps\, n)$ but not~$\Linev((k+c_\loc)\,\eps\, n)$.
\end{itemize} 

In order to apply our reasoning, we will ask that~$\Gamma_L$ and~$\Gamma_R$ are well-behaved, 
that is, we will ask that~$\calG_\wb(k)$ occurs for some~$k$. 
This is guaranteed by the following result. 

\begin{lem}[Paths are well-behaved]\label{lem:wb}
	For any~$C_\wb > 0$, there exists a constant~$\Delta_{\wb} > 0$ such that 
	\begin{align*}
		\Big(\mu [ \calG_\wb(k) ] < 1 - C_\wb \,\,\,\text{ for all } k \geq {\rho_\int}/{\eps}\Big)
		\Rightarrow \mu[\calC_{\srp\srp}^h(\Rect_{2n,n})] > \Delta_{\wb}.
	\end{align*}
\end{lem}

\begin{proof}
	Fix~$C_\wb > 0$ and assume~$\mu [ \calG_\wb(k) ] < 1 - C_\wb$ for all~$k\geq{\rho_\int}/{\eps}$. 
	Then, one out of the three conditions defining~$\calG_\wb(k)$ fails with probability at least~$C_\wb / 3$ for every~$k$. 
	Thus, at least one of the following cases occurs:
    	\begin{itemize}
	   	\item (i) fails with probability at least ~$C_\wb / 3$ for some~$k$. 
    		Then Lemma~\ref{lem:centre} states that~$\mu[\calC_{\srp\srp}^h(\Rect_{2m,m})] > \Delta_{\centre}$ 
    		for some~$ \Delta_{\centre} > 0$ depending only on~$C_\wb$. 
    		Using the horizontal translation invariance of~$\mu$, the same bound applies to any 
    		horizontal translate~$\Rect_{2m,m} +  j(m,0)$ of~$\Rect_{2m,m}$, with~$j \in \bbZ$. 
    		Using this and Corollary~\ref{cor:lengthen_crossings}, we find
    		\begin{align}\label{eq:cross_m_to_n}
    			\mu[\calC_{\srp\srp}^h(\Rect_{2n,n})] \geq \mu[\calC_{\srp\srp}^h(\Rect_{2n,m})] \geq (c\Delta_{\centre})^{2n/m},
    		\end{align}
			for some universal constant~$c> 0$. 
    	\item (ii) fails with probability at least ~$C_\wb / 3$ for some~$k$. 	
        	Then Lemma~\ref{lem:wig} implies that~$\mu[\calC_{\srp\srp}^h(\Rect_{2m,m})] > \Delta_{\wig}$ 
        	for some~$ \Delta_{\wig} > 0$ depending only on~$C_\wb$. 
        	As above, we conclude that ~$\mu[\calC_{\srp\srp}^h(\Rect_{2n,n})] \geq (c\Delta_{\wig})^{2n/m}$.
    	\item (iii) fails with probability at least ~$C_\wb / 3$ for all~$k$.
    		Then, by Lemma~\ref{lem:loc} applied with~$C_\loc = C_\wb / 3$, 
    		we deduce~$\mu[\calC_{\srp\srp}^h(\Rect_{2n,n})] > \Delta_{\loc}$.
    	\end{itemize}
    	We conclude that in all cases,~$\mu[\calC_{\srp\srp}^h(\Rect_{2n,n})]$ 
    	is bounded below by a constant depending only on~$C_\wb$, as required. 
\end{proof}

Now that we proved that some~$\calG_\wb(k)$ occurs with high probability, we will show that, when it does occur, 
$\Gamma_R$ and~$\Gamma_L$ connect to each other. 
Since~$\Gamma_R$ and~$\Gamma_L$ are measurable in terms of the spins in~$\Ext(\Gamma_L,\Gamma_R)$, the same applies to ~$\calG_\wb(k)$. 
Indeed, if~$\Gamma_L$ and~$\Gamma_R$ satisfy the conditions of~$\calG_\wb(k)$, 
then so do all double-$\rp$ paths contained in~$\Strip_n$, from~$\Mid_0(\eps n)$ to~$\Mid_n(\eps n)$.

\begin{lem}\label{lem:connect}
	There exists some universal constant~$C_{\connect} > 0$ such that, for any possible realisations~$\gamma_L,\gamma_R$ of~$\Gamma_L,\Gamma_R$  
 	with the property that~$\calG_\wb(k)$ occurs for some~$k \geq {\rho_\int}/{ \eps}$
	and any red spin configuration~$\zeta$ such that~$\Gamma_L = \gamma_L$ and~$\Gamma_R = \gamma_R$
	\begin{align}\label{eq:connect}
		\mu\big[ \gamma_L \xlra{\srp\srp \text{ in~$\Int(\gamma_L, \gamma_R)$}} \gamma_R \, \big|\, 
		\Gamma_L = \gamma_L, \, \Gamma_R = \gamma_R, \, \calG_\wb(k), \, \sigma_r = \zeta \text{ on } \Ext(\gamma_L,\gamma_R)\big]
		\geq C_\connect.
	\end{align}
\end{lem}

The conditioning in~\eqref{eq:connect} may be reduced simply to \{$\sigma_r = \zeta$ on~$\Ext(\gamma_L,\gamma_R)$\}, 
since this determines~$\Gamma_L = \gamma_L, \, \Gamma_R = \gamma_R$, which in turn implies that~$\calG_\wb(k)$ occurs. 
We included the latter conditions in~\eqref{eq:connect} to emphasise their importance. 

The lemma above is the heart of the proof of Theorem~\ref{prop:RSW_Hugo}. 
	
\begin{proof}
	Fix~$\gamma_L$,~$\gamma_R$ and~$\zeta$ as in the statement. 
	Let~$k \geq {\rho_\int}/{ \eps}$ be some value for which~$\calG_\wb(k)$ occurs.
	We may assume that~$\gamma_L$ and~$\gamma_R$  are disjoint, otherwise the conclusion is trivially attained. 
	We will proceed in two steps, 
	first we will create simple-$\rp$ connections between~$\gamma_L$ and~$\gamma_R$, 
	close to the top and bottom of~$\Strip_n$, respectively. 
	In a second stage, we connect~$\gamma_L$ and~$\gamma_R$ by a double-$\rp$ path 
	contained between the two simple-$\rp$ paths shown to exist in the previous step. 
	
	Recall that~$\gamma_L$ and~$\gamma_R$ are oriented from bottom to top. 
	Let~$R_1 = \Rect(\rho_\out m, m)$ and~$R_2 = R_1 + (0,n-m)$ 
	be the vertical translation of~$R_1$ contained between~$\Line(n-m)$ and~$\Line(n)$. 
	Due to~$\calG_\wb(k)$ occuring, 
	$\gamma_L^b$ and~$\gamma_R^b$ are contained in~$R_1$, 
	while~$\gamma_L^t$ and~$\gamma_R^t$ are contained in~$R_2$. 

	\smallskip
	
	\noindent{\bf Step 1: Simple-$\rp$ crossings.} 
	Let~$\calI$ be the event that 
	$\gamma_L$ and~$\gamma_R$ are connected by two simple-$\rp$ paths contained in 
	$R_1$ and~$R_2$, respectively. 
	We will now prove that~$\calI$ has positive probability, uniformly in~$m,n$,~$\gamma_L$,~$\gamma_R$ and~$\zeta$. 
	We do this for the connection in~$R_1$; the same argument applies in~$R_2$.
	The argument used in this step is exactly that of \cite{DumRaoTas18}. 
	Figure~\ref{fig:connect} contains an illustration of the construction below. 
	
	Recall that both~$\gamma_L^b$ and~$\gamma_R^b$ intersect~$\Linev(\rho_\int m)$
	but that their endpoints are in~$\Mid_{0}(\eps n)$ and~$\Mid_{m}(2\eps n)$, hence to the left of~$\Linev(\rho_\int m)$.
	Let~$A$ be the first point where~$\gamma_L$ intersects~$\Linev(\rho_\int m)$ and
	write~$\gamma_L'$ for the subpath of~$\gamma_L$ from its starting point up to~$A$.
	Then~$\gamma_R^b$ contains at least one subpath contained in the part of~$R_1$ to the right of~$\Linev(\rho_\int m)$, 
	which has both endpoints on~$\Linev(\rho_\int m)$, one below~$A$ and one above~$A$
	(this is because~$\gamma_R^b$ has both its endpoints to the left of~$\Linev(\rho_\int m)$). 
	Write~$\gamma_R'$ for the left-most such path and let~$C$ be the endpoint of~$\gamma_R'$ above~$A$. 
	
	Write~$\tau$ for the reflection with respect to~$\Linev(\rho_\int m)$
	(actually, with respect to the vertical axis~$\{\lfloor \rho_\int m \rfloor \}\times \bbR$). 
	Then~$\tau$ leaves the lattice invariant. 

	Observe now that~$\tau(\gamma_L')$ intersects~$\gamma_R'$. 
	Indeed,~$\tau(\gamma_L')$ runs from~$A$ to~$\Line(0)$ and is contained in the region to the right of~$\Linev(\rho_\int m)$, 
	while~$\gamma_R'$ separates~$A$ from~$\Line(0)$ in this same region.
	Let~$B$ be the first intersection point of~$\tau(\gamma_L')$ with~$\gamma_R'$ when starting from~$A$
	and let~$\gamma_A$ be the subpath of~$\tau(\gamma_L')$ between~$A$ and~$B$. 
	Let~$\gamma_B$ be the subpath of~$\gamma_R'$ between~$B$ and~$C$. 
	Finally set~$\gamma_C = \tau(\gamma_B)$,~$\gamma_D =\tau(\gamma_A)$ and~$D = \tau(B)$. 
	
	The paths~$\gamma_A,\gamma_B,\gamma_C$ and~$\gamma_D$ only intersect at their endpoints 
	and their concatenation bounds a domain which we call~$\calD$. 
	
	\begin{figure}
    	\begin{center}
    	\includegraphics[width = 0.55\textwidth, page =2]{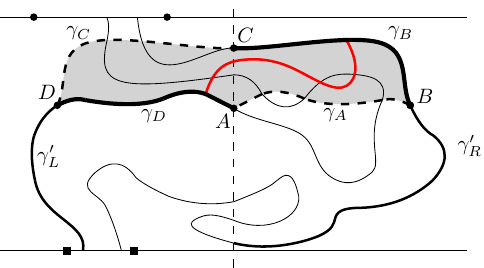}\qquad
    	\caption{The paths~$\gamma_L^b$ and~$\gamma_R^b$ are drawn in solid lines; the thicker parts are~$\gamma_L'$ and~$\gamma_R'$. 
		The reflections of parts of~$\gamma_L'$ and~$\gamma_R'$ are in dashed lines. The domain~$\calD$ is shaded. 
		Observe that any crossing from~$\gamma_B$ to~$\gamma_D$ in~$\calD$  induces a connection between~$\gamma_L^b$ and~$\gamma_R^b$ }
		\label{fig:connect}
    	\end{center}
	\end{figure}
	
	Let us derive a bound on the crossing probability of~$\calD$, independently of how~$\calD$ was formed. 
	Consider the red spin configuration~$\xi$ on~$\Cyl$ consisting only of~$\rm$ with the exception of the faces adjacent to
	$\gamma_B$ and those adjacent to~$\gamma_D$, which have spin~$\rp$. 
	By the same reasoning as in Lemma~\ref{lem:simple_self_duality} and due to the invariance of~$\calD$ under~$\tau$, we obtain
	\begin{align*}
		\mu \big( \gamma_B \xlra{\srp \text{ in } \calD} \gamma_D \,\big|\, \sigma_r = \xi \text{ on~$\calD^c \cup \partial_\int \calD$} \big) 
		\geq \tfrac13. 
	\end{align*}
	Write~$\mu_\calD^\xi$ for the conditional measure above.
	
	By Corollary~\ref{cor:monotonicity_bc} and due to the condition~$\Gamma_L = \gamma_L$ and~$\Gamma_R = \gamma_R$, 
	the measure~$\mu[.\,|\,\sigma_r = \zeta \text{ on~$\Ext(\gamma_L,\gamma_R)$}]$
	restricted to~$\calD \cap \Int(\gamma_L, \gamma_R)$ 
	dominates the restriction of~$\mu_\calD^\xi$ to this same set of faces.
	
	Set~$\calA$ to be the event that there exists a face-path~$\chi$ in~$\calD$, 
	with the first and last faces adjacent to~$\gamma_B$ and~$\gamma_D$, respectively, 
	and such that all faces of~$\chi$ that are contained in~$\Int(\gamma_L, \gamma_R)$ have spin~$\rp$. 
	Then 
	\begin{align*}
		\mu [\calA \,|\,\sigma_r = \zeta \text{ on }\Ext(\gamma_L,\gamma_R)]
		\geq \mu^\xi_\calD(\calA) 
		\geq \mu^\xi_\calD ( \gamma_B \xlra{\srp \text{ in } \calD} \gamma_D)
		\geq \tfrac{1}{3}. 
	\end{align*}
	Now observe that a path~$\chi$ as in the definition of~$\calA$ necessarily contains a subpath contained in~$\Int(\gamma_L, \gamma_R)$
	with the first and last faces adjacent to~$\gamma_L$ and~$\gamma_R$, respectively. 
	We conclude that 
	\begin{align*}
		\mu\big[ \gamma_L \xlra{\srp \text{ in } \Int(\gamma_L, \gamma_R) \cap R_1} \gamma_R \, \big|\, 
		\sigma_r = \zeta \text{ on } \Ext(\gamma_L,\gamma_R)\big]
		\geq \tfrac{1}{3}.
	\end{align*}
	Using the same argument in~$R_2$ and the FKG inequality, we obtain 
	\begin{align}\label{eq:simple_connect}
		\mu\big[\calI \, \big|\, 
		\sigma_r = \zeta \text{ on } \Ext(\gamma_L,\gamma_R)\big]
		\geq \tfrac{1}{9}.
	\end{align}
	\smallskip
	
	\noindent{\bf Step 2: Double-$\rp$ crossing.} 
	We will now prove that 
	\begin{align}\label{eq:double_connect}
		\mu\big[ \gamma_L \xlra{\srp\srp \, \text{ in } \Int(\gamma_L, \gamma_R)} \gamma_R \big|\, 
		\sigma_r = \zeta \text{ on } \Ext(\gamma_L,\gamma_R) \text{ and~$\calI$ occurs}\big]
		\geq \tfrac18.
	\end{align}
	The procedure is similar to that of Step 1, but at scale~$n$ rather than~$m$ and with some additional difficulties. 
	We recommend that the reader inspects Figure~\ref{fig:RSW2}, which contains the strategy of the proof
	as well as the relevant notation.

	When~$\calI$ occurs, we will denote by~$\Xi^1$ and~$\Xi^2$ be the lowest and highest, respectively, 
	paths of simple-$\rp$ from~$\gamma_L$ to~$\gamma_R$, contained in~$\Int (\gamma_L, \gamma_R)$.
	More precisely, define~$\Xi^1$ to be the lowest edge-path contained in~$\Int(\gamma_L, \gamma_R)$, 
	with endpoints on~$\gamma_L$ and~$\gamma_R$, respectively, with the property that all faces above it have spin~$\rp$.
	Define~$\Xi^2$ similarly, only that it is highest and that all faces below it are required to have spin~$\rp$. 
	By the definition of~$\calI$ and the first condition of~$\calG(k)$, 
	$\Xi^1$ and~$\Xi^2$ are contained in~$R_1$ and~$R_2$, respectively, whenever~$\calI$ occurs. 

	Let~$\chi^1$ and~$\chi^2$ be possible realisations of~$\Xi^1$ and~$\Xi^2$, respectively, such that~$\calI$ occurs.
	Define the domain~$\Int = \Int(\gamma_L, \gamma_R, \chi^1, \chi^2)$ as the set of faces delimited by these four paths.
	Also let~$\Ext=\Ext (\gamma_L, \gamma_R, \chi^1, \chi^2)$ be the set of faces  outside~$\Int$ along with those of~$\partial_\int \Int$. 
	By a standard exploration argument, the event~$\{\Xi^1 = \chi^1, \Xi^2 = \chi^2\}$
	is measurable with respect to the spins on~$\Ext$.
	Fix a red spin configuration~$\xi$ on~$\Ext (\gamma_L, \gamma_R, \chi^1, \chi^2)$ with 
	$\xi = \zeta$ on~$\Ext(\gamma_L,\gamma_R)$ and such that~$\Xi^1 = \chi^1$,~$\Xi^2 = \chi^2$.
	This implies in particular that all faces of~$\partial_\int\Int$ have spin~$\rp$ in~$\xi$. 
	 
	\begin{figure}
	\begin{center}
	\includegraphics[width = 0.48 \textwidth, page =1]{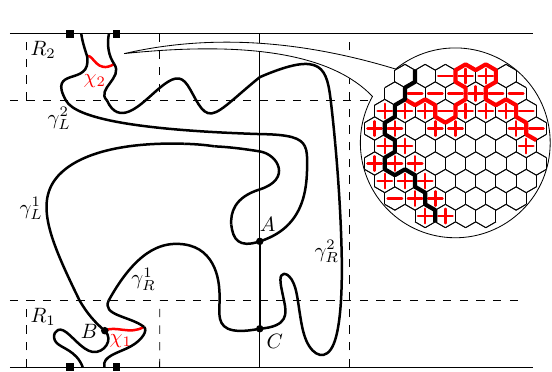}\quad
	\includegraphics[width = 0.48 \textwidth, page =2]{RSW3.pdf} \medskip\\		
	\includegraphics[width = 0.48 \textwidth, page =3]{RSW3.pdf}\quad
	\includegraphics[width = 0.48 \textwidth, page =4]{RSW3.pdf}		
	\caption{{\bf Top left:} 
		the paths~$\gamma_L$ and~$\gamma_R$ are connected in~$R_1$ and~$R_2$ by simple-$\frp$ paths (in red). 
		The dashed lines delimit~$\Strip_m$ and~$(0,n-m) + \Strip_m$;~$\Linev((k+c_\loc )\,\eps\, n)$ is also dashed. 
		{\bf Top right:} The domain~$\calD^1$ contains~$\chi^1$ and~$\tau(\chi^2)$ in its boundary. 
		{\bf Bottom left:} the domain~$\calD$ is formed from~$\calD^1$ (blue) and~$\calD^2$ (orange). 
		These two domains are in different copies of~$\bbH$, glued along the segment~$[A,C]$.
		{\bf Bottom right:} a deformation of~$\calD$ allows to embed it in the plane;
		it contains~$\Int$, whose deformation is shaded. 
		Any crossing from~$(AB)$ to~$(CD)$ contains a path from~$\gamma_L$ to~$\gamma_R$ in~$\Int$. }
		\label{fig:RSW2}
	\end{center}
	\end{figure}
	 
	The line~$\Linev(k\,\eps\, n)$ separates~$\chi^1$ from~$\chi^2$ inside the simply connected domain~$\Int$. 
	It follows that there exists at least one segment of~$\Linev(k\,\eps\, n)$
	that is fully contained in~$\Int$ and that separates~$\chi^1$ from~$\chi^2$ inside this domain. 
	Indeed, 	
	$\Linev(k\,\eps\, n)$ needs to intersect both~$\gamma_L$ to~$\gamma_R$ in order to separate~$\chi^1$ from~$\chi^2$. 
	Consider the intersections of~$\Linev(k\,\eps\, n)$ with~$\gamma_L$ and~$\gamma_R$ in increasing vertical order; 
	there necessarily exists one intersection with~$\gamma_R$ followed by one with~$\gamma_L$. 
	The segment of~$\Linev(k\,\eps\, n)$ between these two intersections has the desired property. 
		
	Let~$[A,C]$ be the first such segment when going from~$\chi^1$ to~$\chi^2$, where~$A$ denotes its higher endpoint
	(the segments with this property are naturally ordered, for instance by their end-points on~$\Gamma_L$). 
	Then~$A$ is a point of~$\gamma_L$ while~$C$ is a point of~$\gamma_R$. 
	Write~$\gamma^1_L$ and~$\gamma^2_L$ for the subpaths of~$\gamma_L$ 
	from the intersection with~$\chi^1$ to~$A$ 
	and from~$A$ to the intersection with~$\chi^2$, respectively. 
	The same notation applies to~$\gamma_R$. 
	
	Then~$[A,C]$ separates~$\Int$ into two sub-domains. 
	The first, which we call~$\Int^1$, has boundary formed of~$\chi^1$,~$\gamma_L^1$,~$[A,C]$ and~$\gamma_R^1$. 
	The boundary of the second, called~$\Int^2$, is the concatenation of~$\chi^2$,~$\gamma_R^2$,~$[A,C]$ and~$\gamma_L^2$. 

	Let~$\tau$ be the reflection with respect to the vertical axis~$\Linev(k\,\eps\, n)$. 
	Now define~$\calD^1$ as the union of the sets of faces of~$\Int^1$ and~$\tau(\Int^2)$. 
	Then~$\calD^1$ is itself a domain, whose boundary consists of~$\chi^1$,~$\tau(\chi^2)$,~$[A,C]$ and 
	pieces of~$\gamma_L^1$,~$\gamma_R^1$,~$\tau(\gamma_L^2)$ and~$\tau(\gamma_R^2)$. 
	It is particularly important that~$\chi^1$ is fully part of the boundary of~$\calD^1$. 
	This is because~$\tau(\Int^2)$ lies entirely to the right of~$\Linev((k-c_\loc)\,\eps\, n)$, 
	and thus does not intersect~$R_1$ (property~$(iii)$ of well-behaved paths, see also \eqref{eq:fix_eps}). 
	For similar reasons,~$\tau(\chi^2)$ is also fully contained in the boundary of~$\calD^1$. 
	
	Let~$B$ be the intersection point of~$\gamma_L$ and~$\chi^1$; it is on the boundary of~$\calD^1$. 
	Write~$\calD^2 = \tau(\calD^1)$ and~$D = \tau(B)$. 
	Define the domain~$\calD$ by gluing~$\calD^1$ and~$\calD^2$ along the segment~$[A,C]$.
	The result of this operation is not a domain of~$\bbH$. Indeed,~$\calD^1$ and~$\calD^2$ may intersect in~$\bbH$; 
	we will consider them as embedded in two different copies of~$\bbH$ that are then glued along the segment~$[A,C]$. 
	Nevertheless,~$\calD$ is planar (that is, it may be embedded in the plane after some distortion; see Figure~\ref{fig:RSW2}) 
	and is simply connected. 
	Orient~$\partial_E \calD$ in counter-clockwise order and 
	write~$(AB)$,~$(BC)$ etc. for the portions of~$\partial_E \calD$ between~$A$ and~$B$,~$B$ and~$C$ etc.
	
	Let us study the measure with boundary conditions~$\rp\rm$ on~$\calD$. 
	By the same argument as for Lemma~\ref{lem:monochrom}, either~$(AB)$ is connected to~$(CD)$ inside~$\calD$ 
	by a path of double-$\rp$ or double-$\rm$,
	or~$(BC)$ is connected to~$(DA)$ by a path of double-$\bp$ or double-$\bm$.
	As in Lemma~\ref{lem:square_crossed},
	the domain is symmetric with respect to~$\tau$ and 
	the boundary conditions~$\rp\rm$ favour the connection with double-$\rp$. Thus we find
	\begin{align}\label{eq:RSW2_double}
		\mu_\calD^{\srp\srm}\big[ (AB)\xlra{\srp\srp \text{ in~$\calD$}}(CD)\big] \geq \tfrac14.
	\end{align}
	
	Observe now that~$\calD$ contains~$\Int$. 
	Moreover, any path crossing from ~$(AB)$ to~$(CD)$ in~$\calD$ contains a subpath 
	which is contained in~$\Int$ and which has endpoints on~$\gamma_L$ and~$\gamma_R$, respectively. 
	Indeed, the segment~$(AB)$ is above~$\gamma_L^1$, while~$(CD)$ is below~$\gamma^2_R$. 
	
	Finally, we claim that the restriction of~$\mu(.\, |\, \sigma_r = \xi \text{ on~$\Ext$})$
	to~$\Int$ dominates that of~$\mu_\calD^{\srp\srm}$. 
	We start off with a heuristic explanation. 
	The key to this argument is to observe that~$\calD$ may be obtained from~$\Int$ by ``pushing away'' parts of the boundary of~$\calD$, 
	but that these only belong to~$\gamma_L$ and~$\gamma_R$, not to~$\chi^1$ or~$\chi^2$. 
	Since these are double-$\rp$ paths in~$\xi$, the monotonicity of boundary conditions applies, and we may conclude. 
	
	Let us now present a rigorous proof of this domination with a slightly weaker conclusion. 
	As already explained,~$\calD$ is part of two copies of~$\bbH$ glued along the segment~$[A,C]$. 
	Let~$\calD'$ be a planar domain of this graph that contains~$\calD$ along with all faces adjacent to it. 
	Then, due to the Spatial Markov property that also applies in this slightly different setting, 
	$\mu_\calD^{\srp\srm}$ is the restriction to~$\calD$ of 
	$\mu_{\calD'}(.\,|\, \sigma_r \equiv \rm \text{ on~$\calD' \setminus \calD$ and }\sigma_r \equiv \rp \text{ on~$\partial_\int \calD$})$.

	Since~$\calD'$ is planar, the FKG inequality holds for~$\mu_{\calD'}$. 
	Let~$\calA^-$ be the set of faces of~$\partial_\out \Int$ adjacent to~$\chi^1$ or~$\chi^2$. 
	Let~$\calA^+$ be all the other faces of~$\calD' \setminus \Int$ along with~$\partial_\int \Int$.
	Then, by the monotonicity of boundary conditions (Corollary~\ref{cor:monotonicity_bc} \textit{(i)}) and the consideration above,
	the restriction of ~$\mu_\calD^{\srp\srm}$ to~$\Int$ is dominated by that of 
	$\mu_{\calD'}(.\, | \, \sigma_r \equiv \rp \text{ on~$\calA^+$ and }\sigma_r \equiv \rm \text{ on~$\calA^-$})$. 
	Moreover, the latter is equal to the restriction of 
	$\mu(.\, |\, \sigma_r \equiv \rp \text{ on~$\Ext\setminus \calA^-$ and } \sigma_r \equiv \rm \text{ on~$\calA^-$})$ to~$\Int$.
	(Here~$\calA^-$ is also viewed as a subset of~$\Cyl$.) 
	In conclusion
	\begin{align*}
		&\mu\big[\gamma_L \xlra{\srp\srp \text{ in~$\Int$}} \gamma_R
			  \, \big|\, \sigma_r = \xi \text{ on~$\Ext$}\big]\\
		&\qquad \geq \mu\big[\gamma_L \xlra{\srp\srp \text{ in~$\Int$}} \gamma_R
			  \, \big|\, \sigma_r\equiv \rp \text{ on~$\Ext\setminus\calA^-$ and }\sigma_r \equiv \rm \text{ on~$\calA^-$}\big]\\
		&\qquad = \mu_{\calD'}\big[\gamma_L \xlra{\srp\srp \text{ in~$\Int$}} \gamma_R
			 \, \big|\,\sigma_r \equiv \rp \text{ on~$\calA^+$ and }\sigma_r \equiv \rm \text{ on~$\calA^-$}\big]\\
		&\qquad \geq  \mu_\calD^{\srp\srm}\big[\gamma_L \xlra{\srp\srp \text{ in~$\Int$}} \gamma_R\big] \\
		&\qquad\geq \mu_\calD^{\srp\srm}\big[ (AB)\xlra{\srp\srp \text{ in~$\calD$}}(CD)\big] 
		\geq \tfrac14.
	\end{align*}
	\smallskip
	
	\noindent{\bf Conclusion.} 
	Equations~\eqref{eq:simple_connect} and~\eqref{eq:double_connect} imply that 
	\begin{align*}
		\mu\big[ \gamma_L \xlra{\srp\srp \, \text{ in } \Int(\gamma_L, \gamma_R)} \gamma_R \big|\, 
		\sigma_r = \zeta \text{ on } \Ext(\gamma_L,\gamma_R)\big]
		\geq \tfrac1{36}, 
	\end{align*}
	which is the desired bound.
\end{proof}

We are finally ready to prove the main result of the section, namely Proposition~\ref{prop:RSW_Hugo}.

\begin{proof}[Proposition~\ref{prop:RSW_Hugo}]	
	Recall from \eqref{eq:fix_eps} that~$m$ and~$\eps$ are fixed.
	Let~$C_v = \mu[\calC_{\srp\srp}^v(\Rect_{2n,n})]$. 
	The bottom boundary of~$\Rect_{3n,n}$ may be partitioned into~$18/\eps$ segments of length~$\eps n /3$.
	At least one of these segments is connected inside~$\Strip_n$ 
	by a double-$\rp$ path to ~$\Linev(n)$ with probability at least~$C_v \eps /18$.  
	Since the measure~$\mu$ is translation invariant, 
	\begin{align}\label{eq:mid_to_top}
		\mu\big[ \Mid_0(\eps n/6) \xlra{\srp\srp \text{ in }\Strip_n } \Line(n) \big] \geq \frac{C_v \,\eps }{18}.
	\end{align}
	Let~$\Delta_\centre$ be given by Lemma~\ref{lem:centre} with~$C_{\centre} = \frac{C_v \eps}{36}$ and~$\eps/6$ instead of~$\eps$. 
	If~$\mu[\calC_{\srp\srp}^h(\Rect_{2n,n})] > \Delta_{\centre}$ the proof is complete. 
	We will therefore assume that~$\mu[\calH_\centre(m)] < C_{\centre}$, 
	which along with~\eqref{eq:mid_to_top} implies 
	\begin{align}\label{eq:mid_to_mid}
		\mu\big[ \Mid_0(\eps n/3) \xlra{\srp\srp \text{ in }\Strip_n } \Mid_n(\eps n/3)  \big] \geq \frac{C_v \, \eps}{36}.
	\end{align}		

	For~$j \in \bbZ$, write~$M_j^{b} = (j \frac{2}3\eps n, 0) + \Mid_0(\eps n/3)$ and~$M_j^{t} = (j \frac{2}3\eps n, 0) + \Mid_n(\eps n/3)$.
	Let~$\calJ$ be the event that~$M_j^{b}$ is connected to~$M_j^{t}$ by a double-$\rp$ path inside~$\Strip_n$ for both~$j = -1$ and~$j = 1$. 
	Using again the translation invariance of~$\mu$, the FKG inequality and~\eqref{eq:mid_to_mid}, we find 
	\begin{align}\label{eq:calJ}
		\mu [\calJ  ] \geq \Big(\tfrac{C_v \eps}{36}\Big)^2. 
	\end{align}
		
	Now let~$\Delta_\wb$ be the constant given by Lemma~\ref{lem:wb} with~$C_\wb = \frac12 \big(\tfrac{C_v \eps}{36}\big)^2$. 
	If~$\mu[\calC_{\srp\srp}^h(\Rect_{2n,n})] > \Delta_{\wb}$ we have obtained the result. 
	We may therefore assume the opposite, thus that 
	\begin{align*}
		\mu [ \calG_\wb(k) ] > 1 - C_\wb,
	\end{align*}
	for some~$k \geq {\rho_\int}/{\eps}$. By choice of~$C_\wb$ and using a union bound, we conclude that
	\begin{align}\label{eq:calJG}
		\mu [\calJ  \cap \calG_\wb(k) ] \geq \tfrac12 \Big(\tfrac{C_v \eps}{72}\Big)^2 = C_\wb. 
	\end{align}
	Applying now Lemma~\ref{lem:connect}, we find
	\begin{align*}
		\mu\big[ \{\Gamma_L \xlra{\srp\srp \text{ in~$\Int(\Gamma_L, \Gamma_R)$}} \Gamma_R\} \cap \calJ \cap \calG_\wb(k)\big]
		\geq C_\connect \cdot C_\wb.
	\end{align*}
	When~$\calJ$ occurs, the endpoints of~$\Gamma_L$ are contained in~$M_{-1}^b$ and~$M_{-1}^t$, respectively, 
	while those of~$\Gamma_R$ are in~$M_{1}^b$ and~$M_{1}^t$. 
	Thus, when all three events above occur simultaneously,~$M_{-1}^b$ and~$M_{1}^b$ are connected inside~$\Strip_n$ by a path of double-$\rp$. 
	We conclude that 
	\begin{align*}
		\mu\big[ M_{-1}^b \xlra{\srp\srp \text{ in~$\Strip_n$}} M_{1}^t\big]
		\geq C_\connect \cdot C_\wb.
	\end{align*}
	We conclude in the same way as in the proof of Lemma~\ref{lem:centre}:
	the lower bound above applies also to~$\{M_{j-1}^b \xlra{\srp\srp \text{ in~$\Strip_n$}} M_{j+1}^t\}$ for all~$-\tfrac3{2 \eps} \leq j \leq \tfrac3{2 \eps}$. 
	Using the FKG inequality, the intersection of all these translations occurs with probability at least~$(C_\connect \cdot C_\wb)^{3/\eps +1}$.
	When all the events above occur,~$\Rect_{2n,n}$ contains a double-$\rp$ horizontal crossing. Thus
	\begin{align*}
		\mu[\calC_{\srp\srp}^h(\Rect_{2n,n})] 
		\geq (C_\connect \cdot C_\wb)^{3/\eps +1}.
	\end{align*}
	Since~$\eps$ is a universal constant and~$C_\connect$ and~$C_\wb$ only depend on~$C_v$, the above provides the desired bound. 
\end{proof}

\subsection{Crossing rectangles in mixed boundary conditions}\label{sec:mixed_bc}

We give two statements that are crucial in the proof of Theorem~\ref{thm:dicho}.
They are crossing probability estimates similar to those of Proposition~\ref{prop:RSW_Vincent}.
What is essential here is that they are in finite domains with mixed boundary conditions.

\begin{prop}\label{prop:cross_cyl}
	For~$C \geq 3$ there exists~$\delta = \delta(C) > 0$ such that, for all~$n\geq 1$, 
	\begin{align*}
	\mu_{\Cyl_{Cn,5n}}^{\srp\srp/\srm\srm} \big[ \calC^h_{\srp\srp}([-2n,2n] \times [3n,4n])\big] \geq \delta 
	\quad\text{ or }\quad
	\mu_{\Cyl_{Cn,5n}}^{\srp\srp/\srm\srm} \big[ \calC^v_{\srp\srp}([-3n,3n] \times [3n,4n])\big] \geq \delta.	
	\end{align*}
\end{prop}

\begin{cor}\label{cor:RSW_mixed_bc}
	For all~$C_h \geq 3$ and~$C_v \geq 1$ there exists~$\delta = \delta (C_h,C_v)> 0$ such that, for all~$n\geq 1$, 
	\begin{align*}
		\mu_{\Rect_{C_h n,C_v n}}^{\srp\srp/\srm\srm} \big[ \calC^h_{\srp\srp}(\Rect_{C_h n,n})\big] > \delta. 	
	\end{align*}
\end{cor}

Corollary~\ref{cor:RSW_mixed_bc} is referred to in \cite{DumRaoTas18} as the ``pushing'' lemma; 
it is an essential result in establishing the dichotomy of Corollary~\ref{cor:dicho}.

The results stated above mimic the structure of the original RSW theory: the first result serves as an input 
(such as self-duality in critical bond percolation on~$\bbZ^2$ 
or as Corollary~\ref{cor:square_crossed} for the weaker RSW statement of Proposition~\ref{prop:RSW_Vincent}), 
the second states that horizontal crossings may be extended to longer rectangles.
The latter follows from the former in a fairly standard way using Proposition~\ref{prop:RSW_Hugo}.
For clarity, we will avoid using Proposition~\ref{prop:RSW_Hugo} in the proof of Proposition~\ref{prop:cross_cyl}.
We start with the proof of the proposition; the proof of the corollary may be found at the end of the section.

\begin{proof}[Proposition~\ref{prop:cross_cyl}]
	Fix~$C\geq3$. We will proceed by contradiction and will assume that  
	\begin{align}\label{eq:assumption}
	    &\mu_{\Cyl_{Cn,5n}}^{\srp\srp/\srm\srm}\big[ \calC^h_{\srp\srp}([-2n,2n] \times [3n,4n])\big] < \delta 
  	 	\qquad \text{ and }\nonumber\\
    	&\mu_{\Cyl_{Cn,5n}}^{\srp\srp/\srm\srm} \big[ \calC^v_{\srp\srp}([-3n,3n] \times [3n,4n])\big] < \delta, 
	\end{align} 
	for some constant~$\delta > 0$ that we will choose later. It will be obvious that the choice of~$\delta$ only depends on~$C$.  

	Write~$\Cyl$ for~$\Cyl_{Cn,5n}$ and~$\mu = \mu_\Cyl^{\srp\srp/\srm\srm}$.
	The cylinder is split into five strips of height~$n$: 
	$\Strip_i = [-Cn - 2, Cn + 2] \times [(i-1)n,in]$ for~$i =1,\dots, 5$. 
	
	The proof of the proposition is based on two claims that we state and prove below. 
	The whole argument is summarised in Figure~\ref{fig:Cyl}.
	
	\begin{figure}
	\begin{center}
    	\includegraphics[width = 0.3\textwidth, page = 1]{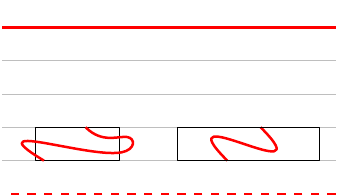}\quad
    	\includegraphics[width = 0.3\textwidth, page = 2]{Cyl2.pdf}\quad
    	\includegraphics[width = 0.3\textwidth, page = 3]{Cyl2.pdf}
	  	\caption{
       	{\em Left:} A vertical double-crossing of red plus of~$\Strip_2$
    		contains either a vertical crossing of a~$6n\times n$ rectangle or a horizontal one of a~$4n\times n$ rectangle. 
       	{\em Middle:} The absence of a vertical double-crossing of~$\Strip_2$ of constant red spin implies the existence 
        	of a double-circuit of constant blue spin. 
	 		Conditionally on it, we study the probability that~$\R$ contains a double crossing of red-plus spins. 
    	{\em Right:} The symmetric cylinder~$\widetilde\Cyl$.
		The probability that~$\R$ is crossed (horizontally or vertically) 
		by a double red-plus path is higher in the right image than in the middle. 
     	}
		\label{fig:Cyl}
	\end{center}
	\end{figure}
	
	\begin{claim}\label{cl:rsw1}
		Write~$\Circ_{\sbp\sbp}(\Strip_2)$ for the event that there exists a double-$\bp$ path 
		winding around~$\Cyl$ and contained in~$\Strip_2$. 
		Assuming~$\delta$ is small enough, 
		\begin{align*}
			\mu [\Circ_{\sbp\sbp}(\Strip_2)] \geq 1/4.
		\end{align*}
	\end{claim}
	
	\begin{proof}
		The same argument as in the proof of Lemma~\ref{lem:square_crossed} shows that
		either~$\Strip_2$ is crossed vertically by double-path of constant red spins, 
		or it contains a horizontal double-circuit (winding around the cylinder) of constant blue spins. 
		Thus 
		\begin{align}\label{eq:m1}
			\mu (\Circ_{\sbp\sbp}(\Strip_2)) + \mu (\Circ_{\sbm\sbm}(\Strip_2))
			+ \mu (\calC^{v}_{\srp\srp}(\Strip_2))+ \mu (\calC^{v}_{\srm\srm}(\Strip_2))
			\geq 1,
		\end{align}
		where~$\Circ_{\sbm\sbm}$ is defined similarly to~$\Circ_{\sbp\sbp}$ and~$\calC^{v}_{\srp\srp}(\Strip_2)$ 
		is the event that~$\Strip_2$ contains a path of double-$\rp$ with one endpoint on its bottom and one on its top. 

		If~$\calC^{v}_{\srp\srp}(\Strip_2)$ occurs, then at least one of the 
		rectangles~$[ kn,(k+6)n] \times[2n,3n]$ with~$-C \leq k < C$ is crossed vertically by a double-$\rp$ path, 
		or one of the rectangles~$[ kn,(k+4)n] \times[2n,3n]$ with~$-C \leq k < C$ is crossed horizontally by a double-$\rp$ path. 
		Due to our assumptions~\eqref{eq:assumption} and to the monotonicity with respect to boundary conditions, 
		all of the crossing events above occur with probabilities at most~$\delta$. Thus, 
		\begin{align*}
			\mu [\calC^{v}_{\srp\srp}(\Strip_2)] \leq 4C \delta.
		\end{align*}
		The same argument applies to double-$\rm$ crossings and we find
		\begin{align*}
			\mu [\calC^{v}_{\srm\srm}(\Strip_2)] \leq 4C \delta.
		\end{align*}
		For this second inequality, the monotonicity of boundary conditions was not used, 
		but rather the invariance of~$\mu$ under vertical reflection composed with red spin flip.  
		
		Assume now that~$\delta \leq \frac{1}{16 C}$.
		Then the first two terms of~\eqref{eq:m1} sum up to at least~$1/2$.
		Moreover,~$\mu$ is invariant under blue spin flip, hence these two terms are equal. In conclusion, each is larger than~$1/4$. 
	\end{proof}
	
	\begin{claim}\label{cl:rsw2}
		Let~$\R = [-2n,2n] \times [3n,4n]$. Then
		\begin{align*}
			\mu \big[\calC^h_{\srp\srp}(\R) \big|\, \Circ_{\sbp\sbp}(\Strip_2)\big] +
			\mu \big[\calC^v_{\srp\srp}(\R) \big|\, \Circ_{\sbp\sbp}(\Strip_2)\big] \geq \tfrac12.
		\end{align*}
	\end{claim}	
	\begin{proof}
		 If~$\Circ_{\sbp\sbp}(\Strip_2)$ occurs, let~$\Gamma$ be the lowest circuit as in its definition.
		 Let~$\gamma$ be a possible realisation for~$\Gamma$ and let~$\overline \gamma$ be the reflection of~$\gamma$ 
		 with respect to the horizontal line~$\bbR \times \{7n/2\}$. 
		 Then~$\overline \gamma$ lies entirely above~$\bbR \times \{5n\}$, hence above the top of~$\Cyl$. 
		 It will be useful to view~$\gamma$ and~$\overline\gamma$ as drawn on the infinite vertical cylinder 
		~$\Cyl_\infty := (\bbR / (2Cn+4)\bbZ) \times \bbR$, of whom~$\Cyl$ is a subset.  
		 
		 Let~$\widetilde\Cyl$ be the cylinder contained between~$\gamma$ and~$\overline \gamma$
		 and let~$\textsf{Top}$ be the top boundary of~$\Cyl$ (it may be seen as a horizontal circuit in~$\widetilde \Cyl$). 
		 Let~$\mu_{\widetilde\Cyl}^{\srp\srp/\sbp\sbp}$ be the measure on~$\widetilde \Cyl$ 
		 with boundary conditions~$\rp\rp$ on~$\overline \gamma$ and~$\bp\bp$ on~$\gamma$.
		 Precisely,~$\mu_{\widetilde\Cyl}^{\srp\srp/\sbp\sbp}$ is the uniform measure on pairs of coherent spin configuration 
		~$(\sigma_r, \sigma_b)$ on~$\widetilde{\Cyl}$ with the property that 
		 all faces adjacent to~$\gamma$ have blue spin~$\sigma_b = \bp$ and all faces adjacent to~$\overline\gamma$ have~$\sigma_r = \rp$. 
		 
		 Then, both the red-spin and blue-spin marginal of~$\mu_{\widetilde\Cyl}^{\srp\srp/\sbp\sbp}$ have the FKG property.
		 We sketch the proof of this fact next. 
		 Embed~$\widetilde\Cyl$ in the plane in the same way that~$\Cyl_{m,n}$ was embedded in Figure~\ref{fig:R_Cyl};
		 call~$\calD$ the planar graph thus obtained. 
		 The measure~$\mu_{\widetilde\Cyl}^{\srp\srp/\sbp\sbp}$ 
		 is equal to that on~$\calD$ with the faces adjacent to~$\gamma$ conditioned to have blue spin~$\bp$ 
		 and those adjacent to~$\overline\gamma$ to have red spin~$\rp$. 
		 By Corollary~\ref{cor:fkg}, this conditioned measure does satisfy the FKG inequality for both the blue and red-spin marginals. 
		
		 Moreover, the boundary conditions of~$\mu_{\widetilde\Cyl}^{\srp\srp/\sbp\sbp}$ 
		 satisfy the following Spatial Markov property for measures on~$\Cyl_\infty$.
		 Let~$A_r$ be the set of faces that are either 
		 below~$\gamma$, above~$\overline \gamma$, or below~$\overline \gamma$ but adjacent to it. 
		 Similarly, let~$A_b$ be the set of faces that are above~$\overline \gamma$, below~$\gamma$, or above~$\gamma$ but adjacent to it. 
		 Then, for any red spin configuration~$\xi_r$ with the property that all faces adjacent to~$\overline \gamma$ have spin~$\rp$
		 and any blue spin configuration~$\xi_b$ with the property that all faces adjacent to~$ \gamma$ have spin~$\bp$, we have 
		~$$ \mu_{\Cyl_\infty}(. \, | \,  \sigma_r = \xi_r \text{ on~$A_r$ and }\sigma_b = \xi_b \text{ on~$A_b$} ) = 
		 \mu_{\widetilde\Cyl}^{\srp\srp/\sbp\sbp},$$
		 where the equality refers only to the restriction on~$\widetilde\Cyl$. 
		 This fact may be proved exactly as Theorem~\ref{thm:DMP} and we do not give further details. 
		 
		 In addition, the Spatial Markov property for boundary conditions~$\rp\rp$ holds under~$\mu_{\widetilde\Cyl}^{\srp\srp/\sbp\sbp}$. 
		 Thus, using the FKG property for~$\mu_{\widetilde\Cyl}^{\srp\srp/\sbp\sbp}$, we find
		 \begin{align}\label{eq:Gamma}
		 	\mu \big(\calC^h_{\srp\srp}(\R) \,\big|\, \Gamma = \gamma \big)
			=\mu_{\widetilde\Cyl}^{\srp\srp/\sbp\sbp} \big(\calC^h_{\srp\srp}(\R) \,\big|\, \textsf{Top} \equiv \rp\big) 
			\geq \mu_{\widetilde\Cyl}^{\srp\srp/\sbp\sbp} \big(\calC^h_{\srp\srp}(\R)\big),
		 \end{align}
		 where~$\textsf{Top} \equiv \rp$ stands for the event that all faces adjacent to the top of~$\Cyl$ have red spin~$\rp$.
		The same holds for~$\calC^v_{\srp\srp}(\R)$. 
		
		Finally, let us mention that~$\mu_{\widetilde\Cyl}^{\srp\srp/\sbp\sbp}$ 
		is invariant under reflection with respect to the horizontal line~$\bbR \times \{7n/2\}$
		composed with colour inversion. 
		This transformation maps~$\calC^v_{\srp\srp}(\R)$ onto~$\calC^v_{\sbp\sbp}(\R)$
		and~$\calC^v_{\srm\srm}(\R)$ onto~$\calC^v_{\sbm\sbm}(\R)$.
		Now, due to Lemma~\ref{lem:square_crossed},~$\R$ contains either a horizontal double-crossing of constant red spin, 
		or a vertical one of constant blue spin. Thus,
		\begin{align}\label{eq:wtCyl}
    		1 &\leq \mu_{\widetilde\Cyl}^{ \srp\srp / \sbp\sbp} \big(\calC^h_{\srp\srp}(\R)\big)
    		+ \mu_{\widetilde\Cyl}^{ \srp\srp / \sbp\sbp} \big(\calC^h_{\srm\srm}(\R)\big)
    		+ \mu_{\widetilde\Cyl}^{ \srp\srp / \sbp\sbp} \big(\calC^v_{\sbp\sbp}(\R)\big)
    		+ \mu_{\widetilde\Cyl}^{ \srp\srp / \sbp\sbp} \big(\calC^v_{\sbm\sbm}(\R)\big)\\
    		&= \mu_{\widetilde\Cyl}^{ \srp\srp / \sbp\sbp} \big(\calC^h_{\srp\srp}(\R)\big)
    		+ \mu_{\widetilde\Cyl}^{ \srp\srp / \sbp\sbp} \big(\calC^h_{\srm\srm}(\R)\big)
    		+ \mu_{\widetilde\Cyl}^{ \srp\srp / \sbp\sbp} \big(\calC^v_{\srp\srp}(\R)\big)
    		+ \mu_{\widetilde\Cyl}^{ \srp\srp / \sbp\sbp} \big(\calC^v_{\srm\srm}(\R)\big).\nonumber
		\end{align}
		Define the boundary conditions~$\rm\rm /\bp\bp$ on~$\widetilde\Cyl$ in the same way as~$\rp\rp /\bp\bp$.
		Then~$\mu_{\widetilde\Cyl}^{ \srm\srm / \sbp\sbp}$ 
		is obtained from~$\mu_{\widetilde\Cyl}^{ \srp\srp / \sbp\sbp}$ by flipping the sign of all red spins. 
		Using the same argument as in Corollary~\ref{cor:monotonicity_bc} \textit{(ii)}, it may be shown that the red spin marginal
		of~$\mu_{\widetilde\Cyl}^{ \srp\srp / \sbp\sbp}$ dominates that of~$\mu_{\widetilde\Cyl}^{\srm\srm/\sbp\sbp}$.
		In particular,  
		\begin{align*}
    		\mu_{\widetilde\Cyl}^{ \srp\srp / \sbp\sbp} \big(\calC^h_{\srm\srm}(\R)\big) = 
    		\mu_{\widetilde\Cyl}^{ \srm\srm / \sbp\sbp} \big(\calC^h_{\srp\srp}(\R)\big) 
    		\leq \mu_{\widetilde\Cyl}^{ \srp\srp / \sbp\sbp} \big(\calC^h_{\srp\srp}(\R)\big).
		\end{align*}
		The same holds for vertical crossings. 
		Insert the above in~\eqref{eq:wtCyl} and use~\eqref{eq:Gamma}, to find 
		\begin{align*}
    		1
			\leq 2\mu_{\widetilde\Cyl}^{ \srp\srp / \sbp\sbp} \big(\calC^h_{\srp\srp}(\R)\big) + 2\mu_{\widetilde\Cyl}^{ \srp\srp / \sbp\sbp} \big(\calC^v_{\srp\srp}(\R)\big)
			\leq 
			2\mu \big(\calC^h_{\srp\srp}(\R) \big|\, \Gamma = \gamma\big)
    		+ 2\mu \big(\calC^v_{\srp\srp}(\R) \big|\, \Gamma = \gamma\big).
		\end{align*}
%
		In conclusion 
		\begin{align*}
    		&\mu \big[\calC^h_{\srp\srp}(\R) \big|\, \Circ_{\sbp\sbp}(\Strip_2)\big] +
    		\mu \big[\calC^v_{\srp\srp}(\R) \big|\, \Circ_{\sbp\sbp}(\Strip_2)\big] \\
    		 &= \sum_{\gamma} 
    		\left[\mu \big(\calC^h_{\srp\srp}(\R) \big|\, \Gamma = \gamma\big)
    		+\mu \big(\calC^v_{\srp\srp}(\R) \big|\, \Gamma = \gamma\big)\right]\cdot \mu \big[\Gamma = \gamma \big|\, \Circ_{\sbp\sbp}(\Strip_2)\big]
    		 \geq \tfrac12,
		\end{align*}
		where the sum is over all possible realisations~$\gamma$ of~$\Gamma$.
	\end{proof}
	
	Finally, let us finish the proof of Proposition~\ref{prop:cross_cyl}. 
	For~$\delta > 0$ small enough for Claim~\ref{cl:rsw1} to hold, using Claims~\ref{cl:rsw1} and~\ref{cl:rsw2}, we find
	\begin{align*}
		&\mu \big(\calC^h_{\srp\srp}(\R) \big) +
		\mu \big(\calC^v_{\srp\srp}(\R)\big) \\
		\geq& 
		\Big(\mu \big[\calC^h_{\srp\srp}(\R) \big|\, \Circ_{\sbp\sbp}(\Strip_3)\big] +
		\mu \big[\calC^v_{\srp\srp}(\R) \big|\, \Circ_{\sbp\sbp}(\Strip_3)\big]\Big)\cdot \mu \big[\Circ_{\sbp\sbp}(\Strip_3)\big]
		\geq \frac18.
	\end{align*}
	However, by our assumption~\eqref{eq:assumption}, the left-hand side is bounded above by~$2\delta$. 
	This leads to a contradiction if~$\delta$ is chosen smaller than~$1/16$. 
\end{proof}

\begin{proof}[Corollary~\ref{cor:RSW_mixed_bc}]
	Fix~$C_h,C_v$ and~$n$ as in the statement. 
	We may assume~$n$ larger than some constant depending on~$C_v$ and~$C_h$;
	the inequality for smaller values may be satisfied by altering the value of~$\delta(C_v,C_h)$. 
	
	Apply Proposition~\ref{prop:cross_cyl} to~$N = \frac{C_v}5 n$ and~$C = \frac{C_h}{5C_v}$ to obtain that
	\begin{align}
    	&\mu_{\Cyl_{C_h n ,C_v n}}^{\srp\srp/\srm\srm} \big[ \calC^h_{\srp\srp}([-2N,2N] \times [3N,4N])\big] \geq \delta 	
		\label{eq:first_ineq}
    	\qquad\text{ or }\\
    	&\mu_{\Cyl_{C_h n ,C_v n}}^{\srp\srp/\srm\srm} \big[ \calC^v_{\srp\srp}([-3N,3N] \times [3N,4N])\big] \geq \delta, 
	\end{align}
	for some~$\delta> 0$ depending only on~$C_h$ and~$C_v$. 
	If the second inequality occurs, then Proposition~\ref{prop:RSW_Hugo} implies that
	\begin{align*}
    	\mu_{\Cyl_{C_h n ,C_v n}}^{\srp\srp/\srm\srm} \big[ \calC^h_{\srp\srp}([-2N,2N] \times [3N,4N])\big] \geq \psi(\delta) > 0.
	\end{align*}
	Thus, up to replacing~$\delta$ by~$\psi(\delta)$, we may suppose that \eqref{eq:first_ineq} holds always. 
	Then, by repeated applications of Corollary~\ref{cor:lengthen_crossings} we deduce that
	\begin{align}\label{eq:cyl_long_cross}
    	\mu_{\Cyl_{C_h n ,C_v n}}^{\srp\srp/\srm\srm} \big[ \calC^h_{\srp\srp}([-C_h n,C_h n] \times [3N,4N])\big] \geq \delta_0,
    \end{align}
    for some~$\delta_0 > 0$ depending only on~$C_h$ and~$C_v$. 
	
	
	A consequence of Lemma~\ref{lem:cyl_to_rect} 
	and of the FKG property is that the red-spin marginal of~$\mu_{\Cyl_{C_h n ,C_v n}}^{\srp\srp/\srm\srm}$
	is dominated by that of ~$\mu_{\Rect_{C_h n ,C_v n}}^{\srp\srp/\srm\srm}$. Using this, and the fact that~$N = \frac{C_v}5 n$, we find
	\begin{align}\label{eq:rect_long_cross}
    	&\mu_{\Rect_{C_h n ,C_v n}}^{\srp\srp/\srm\srm}\big[ \calC^h_{\srp\srp}([-C_h n,C_h n] \times [\tfrac35 C_v n,\tfrac45 C_v n])\big]\nonumber\\
	    &\geq \mu_{\Cyl_{C_h n ,C_v n}}^{\srp\srp/\srm\srm} \big[ \calC^h_{\srp\srp}([-C_h n,C_h n] \times [3N,4N])\big] 
	    \geq \delta_0.
    \end{align}
    
    Define the rectangles~$R_j = \Rect_{C_h n, (4/5)^j C_v n}$ and 
   ~$S_j = [-C_h n,C_h n] \times [\frac34 \cdot (\frac45)^{j} C_v n, (\frac45)^{j+1} C_v n]$.
    Then~\eqref{eq:rect_long_cross} applies to any of the rectangles~$R_j$ with~$j \geq 0$, and we find 
    \begin{align}\label{eq:S_j}
	    &\mu_{R_j}^{\srp\srp/\srm\srm} \big[ \calC^h_{\srp\srp}(S_j)\big] \geq \delta_j,
    \end{align}
    for some~$\delta_j > 0$ that depends on~$C_h$,~$C_v$ and~$j$, but not on~$n$ 
    \footnote{We may actually restrict ourselves to~$n$ and~$j$ such that the rectangles~$R_j$ and~$S_j$ do not degenerate below the mesh size. Indeed,~\eqref{eq:S_j} will only be used with~$j \leq \log_{5/4} C_v$ and~$n$ may be assumed large enough.}.
 
	When~$\calC^h_{\srp\srp}(S_j)$ occurs for some~$j \geq 0$, 
	there exists a double-$\rp$ path contained in~$S_j \subset R_{j+1}$, 
    connecting the left and right side of~$R_0$. 
    Any such path separates the top of~$R_0$ from its bottom.
    Let~$\Gamma$ be the highest such path and~$\textsf{Under}(\Gamma)$ be the set of faces of 
   ~$R_0$ that are separated from the top of~$R_0$ by~$\Gamma$.  
    Then~$\Gamma$ is measurable with respect to the spins above and adjacent to~$\Gamma$.
    For any possible realisation~$\gamma$ of~$\Gamma$,
   	due to the Spatial Markov property, 
   	the red-spin marginal of~$\mu_{R_0}^{\srp\srp/\srm\srm}[.|\Gamma = \gamma]$ restricted to~$\textsf{Under}(\gamma)$ 
    stochastically dominates that of~$\mu_{R_{j+1}}^{\srp\srp/\srm\srm}$. 
    It follows that 
    \begin{align}\label{eq:R_j}
    	\mu_{R_0}^{\srp\srp/\srm\srm} \big[\calC^h_{\srp\srp}(S_{j+1})\,\big|\, \Gamma = \gamma\big] 
		\geq \mu_{R_{j+1}}^{\srp\srp/\srm\srm} \big[\calC^h_{\srp\srp}(S_{j+1})\big] \geq \delta_{j+1}.
    \end{align}
    This may appear surprising, as it is not always the case that~$S_{j+1} \subset \textsf{Under}(\gamma)$. 
    Let us explain briefly why \eqref{eq:R_j} is nevertheless true. 
	Couple the red-spin marginals of 
	$\mu_{R_{j+1}}^{\srp\srp/\srm\srm}$ and~$\mu_{R_0}^{\srp\srp/\srm\srm}[.|\Gamma = \gamma]$ in an increasing fashion 
	(this is possible do to the stochastic domination of the former by the latter). 
	Then, if~$(\sigma_r, \tilde \sigma_r)$ is a sample of this coupling, 
	$\tilde\sigma_r$ is equal to~$\rp$ for all faces adjacent to~$\gamma$ and is  
	greater of equal to~$\sigma_r$ for the faces of~$\textsf{Under}(\gamma)$.
	If~$\sigma_r$ is such that~$\calC^h_{\srp\srp}(S_{j+1})$ occurs,
	then~$\tilde\sigma_r \in \calC^h_{\srp\srp}(S_{j+1})$ as well. See Figure~\ref{fig:R_jS_j} for an illustration.  
	
		\begin{figure}
	\begin{center}
    	\includegraphics[width = 0.6\textwidth]{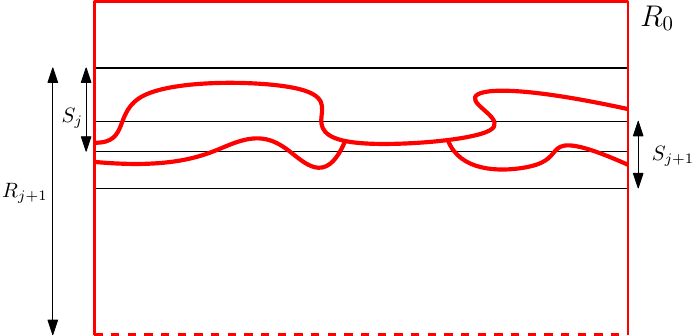}
	  	\caption{The rectangle~$R_0 = \Rect_{C_h n ,C_v n}$ with~$\frp\frp$ boundary conditions on the top and lateral sides 
		and~$\frm\frm$ on the bottom. 
		If~$\calC^h_{\sfrp\sfrp}(S_{j})$ occurs, the measure under~$\Gamma$ dominates 
		$\mu_{R_{j+1}}^{\sfrp\sfrp/\sfrm\sfrm}$ 
		and~$\calC^h_{\sfrp\sfrp}(S_{j+1})$ is more likely to occur than in~$R_{j+1}$.}
		\label{fig:R_jS_j}
	\end{center}
	\end{figure}

	Summing over all possible values of~$\gamma$, we find 
	\begin{align*}
    	\mu_{R_0}^{\srp\srp/\srm\srm} \big[\calC^h_{\srp\srp}(S_{j+1})\,\big|\, \calC^h_{\srp\srp}(S_{j})\big] 
		\geq \mu_{R_{j+1}}^{\srp\srp/\srm\srm} \big[\calC^h_{\srp\srp}(S_{j+1})\big] \geq \delta_{j+1}.
    \end{align*}
   	Iterating this for~$j < J := \lfloor \log_{5/4} C_v \rfloor$, we find
    \begin{align*}
    	\mu_{\Rect_{C_h n ,C_v n}}^{\srp\srp/\srm\srm} \big[ \bigcap_{j=0}^{J}\calC^h_{\srp\srp}(S_j)\big]
		&= \mu_{R_0}^{\srp\srp/\srm\srm} \big[\calC^h_{\srp\srp}(S_0)\big]
		\cdot \prod_{j=1}^{J} \mu_{R_0}^{\srp\srp/\srm\srm} \big[\calC^h_{\srp\srp}(S_{j})\big| \calC^h_{\srp\srp}(S_{j-1})\big] 
		\geq \prod_{j=0}^{J} \delta_j.
    \end{align*}
 	Notice that~$S_J$ is included in~$\Rect_{C_h n,n}$, hence the above implies that
	\begin{align*}
		\mu_{\Rect_{C_h n ,C_v n}}^{\srp\srp/\srm\srm} \big[ \calC^h_{\srp\srp}(\Rect_{C_h n,n})\big]	\geq \prod_{j=0}^{J} \delta_j.
	\end{align*}
	The right-hand side of the above is a positive constant depending only on~$C_h$ and~$C_v$, and the proof is complete. 
\end{proof}

\subsection{Proof of dichotomy theorem (Theorem~\ref{thm:dicho} and Corollary~\ref{cor:dicho})}

\begin{proof}[Theorem~\ref{thm:dicho}]
	\newcommand{\B}{{\sf B}}
	Fix~$n$ and let~$\rho$ be some large constant (we will see below how to choose it). 
	We will work in the domain~$\B := \La_{\rho(\rho+2) n}$, under the measure~$\mu_{\B}^{\srm\srm}$. 
	The steps of the proof are described in Figure~\ref{fig:dicho}.

    \begin{figure}
    \begin{center}
    	\includegraphics[width = 0.45\textwidth, page = 1]{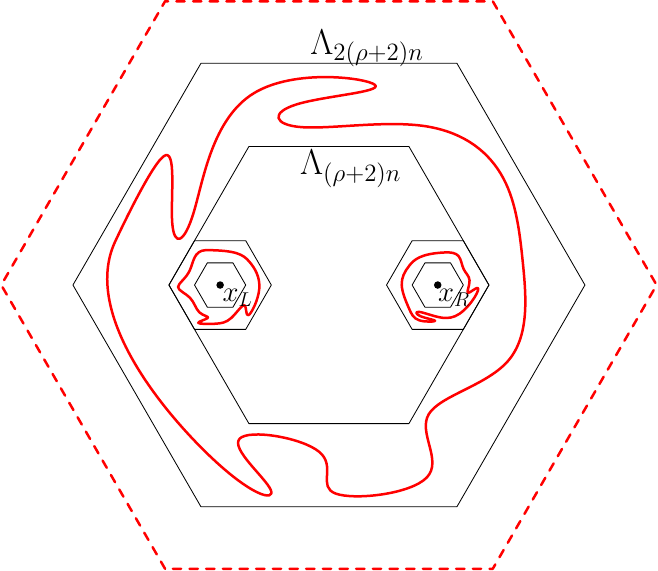}\quad
    	\includegraphics[width = 0.45\textwidth, page = 2]{dicho.pdf}\\
      	\caption{\textbf{Left:} Create~$\Circ_{\sfrp\sfrp}(x_L)$ and~$\Circ_{\sfrp\sfrp}(x_R)$ 
    	by first creating one double-$\frm$ circuit surrounding the whole of~$\La_{(\rho+2)n}$ 
		(at a cost~$\alpha_{(\rho+2)n}$), 
    	then creating two smaller circuits inside~$\La_{(\rho+2)n}$ which come at constant cost. 
    	\textbf{Right:} when~$\Circ_{\sfrp\sfrp}(x_L) \cap \Circ_{\sfrp\sfrp}(x_R)$ occurs, 
    	Corollary~\ref{cor:RSW_mixed_bc} allows us to create two long double-$\frm$ crossings 
    	in the strips above and below~$\La_{2n}(x_L)$.
    	}
    	\label{fig:dicho}
    \end{center}
    \end{figure}
	
	Let~$x_L = (-\rho n,0)$ and~$x_R = (\rho n,0)$. 
	Write~$\La_k(x_L)$ for the ball of radius~$k$ centred at~$x_L$, and use the same notation for~$x_R$.
	Let~$\Circ_{\srp\srp}(x_L)$ and~$\Circ_{\srp\srp}(x_R)$ 
	be the events that there exists a double-$\rp$ circuit in~$\La_{2n}(x_L) \setminus \La_{n}(x_L)$
	and~$\La_{2n}(x_R) \setminus \La_{n}(x_R)$, respectively. 
	Notice that both~$\Circ_{\srp\srp}(x_L)$ and~$\Circ_{\srp\srp}(x_R)$ depend only on the spins inside~$\La_{(\rho + 2)n}$. 
	Recall that~$\Circ_{\srp\srp}(k,\ell)$ is the event that there exists a double-$\rp$ circuit in~$\La_{\ell}$ that surrounds~$\La_{k}$.
	
	By the monotonicity of boundary conditions 
	\begin{align*}
    	&\mu_{\B}^{\srm\srm}\big[\Circ_{\srp\srp}(x_L)\cap\Circ_{\srp\srp}(x_R)\big]\\
    	&\quad\geq \mu_{\B}^{\srm\srm}\big[\Circ_{\srp\srp}(x_L)\cap\Circ_{\srp\srp}(x_R) \,\big|\, \Circ_{\srp\srp}((\rho+2)n,2(\rho+2)n)\big]
    	\cdot \mu_{\B}^{\srm\srm}\big[\Circ_{\srp\srp}((\rho+2)n, 2(\rho+2)n)\big]\\
    	&\quad\geq \mu_{\La_{2(\rho+2)n}}^{\srp\srp}\big[\Circ_{\srp\srp}(x_L)\cap\Circ_{\srp\srp}(x_R)\big] \, \alpha_{(\rho+2)n}\\
	    &\quad\geq \mu_{\La_{2(\rho+2)n}}^{\srp\srp}\big[\Circ_{\srp\srp}(x_L)\big]^2 \, \alpha_{(\rho+2)n}.
	\end{align*}
	In the second inequality we used the monotonicity of boundary conditions (Corollary~\ref{cor:monotonicity_bc}) 
	and the definition of~$\alpha_{(\rho+2)n}$;
	the last inequality is due to the positive association of~$\sigma_r$ under~$\mu_{\La_{2(\rho+2)n}}^{\srp\srp}$.
	It is a standard consequence of the comparison of boundary conditions and Corollary~\ref{cor:RSW_mixed_bc} that 
	$ \mu_{\La_{2(\rho+2)n}}^{\srp\srp}[\Circ_{\srp\srp}(x_L)] > {c_0}$ for some constant~$c_0>0$ that does not depend on~$n$. 
	In conclusion 
	\begin{align}\label{eq:2circ}
		\mu_{\B}^{\srm\srm}\big[\Circ_{\srp\srp}(x_L)\cap\Circ_{\srp\srp}(x_R)\big] \geq c_0^2\, \alpha_{2(\rho+2)n}.
	\end{align}
	
	We will now condition ~$\mu_B^{\srm\srm}$ on the event~$\Circ_{\srp\srp}(x_L)\cap\Circ_{\srp\srp}(x_R)$, 
	and will construct double-$\rm$ circuits around~$\La_{2n}(x_L)$ and~$\La_{2n}(x_R)$. 
	Using the Spatial Markov property, these will allow to bound the probability in~\eqref{eq:2circ} 
	as a product of two probabilities~$\alpha_n$. 
	
	When ~$\Circ_{\srp\srp}(x_L)$ occurs, write~$\Xi_L$ for the innermost double-$\rp$ circuit 
	as in the definition of~$\Circ_{\srp\srp}(x_L)$.
	Then~$\Xi_L$ is measurable in terms of the spins of the faces inside and adjacent to it. 
	Define~$\Xi_R$ in the same way. 
	Let~$\chi_L$ and~$\chi_R$ be two possible realisations of~$\Xi_L$ and~$\Xi_R$, respectively.
	A straightforward variant of the Spatial Markov property (Theorem~\ref{thm:DMP}) 
	states that the restriction of~$\mu_{\B}^{\srm\srm}[. \, |\,  \Xi_L = \chi_L \text{ and }\Xi_R = \chi_R ]$
	to the faces of~$\B$ outside of~$\chi_L$ and~$\chi_R$ is independent of the values of the spins strictly inside~$\chi_L$ and~$\chi_R$. 
	In particular, the restricted measure above is equal to 
	$\mu_{\B}^{\srm\srm}[. \, |\,  \chi_L\equiv \rp \text{ and }\chi_R\equiv\rp ]$, 
	and its red-spin marginal satisfies the FKG inequality (see Corollary~\ref{cor:fkg}).
	
	Consider the horizontal strip~$\Strip_T = \bbR \times [\sqrt3n,2\sqrt3n]$; 
	it sits above~$\La_{2n}(x_L)$ and~$\La_{2n}(x_R)$. 
	Write~$\calC^h_{\srm\srm}(\Strip_T\cap \B)$ for the event that
	$\Strip_T\cap \B$  is crossed horizontally by a double-$\rm$ path
	($\Strip_T\cap \B$ is not technically a rectangle, but we use the same notation).
	Then Corollary~\ref{cor:RSW_mixed_bc} (or rather its variant with~$\rp$ and~$\rm$ inverted) implies the existence of a constant~$c_1> 0$
	independent of~$n$,~$\chi_L$ and~$\chi_R$ such that 
	\begin{align}\label{eq:StripT}
		\mu_{\B}^{\srm\srm}\big[\calC^h_{\srm\srm}(\Strip_T\cap \B) \, \big|\, \Xi_L = \chi_L \text{ and }\Xi_R = \chi_R\big]
		\geq c_1. 
	\end{align}
	Indeed, the red-spin marginal of~$\mu_{\B}^{\srm\srm}\big[. \, \big|\, \Xi_L = \chi_L \text{ and }\Xi_R = \chi_R\big]$ 
	restricted to~$\Strip_T \cap \B$ is dominated by that in the rectangle~$[-\rho(\rho+2) n, \rho(\rho+2) n]\times [\sqrt3n,\rho(\rho+2) n]$
	with boundary conditions~$\rp\rp$ on the bottom and~$\rm\rm$ on all other sides. 
	
	The estimate \eqref{eq:StripT} also holds for~$\Strip_B$, the symmetric of~$\Strip_T$ with respect to the horizontal axis~$\bbR\times\{0\}$. 
	Thus, by the FKG inequality,
	\begin{align*}
		\mu_{\B}^{\srm\srm}\big[\calC^h_{\srm\srm}(\Strip_T\cap \B) \cap \calC^h_{\srm\srm}(\Strip_L\cap \B)  \, \big|\, \Xi_L = \chi_L \text{ and }\Xi_R = \chi_R\big]
		\geq c_1^2. 
	\end{align*}
	Summing over all possible values~$\chi_L$ and~$\chi_R$ of~$\Xi_L$ and~$\Xi_R$ and using~\eqref{eq:2circ}, 
	we find
	\begin{align*}
		\mu_{\B}^{\srm\srm}\big[\calC^h_{\srm\srm}(\Strip_T\cap \B)\cap \calC^h_{\srm\srm}(\Strip_B \cap \B) \cap\Circ_{\srp\srp}(x_L)\cap\Circ_{\srp\srp}(x_R)\big]
		\geq c_1^2 \,c_0\, \alpha_{(\rho+2)n} . 
	\end{align*}
	As a consequence 
	\begin{align*}
		\mu_{\B}^{\srm\srm}\big[\Circ_{\srp\srp}(x_L)\cap\Circ_{\srp\srp}(x_R)\,\big|\, \calC^h_{\srm\srm}(\Strip_T\cap \B)\cap \calC^h_{\srm\srm}(\Strip_B \cap \B) \big]
		\geq c_1^2 \,c_0^2\, \alpha_{(\rho+2)n}. 
	\end{align*}
	
	\newcommand{\D}{{\sf D}}
	Write~$\D$ for the domain that is the intersection of~$\B$ with the strip~$\bbR \times [-2\sqrt3 n,2\sqrt3 n]$. 
	Then, by conditioning on the highest and lowest double-$\rm$ crossings of~$\Strip_T$ and~$\Strip_B$, respectively, 
	using the spatial Markov property and the monotonicity of boundary conditions, we find
	\begin{align*}
    	&\mu_{\D}^{\srm\srm}\big[\Circ_{\srp\srp}(x_L)\cap\Circ_{\srp\srp}(x_R)\big]\\
    	&\quad \geq \mu_{\B}^{\srm\srm}\big[\Circ_{\srp\srp}(x_L)\cap\Circ_{\srp\srp}(x_R)
			\,\big|\, \calC^h_{\srm\srm}(\Strip_T\cap \B)\cap \calC^h_{\srm\srm}(\Strip_B \cap \B) \big]
		\geq c_1^2 \,c_0^2\, \alpha_{(\rho+2)n}. 
	\end{align*}
	
	\begin{figure}
	\begin{center}
    	\includegraphics[width = 0.9\textwidth]{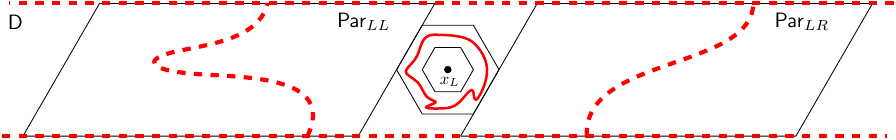}\quad
    	\caption{The rectangles~$\Par_{LL}$ and~$\Par_{LR}$ are to the left and right of~$\La_{2n}(x_L)$, respectively,
    	and their top and bottom is part of the boundary of~$\D$. Due to the~$\frm\frm$ boundary conditions on~$\D$, 
    	Lemma~\ref{lem:double_cross} applies to both these parallelograms. 
    	Thus, each is crossed vertically by a double-$\frm$ with uniform positive probability, 
    	independent of the configuration in the rest of~$\D$.}
    	\label{fig:dicho2}
	\end{center}
	\end{figure}
	
	Now consider the parallelogram~$\Par$ formed of the faces with centres 
	at~$k + \ell e^{i\pi/3}$ with~$0 \leq k \leq 24n$ and~$-4n \leq \ell \leq 4n$.
	Define its horizontal translates~$\Par_{LL} = \Par - ((\rho + 26)n, 0)$,~$\Par_{LR} = \Par - ((\rho -2)n, 0)$,
	$\Par_{RL} = \Par + ((\rho -26)n, 0)$ and~$\Par_{RR} = \Par + ((\rho +2)n, 0)$. 
	These are all contained in~$\D$, touch its top and bottom and are left of~$\La_{2n}(x_L)$, right of~$\La_{2n}(x_L)$, 
	left of~$\La_{2n}(x_R)$ and right of~$\La_{2n}(x_r)$, respectively. 
	Let us assume that~$\rho$ is large enough so that ~$\Par_{LL}$ and~$\Par_{LR}$ are included in~$\La_{\rho n}(x_L)$ 
	($\rho \geq 30$ suffices). 
	Then~$\Par_{RL}$ and~$\Par_{RR}$ are included in~$\La_{\rho n}(x_R)$ and, in particular, are disjoint from the first two parallelograms. 
	
	Now observe that, due to Lemma~\ref{lem:double_cross} (applied with~$\rm$ and~$\rp$ exchanged) 
	\begin{align*}
		&\mu_{\D}^{\srm\srm}\big[
			\calC_{\srm\srm}^{v}(\Par_{LL})\cap\calC_{\srm\srm}^{v}(\Par_{LR})
			\cap\calC_{\srm\srm}^{v}(\Par_{RL})\cap\calC_{\srm\srm}^{v}(\Par_{RR}) 
			\,\big|\, 
			\Circ_{\srp\srp}(x_L)\cap\Circ_{\srp\srp}(x_R)\big] 
		\geq c_2, 
	\end{align*}
	for some universal constant~$c_2 >0$.
	Then, using Bayes rule
	\begin{align*}
		&\mu_{\D}^{\srm\srm}\big[\Circ_{\srp\srp}(x_L)\cap\Circ_{\srp\srp}(x_R) \,\big|\, 
			\calC_{\srm\srm}^{v}(\Par_{LL})\cap\calC_{\srm\srm}^{v}(\Par_{LR})
			\cap\calC_{\srm\srm}^{v}(\Par_{RL})\cap\calC_{\srm\srm}^{v}(\Par_{RR})\big]\\
    	&\qquad \geq 
    	\mu_{\D}^{\srm\srm}\big[\calC_{\srm\srm}^{v}(\Par_{LL})\cap\calC_{\srm\srm}^{v}(\Par_{LR})
			\cap\calC_{\srm\srm}^{v}(\Par_{RL})\cap\calC_{\srm\srm}^{v}(\Par_{RR})  \,\big|\, 
				\Circ_{\srp\srp}(x_L)\cap\Circ_{\srp\srp}(x_R)\big] \\
		&\qquad  \qquad \times \mu_{\D}^{\srm\srm}\big[\Circ_{\srp\srp}(x_L)\cap\Circ_{\srp\srp}(x_R)\big] \\
		& \qquad \geq c_2\,c_1^2\,c_0^2\, \alpha_{(\rho+2)n}.
	\end{align*}
	Finally, by conditioning on the left-most vertical double-$\rm$ crossing of~$\Par_{LL}$ and the right-most of~$\Par_{LR}$, 
	and using the monotonicity of boundary conditions,
	it may be shown that the restriction of~$\mu_{\D}^{\srm\srm}[. \,|\, \calC_{\srm\srm}^{v}(\Par_{LL})\cap\calC_{\srm\srm}^{v}(\Par_{LR})]$ 
	to~$\La_{2n}(x_L)$ is dominated by that of~$\mu_{\La_{\rho n}(x_L)}^{\srm\srm}$. 
	Moreover, due to the Spatial Markov property, this is true even when conditioning on the spins to the right of~$\Par_{LR}$. 
	The same procedure may be applied to~$\mu_{\D}^{\srm\srm}[. \,|\,\calC_{\srm\srm}^{v}(\Par_{RL})\cap\calC_{\srm\srm}^{v}(\Par_{RR})]$  for the measure in~$\La_{2n}(x_R)$. 
	Notice that the areas that determine the restriction of~$\mu_{\D}^{\srm\srm}[. \,|\,\calC_{\srm\srm}^{v}(\Par_{LL})\cap\calC_{\srm\srm}^{v}(\Par_{LR})\cap\calC_{\srm\srm}^{v}(\Par_{RL})\cap\calC_{\srm\srm}^{v}(\Par_{RR})]$ to~$\La_{2n}(x_L)$ and~$\La_{2n}(x_R)$ are disjoint.
	Thus, the restriction of~$\mu_{\D}^{\srm\srm}[. \,|\,\calC_{\srm\srm}^{v}(\Par_{LL})\cap\calC_{\srm\srm}^{v}(\Par_{LR})\cap\calC_{\srm\srm}^{v}(\Par_{RL})\cap\calC_{\srm\srm}^{v}(\Par_{RR})]$ to~$\La_{2n}(x_L)\cap \La_{2n}(x_R)$ is dominated by the independent product of~$\mu_{\La_{\rho n}(x_L)}^{\srm\srm}$ and~$\mu_{\La_{\rho n}(x_R)}^{\srm\srm}$.
	In conclusion, 
	\begin{align*}
		&\mu_{\D}^{\srm\srm}\big[\Circ_{\srp\srp}(x_L)\cap\Circ_{\srp\srp}(x_R) \,\big|\, 
			\calC_{\srm\srm}^{v}(\Par_{LL})\cap\calC_{\srm\srm}^{v}(\Par_{LR})
			\cap\calC_{\srm\srm}^{v}(\Par_{RL})\cap\calC_{\srm\srm}^{v}(\Par_{RR})\big]\\
		&\leq \mu_{\La_{\rho n}(x_L)}^{\srm\srm}\big[\Circ_{\srp\srp}(x_L)\big] 
		\cdot \mu_{\La_{\rho n}(x_R)}^{\srm\srm}\big[\Circ_{\srp\srp}(x_R)\big] 
		= \alpha_n^2.
	\end{align*}
	The last two displayed equations yield the desired conclusion.
\end{proof}

\begin{proof}[Corollary~\ref{cor:dicho}]
	Let~$\rho,C$ be the constants of Theorem~\ref{thm:dicho}.  
	Suppose that~$\inf_n \alpha_n =0$. 
	Let~$n_0$ be such that~$\alpha_{n_0} \leq \frac{1}{2C}$. 
	Then a simple induction involving~\eqref{eq:recurrence} implies that~$\alpha_{(\rho +2)^k n_0} \leq \frac1C 2^{-2^k}$ for all~$k \geq 0$.
	This implies the stated inequality for~$c < \log2 / \log(\rho+2)$ and~$C$ chosen accordingly. 
\end{proof}

\section{Conclusions}\label{sec:macro}

In this section we prove Theorems~\ref{thm:loops} and~\ref{thm:Gibbs}. 
To this end, we first resolve the dichotomy stated in Corollary~\ref{cor:dicho} 
and then transfer the results from the spin representation to the loop~$O(2)$ model.

\subsection{Excluding stretched-exponential decay}\label{sec:no_exp_dec}

The goal of this section is to show that the case \textit{(ii)} of Corollary~\ref{cor:dicho} is incoherent with Theorem~\ref{thm:uniqueness_nu}. 
Once it is established that case \textit{(i)} holds, it is  fairly standard to deduce  Theorem~\ref{thm:loops}; 
this is done in Section~\ref{sec:proof_polynomial_decay}.

\begin{prop}\label{prop:dicho-resolve}
	Case \textit{(i)} of Corollary~\ref{cor:dicho} occurs. 
	That is,~$\inf_n \mu_{\La_{\rho n}}^{\srm\srm}\big[\Circ_{\srp\srp}(n,2n)\big] > 0$ for some fixed constant~$\rho > 2$.
\end{prop}

The constant~$\rho$ and the ratio between the inner and outer radii of the annulus above may actually be chosen arbitrarily, as we prove below. 
Other variants referring to rectangle crossings may also be formulated. 

\begin{cor}\label{cor:RSW_strong}
	For any~$a > 1$, 
	\begin{align*}
		\inf\big\{ \mu_{\La_{a n}}^{\srm\srm}\big[\Circ_{\srp\srp}(n,a n)\big]\,:\, n \geq \tfrac{3}{a-1}\big\} > 0.
	\end{align*}
\end{cor}

The lower bound on~$n$ in the infimum above is to ensure that the annulus is thick enough to allow the existence of a double-$\rp$ circuit. 
We start by proving the corollary, based on Proposition~\ref{prop:dicho-resolve}.
The remainder of the section is then dedicated to proving Proposition~\ref{prop:dicho-resolve}. 

\begin{proof}
	This is a standard application of Proposition~\ref{prop:dicho-resolve}, the FKG property and the monotonicity of boundary conditions. 
	
	Fix~$a > 1$ and let~$b = (1+ a)/2$.
	We may limit ourselves to values of~$n$ larger than some threshold depending on~$a$; 
	smaller values of~$n$ only add strictly positive numbers to the set whose infimum we are considering. 
	
	Recall that~$\rho$ is fixed by Theorem~\ref{thm:dicho}.
	Let~$m = \lfloor \min\{ \frac{a-b}{\rho};  \frac {b-1}4\}\cdot  n\rfloor$, 
	and suppose that~$n$ is large enough so that~$m \geq 2$. 
	Then there exists a number~$K = K(a,\rho)$, not depending on~$m$ or~$n$ 
	such that one may place~$K$ translates~$\Ann_1,\dots, \Ann_K$ 
	of the annulus~$\La_{2 m } \setminus \La_m$ 
	inside~$\La_{bn} \setminus \La_n$ in such a way that, if all of them contain a circuit of double-$\rp$, 
	then~$\Circ_{\srp\srp}(n,an)$ occurs. See Figure~\ref{fig:FKG_annuli} for an example. 
	
	\begin{figure}
	\begin{center}
	\includegraphics[width = 0.6\textwidth]{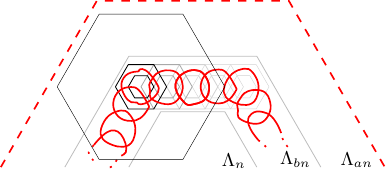}
	\caption{The small annuli~$\Ann_1,\dots,\Ann_K$ placed inside~$\La_{bn} \setminus \La_n$ 
	have inner radius~$m$ and outer radius~$2m$. 
	They are such that the balls of radius~$\rho m$ around each of their centres are contained in~$\La_{an}$. 
	If they all contain double-$\frp$ circuits, then these form a circuit around~$\La_n$, contained in~$\La_{bn} \subset \La_{an}$. }
	\label{fig:FKG_annuli}
	\end{center}
	\end{figure}
	
	Since~$m \rho < (a -b)n$, all faces at distance~$\rho m$ from each~$\Ann_j$ are contained in~$\La_{an}$. 
	It follows from the FKG inequality and the monotonicity of boundary conditions that 
	\begin{align*}
		\mu_{\La_{a n}}^{\srm\srm}\big[\Circ_{\srp\srp}(n,a n)\big] 
		\geq \prod_{j=1}^K\mu_{\La_{a n}}^{\srm\srm}\big[\Circ_{\srp\srp}(\Ann_j) \big] 
		\geq \Big(\mu_{\La_{\rho m}}^{\srm\srm}\big[\Circ_{\srp\srp}(m,2m) \big] \Big)^K 
		> c^K,
	\end{align*}
	where~$c= \inf_n \mu_{\La_{\rho n}}^{\srm\srm}[\Circ_{\srp\srp}(n,2n)]$ 
	is a strictly positive constant due to Proposition~\ref{prop:dicho-resolve}.
	Since the ultimate lower bound above does not depend on~$n$, the proof is complete. 
\end{proof}

We now turn to proving Proposition~\ref{prop:dicho-resolve}.
We will proceed by contradiction. Fix~$\rho > 2$ given by Theorem~\ref{thm:dicho} and 
recall that~$\alpha_n = \mu_{\La_{\rho n}}^{\srm\srm}\big(\Circ_{\srp\srp}(n,2n))$. 
Will assume that case \textit{(ii)} of Corollary~\ref{cor:dicho} occurs, namely that 
there exist constants~$c,C>0$ and~$n_0 \geq 1$ such that 
\begin{align}
	\alpha_n \leq C e^{-n^c} \qquad \text{ for all~$n = (\rho+2)^k n_0$ with~$k\in \bbN$.}
	\tag{ExpDec}\label{assumption:exp_decay}
\end{align}
We start by proving a series of results based on~\eqref{assumption:exp_decay}. 
All constants below depend implicitly on the values of~$n_0$,~$\rho$,~$c$ and~$C$ of~\eqref{assumption:exp_decay}. 

\begin{lem}\label{lem:no_red_plus}
	Under assumption~\eqref{assumption:exp_decay}, for any~$\kappa \geq 2$ there exists~$C_1 = C_1 (\kappa) > 0$ such that
	\begin{align}\label{eq:no_red_plus}
		\mu_{\La_{\kappa n}}^{\srm\srm}\big[\La_n\xlra{\srp\srp}\La_{2n}^c \big] < e^{-C_1 n^{c}} \qquad \forall n \geq 1.
	\end{align}
\end{lem}

\begin{proof}
	Fix~$\kappa \geq2$ and~$n$ arbitrary. 
	Let~$\R := \Rect_{2n, n/2}$.
	The annulus~$\La_{2n}\setminus \La_n$ may be covered by six translations and rotation~$R_1,\dots, R_6$ of~$\R$ 
	in such a way that, if~$\{\La_n\xlra{\srp\srp}\La_{2n}^c\}$ occurs, then at least one of~$R_1,\dots, R_6$
	is crossed in the short direction by a double-$\rp$ path (see Figure~\ref{fig:Rect_to_ann}).
	For~$1 \leq i \leq 6$, write~$\calC_{\srp\srp}^v(R_i)$ for the appropriate rotation and translation of~$\calC_{\srp\srp}^v(\R)$. 
	Then, using the union bound and the monotonicity of boundary conditions, we deduce that 
	\begin{align}\label{eq:circ_to_path}
		\mu_{\La_{\kappa n}}^{\srm\srm}\big[\La_n\xlra{\srp\srp}\La_{2n}^c \big] 
		\leq \mu_{\La_{\kappa n}}^{\srm\srm}\Big[\bigcup_{i=1}^6\calC_{\srp\srp}^v(R_i)\Big] 
		\leq 6 \mu_{\La_{(\kappa+2) n}}^{\srm\srm}\big[\calC_{\srp\srp}^v(\R) \big].
	\end{align}
	Henceforth we aim to prove a stretched-exponential upper bound for~$\mu_{\La_{(\kappa+2) n}}^{\srm\srm}[\calC_{\srp\srp}^v(\R)]$. 
	
	\begin{figure}
	\begin{center}
	\includegraphics[width=0.33\textwidth]{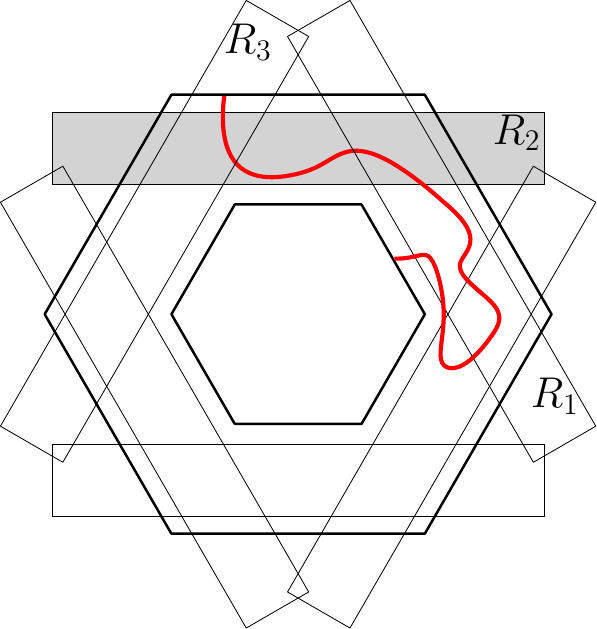}	\qquad
	\includegraphics[width=0.6\textwidth]{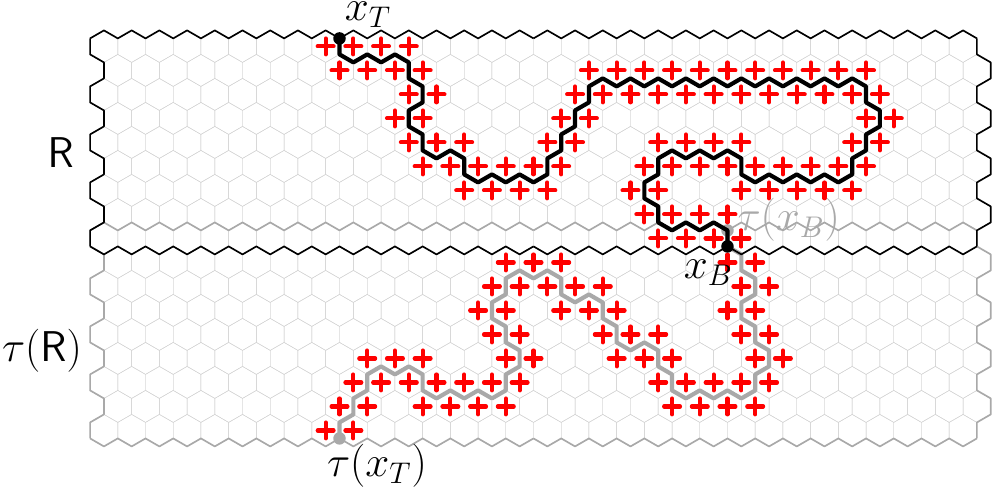}
	\caption{{\em Left:} One may place six copies of~$\R$ around~$\La_n$ so that, if~$\{\La_n\xlra{\sfrp\sfrp}\La_{2n}^c\}$ occurs, 
	then at least one of them is crossed in the short direction (here the crossed copy is shaded). 
	{\em Right:} When~$x_T$ is connected to~$x_B$ in~$\R$ and~$\tau(x_T)$ is connected to~$\tau(x_B)$ in~$\tau(\R)$, 
	then~$\tau(x_T)$ is connected to~$x_T$ in~$\tau(\R)\cup\R$. 
	}
	\label{fig:Rect_to_ann}
	\end{center}
	\end{figure}

	Let~$x_B$ and~$x_T$ to be the points of the bottom and top, respectively, of~$\R$
	that are most probable under~$\mu_{\La_{(\kappa + 2) n}}^{\srm\srm}$ to be connected by a double-$\rp$ path contained in~$\R$. 
	Then 
	\begin{align}\label{eq:2no_red_plus}
		\mu_{\La_{(\kappa + 2) n}}^{\srm\srm}\big(x_B\xlra{\srp\srp \text{ in }\R}x_T\big)
		\geq \tfrac{1}{16n^2}\mu_{\La_{(\kappa + 2) n}}^{\srm\srm}[\calC_{\srp\srp}^v(\R)],
	\end{align}
	since there are~$16n^2$ potential pairs of points~$(x_L,x_R)$. 
	Let~$\tau$ be the reflection with respect to the horizontal line~$\bbR \times \{0\}$. 
	Then we also have 
	\begin{align}\label{eq:3no_red_plus}
		\mu_{\La_{(\kappa + 2) n}}^{\srm\srm}\big[\tau(x_B)\xlra{\srp\srp \text{ in }\tau(\R)}\tau(x_T)\big]
		\geq \tfrac{1}{(4n)^2}\mu_{\La_{(\kappa + 2) n}}^{\srm\srm}[\calC_{\srp\srp}^v(\R)].
	\end{align}	
	If the events of~\eqref{eq:2no_red_plus} and~\eqref{eq:3no_red_plus} occur simultaneously,  then~$x_T$ and~$\tau(x_T)$ are connected 
	inside~$\R \cup \tau(\R) = [-2n,2n] \times [-n/2,n/2]$ (see Figure~\ref{fig:Rect_to_ann}). Thus, by the FKG inequality, 
	\begin{align*}
		\mu_{\La_{(\kappa + 2) n}}^{\srm\srm}\big[\tau(x_T) \xlra{\srp\srp \text{ in }\tau(\R) \cup \R}x_T\big]
		\geq \tfrac{1}{(4n)^4}\mu_{\La_{(\kappa + 2) n}}^{\srm\srm}[\calC_{\srp\srp}^v(\R)]^2.
	\end{align*}
	Using the above, the FKG inequality again and the monotonicity of boundary conditions, we find, 
	\begin{align}\label{eq:28}
		\mu_{\La_{(\kappa + 10) n}}^{\srm\srm}[\calC_{\srp\srp}^v([-2n,2n] \times [-8n,8n])]
		\geq \tfrac{1}{(4n)^{64}}\mu_{\La_{(\kappa + 2) n}}^{\srm\srm}[\calC_{\srp\srp}^v(\R)]^{32}.
	\end{align}
	Indeed, a vertical crossing of~$[-2n,2n] \times [-8n,8n]$ may be obtained by intersecting sixteen translates of the event 
	$\{\tau(x_T) \xlra{\srp\srp \text{ in }\tau(\R) \cup \R}x_T\}$. 
	The box has been increased to~$\La_{(\kappa + 10) n}$ 
	so that all of these events occur in rectangles with distance to the boundary greater than~$(\kappa + 2) n$.

	\begin{figure}
    	\begin{center}
    		\includegraphics[width=0.6\textwidth]{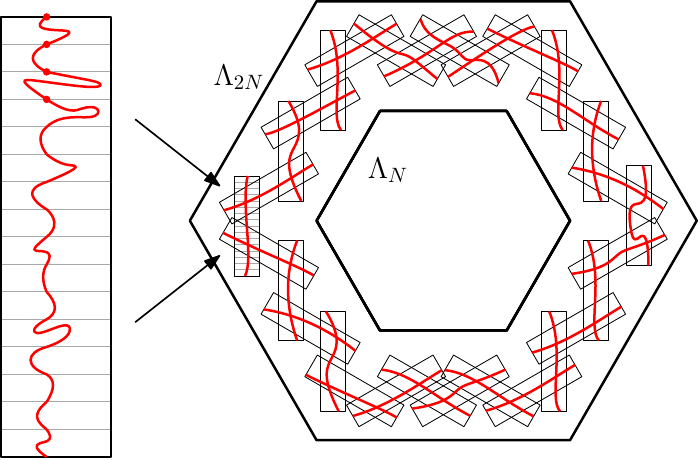}
        	\caption{Combining~$K$ vertical crossings of (rotations and translations of)~$[-2n,2n] \times [-8n,8n]$
        	produces a circuit in~$\La_{2N} \setminus \La_N$. 
        	Each such vertical crossing is constructed from vertical crossings between~$\tau(x_T)$ and~$x_T$ in 
        	sixteen copies of~$\R\cup\tau(\R)$.	}
        	\label{fig:Rect_to_ann2}
    	\end{center}
	\end{figure}

	Recall the fixed values~$\rho >2$ and~$n_0$ of \eqref{assumption:exp_decay}. 
	Let~$k$ be minimal such that, for~$N := (\rho+2)^k n_0$, one has
	\[
		N/n \geq \max \big\{16,  \tfrac{\kappa + 10}{\rho-2}\big\}.
	\]
	By the minimality of~$k$, we have~$N \leq c_0 n$ for some constant~$c_0$ depending on~$\rho$,~$n_0$ and~$\kappa$ only.
	Then, there exists a constant~$K = K(\rho, n_0, \kappa)$, that depends on~$\rho$,~$n_0$ and~$\kappa$ 
	but not on~$n$ or on the resulting value of~$N$,  
	such that one may construct a circuit in~$\La_{2N} \setminus \La_{N}$
	by combining at most~$K$ vertical crossings of translates of~$[-2n,2n] \times [-8n,8n]$ 
	and rotations by~$2\pi/3$ and~$4\pi/3$ of this rectange, 
	all contained in~$\La_{2N}$ (see Figure~\ref{fig:Rect_to_ann2}).
	Due to the choice of~$N$, the faces at distance~$(\kappa + 10)n$ from any of these rectangles are all contained in~$\La_{\rho N}$. 	
	Thus, by the monotonicity of boundary conditions and the FKG inequality, 
	\begin{align*}
		\alpha_N =  \mu_{\La_{\rho N}}^{\srm\srm}\big[\Circ_{\srp\srp}(2N,N)\big]
    	& \geq \mu_{\La_{(\kappa + 10) n}}^{\srm\srm}\big[\calC_{\srp\srp}^v([-2n,2n] \times [-8n,8n])\big]^{K}\\
    	& \geq \tfrac{1}{(4n)^{64 \cdot K}}\,\mu_{\La_{\kappa n}}^{\srm\srm}\big[\calC_{\srp\srp}^v(\R)\big]^{32\cdot K}.
	\end{align*}
	Due to~\eqref{assumption:exp_decay}, this implies 
	\begin{align*}
		\mu_{\La_{\kappa n}}^{\srm\srm}[\calC_{\srp\srp}^v(R)] 
		\leq (4n)^{2}\,\big( C e^{-N^c}\big)^{\frac1{32\cdot K}}
		\leq c_2 (4n)^{2} e^{-c_1 n^c}, 
	\end{align*}
	for constants~$c_1,c_2 > 0$ depending only on~$\kappa$,~$\rho$ and~$n_0$. 
	Finally, from \eqref{eq:circ_to_path} we deduce that
	\begin{align*}
		\mu_{\La_{\kappa n}}^{\srm\srm}\big[\La_n\xlra{\srp\srp}\La_{2n}^c \big] 
		\leq 6 c_2 (4n)^{2} e^{-c_1 n^c}.
	\end{align*}
	This implies~\eqref{eq:no_red_plus} with~$C_1$ chosen small enough to absorb the multiplicative factor. 
\end{proof}

\begin{lem}\label{lem:yes_red_minus}
	Under assumption~\eqref{assumption:exp_decay}, for any~$\kappa \geq 2$ there exists~$C_2 = C_2 (\kappa) > 0$ such that
	\begin{align}\label{eq:yes_red_minus}
		\mu_{\La_{\kappa n}}^{\srm\srm}\big[\La_n\nxlra{\srm\srm}\La_{2n}^c \big] < e^{-C_2 n^{c}} \qquad \forall n \geq 1.
	\end{align}
\end{lem}

\begin{proof}
	Fix~$\kappa \geq 2$ and let~$C_2$ be some small constant to be fixed below 
	(it will be obvious that the bound on~$C_2$ depends only on~$\kappa$).
	It suffices to prove the statement for~$n$ large enough; small values may be incorporated by adjusting~$C_2$. 
	
	Suppose by contradiction that there exists~$n \geq1$ (large) such that  
	$\mu_{\La_{\kappa n}}^{\srm\srm}[\La_n\xlra{\srm\srm}\La_{2n}^c ] < 1 - e^{-C_2 n^{c}}.$
	From now on~$n$ is fixed and it is crucial that we use the assumption above only for this particular value of~$n$. 
	
	Due to the monotonicity of boundary conditions, we deduce that 
	\begin{align*}
		\mu_{\calD}^{\srm\srm}\big[\La_n\xlra{\srm\srm}\La_{2n}^c \big] <1- e^{-C_2 n^{c}}
	\end{align*}
	for any domain~$\calD$ containing~$\La_{\kappa n}$.
	Suppose now that~$C_2$ is chosen smaller than~$C_1(4\kappa)/2$, with~$C_1(4\kappa)$ given by Lemma~\ref{lem:no_red_plus}. 
	Then, assuming~$n$ is above some threshold (which we will do from now on), we have
	\begin{align}\label{eq:no_red_plus2}
		\mu_{\La_{4\kappa n}}^{\srm\srm}\big[\La_n\xlra{\srp\srp}\La_{2n}^c \big] < e^{-C_1(4\kappa) n^{c}} \leq \tfrac12 e^{-C_2 n^{c}}.
	\end{align}
	Due to the two displays above, and to the monotonicity of boundary conditions, 
	\begin{align*}
		\mu_{\calD}^{\srm\srm}\big[\{\La_n\nxlra{\srp\srp}\La_{2n}^c\} \cap \{\La_n\nxlra{\srm\srm}\La_{2n}^c\} \big] 
		\geq 1 - 		\mu_{\calD}^{\srm\srm}\big[\La_n\xlra{\srp\srp}\La_{2n}^c\big]
			- \mu_{\calD}^{\srm\srm}\big[\La_n\xlra{\srm\srm}\La_{2n}^c\big]  
		\geq \tfrac12 e^{-C_2 n^{c}},
	\end{align*}
	for any domain~$\calD$ with~$\La_{\kappa n} \subset \calD \subset \La_{4\kappa n}$. 
	
	As in Lemma~\ref{lem:monochrom} , the absence of a double-$\rp$ or double-$\rm$ connection between 
	$\La_n$ and~$\La_{2n}^c$ implies that at least one of~$\Circ_{\sbm\sbm}(n, 2n)$ and~$\Circ_{\sbp\sbp}(n, 2n)$ occurs (see also Remark~\ref{rem:monochrom}). 
	Under~$\mu_{\calD}^{\srm\srm}$, the blue spins are interchangeable, and we deduce that 
	\begin{align}\label{eq:yes_blue_circ}
		\mu_{\calD}^{\srm\srm}\big[\Circ_{\sbp\sbp}(n, 2n)\big] = 
		\mu_{\calD}^{\srm\srm}\big[\Circ_{\sbm\sbm}(n, 2n)\big]  \geq \tfrac14 e^{-C_2 n^{c}},
	\end{align}
	for any domain~$\calD$ with~$\La_{\kappa n} \subset \calD \subset \La_{4\kappa n}$. 
	
	Next we work in the domain~$\La_{3\kappa n}$.
	Place translations~$\Ann_1,\dots, \Ann_K$ of the annulus~$\La_{2n} \setminus \La_n$ 
	around the outside of~$\La_{\kappa n}$ as in Figure~\ref{fig:FKG_annuli} 
	so that, if all of them contain double-$\bm$ circuits, then there exists a double-$\bm$ circuit 
	in~$\La_{(\kappa+2)n}$ surrounding~$\La_{\kappa n}$.
	As discussed in the proof of Corollary~\ref{cor:RSW_strong}, this procedure employs a number~$K$ of translates 
	that only depends on~$\kappa$, not on~$n$. 
	Thus
	\begin{align}\label{eq:yes_blue_circ2}
		\mu_{\La_{3\kappa n}}^{\srm\srm}\big[\Circ_{\sbm\sbm}(\kappa n, (\kappa + 2)n)\big] 
		&\geq \mu_{\La_{3\kappa n}}^{\srm\srm}\Big[\bigcap_{k = 1}^K\Circ_{\sbm\sbm}(\Ann_k)\Big]\nonumber \\
		&\geq \prod_{k = 1}^K\mu_{\La_{3\kappa n}}^{\srm\srm}\big[\Circ_{\sbm\sbm}(\Ann_k)\big] 
		\geq \Big( \tfrac14 e^{-C_2 n^{c}} \Big)^K.
	\end{align}
	The second inequality is due to the FKG property of blue spins under~$\mu_{\La_{3\kappa n}}^{\srm\srm}$ 
	(see Remark~\ref{rem:FKG} \textit{(i)} with reversed colours).
	The last inequality is a consequence of \eqref{eq:yes_blue_circ}.
	Indeed, if~$\La_{3\kappa n}$ is translated by the translation that sends~$\Ann_k$ to~$\La_{2n} \setminus \La_n$, 
	then it contains~$\La_{\kappa n}$ and is contained in~$\La_{4\kappa n}$, hence \eqref{eq:yes_blue_circ} applies to it.

	When~$\Circ_{\sbm\sbm}(\kappa n, (\kappa + 2)n)$ occurs, write~$\Gamma$ for the exterior most double-$\bm$ circuit in 
	$\La_{(\kappa + 2)n}$ that surrounds~$\La_{\kappa n}$. 
	Let~$\gamma$ be a possible realisation of~$\Gamma$. 
	Due to the Spatial Markov property, 
	the measure~$\mu_{\La_{3\kappa n}}^{\srm\srm}[.\,|\, \Gamma = \gamma]$ restricted to the interior~$\Int(\gamma)$ 
	of~$\gamma$ is simply~$\mu_{\Int(\gamma)}^{\sbm\sbm}$. 
	Then \eqref{eq:yes_blue_circ} with inverted colours applies to the domain~$\Int(\gamma)$, 
	and we find 
	\begin{align*}
		\mu_{\La_{3\kappa n}}^{\srm\srm}\big[\Circ_{\srp\srp}(n, 2n)\,|\, \Gamma =\gamma\big] = 
		\mu_{\Int(\gamma)}^{\sbm\sbm}\big[\Circ_{\srp\srp}(n, 2n)\big] \geq \tfrac14 e^{-C_2 n^{c}}.
	\end{align*}
	Averaging the above over all possible values~$\gamma$ taken by~$\Gamma$ and using \eqref{eq:yes_blue_circ2},  we find
	\begin{align*}
		\mu_{\La_{3\kappa n}}^{\srm\srm}\big[\Circ_{\srp\srp}(n, 2n)\big] 
		&\geq 
		\mu_{\La_{3\kappa n}}^{\srm\srm}\big[\Circ_{\srp\srp}(n, 2n) \cap \Circ_{\sbm\sbm}(\kappa n, (\kappa + 2)n)\big]\\
		&\geq \sum_{\gamma}\mu_{\La_{3\kappa n}}^{\srm\srm}\big[\Circ_{\srp\srp}(n, 2n)\,|\, \Gamma =\gamma\big] 
		\cdot \mu_{\La_{3\kappa n}}^{\srm\srm}(\Gamma =\gamma)
		\geq \Big(\tfrac14 e^{-C_2 n^{c}}\Big)^{K+1}.
	\end{align*}
	This contradicts~\eqref{eq:no_red_plus} provided that~$C_2$ is small enough (any~$C_2 <C_1(3\kappa) / (K+1)$ suffices) 
	and~$n$ is large enough. 
\end{proof}

\begin{lem}\label{lem:yes_red_plus}
	Under assumption~\eqref{assumption:exp_decay}, for any~$\kappa \geq2$ there exists~$C_3 = C_3 (\kappa) > 0$ such that
	\begin{align}\label{eq:C3}
		\mu_{\La_{\kappa n}}^{\srm\srm}[\Circ_{\srm\srm}(n,2n)] > 1 - e^{-C_3 n^{c}}, \qquad \forall n \geq 1.
	\end{align}
	As a consequence, for any~$\epsilon > 0$, there exists~$n_1 \geq 1$ such that 
	\begin{align}\label{eq:C31}
		\mu_{\bbH}\Big[\bigcap_{j\geq0}\Circ_{\srm\srm}(2^{j}n_1,2^{j+1}n_1)\Big] > 1 - \eps.
	\end{align}
\end{lem}

\begin{proof}
	Fix some~$\kappa \geq 2$. 
	Let us first prove that 
	\begin{align}\label{eq:C4}
		\mu_{\La_{(\kappa + 2) n}}^{\srm\srm}[\calC^v_{\srm\srm}(\Rect_{n/16, n/2})] > 1 - e^{-C_4 n^{c}}, 
	\end{align}
	for all~$n$ and some fixed constant~$C_4 > 0$. 
	Notice that we are aiming to show that a thin rectangle is crossed in the long (vertical) direction with very high probability.
	Heuristically, Lemma~\ref{lem:yes_red_minus} 
	says that such rectangles are crossed with high probability in the short (i.e. horizontal) direction. 
	To pass from crossing in the short direction to crossings in the long direction,
	we will use the same argument as in the proof of \eqref{eq:28}. 
	However, since this argument applies to small probabilities rather than large ones, 
	we will use it for the model dual to double-$\rm$ connections.

	Recall the notation~$\dbm(\sigma_r)$ for the set of edges of~$\bbH$ with spin~$\rm$ on either side. 
	Let~$\dbm(\sigma_r)^*$ be the dual of ~$\dbm(\sigma_r)$; 
	it is a percolation configuration on~$\bbT$, with edges open if at least one of their endpoints is a face of spin~$\rp$.
	Then~$\calC^v_{\srm\srm}(\Rect_{n/16, n/2})$ fails if and only if~$\Rect_{n/16, n/2}$ is crossed horizontally by a path in~$\dbm(\sigma_r)^*$.
	The same holds with~$\Rect_{2n, n/2}$ instead of~$\Rect_{n/16, n/2}$.
	Moreover,~$\dbm(\sigma_r)^*$ is increasing in~$\sigma_r$, hence satisfies the FKG inequality under~$\mu_{\La_{(\kappa + 2) n}}^{\srm\srm}$. 

	The same strategy as in the proof of Lemma~\ref{lem:no_red_plus} applies here, 
	namely choosing the points on the left and right sides of~$\Rect_{n/16, n/2}$ that are most likely to be connected in~$\dbm(\sigma_r)^*$, 
	using horizontal reflection, the FKG inequality and monotonicity of boundary conditions, we find that
	\begin{align*}
		1 - \mu_{\La_{(\kappa + 4) n}}^{\srm\srm}[\calC^v_{\srm\srm}(\Rect_{2n, n/2})] 
		\geq \Big[\tfrac{4}{n^2}\big(1 - \mu_{\La_{(\kappa + 2)  n}}^{\srm\srm}[\calC^v_{\srm\srm}(\Rect_{n/16, n/2})] \big)\Big]^{32}. 
	\end{align*}
	Consider now the copies~$R_1,\dots, R_6$ of~$\Rect_{2n, n/2}$ placed around~$\La_n$ as in Figure~\ref{fig:Rect_to_ann}. 
	Then, due to the FKG inequality and the comparison between boundary conditions
	\begin{align*}
		\mu_{\La_{(\kappa + 6) n}}^{\srm\srm}\big[\La_n\nxlra{\srm\srm}\La_{2n}^c \big] 
		\geq \mu_{\La_{(\kappa + 6) n}}^{\srm\srm}\Big[\bigcap_{i=1}^6\calC_{\srm\srm}^v(R_i)^c \Big] 
		\geq \Big(1 - \mu_{\La_{(\kappa + 4) n}}^{\srm\srm}[\calC^v_{\srm\srm}(\Rect_{2n, n/2})] \Big)^6.
	\end{align*}
	Using the upper bound~\eqref{eq:yes_red_minus} for the LHS, we obtain~\eqref{eq:C4} with an adjusted value~$C_4 > 0$.%
	\smallskip 
	
	Let us now prove~\eqref{eq:C3}. 
	There exists some fixed constant~$K$ such that one may place~$K$ translations and rotations by~$2\pi/3$ and~$4\pi/3$ 
	of~$\Rect_{n/16, n/2}$ around~$0$ in such a way that, 
	if they are all crossed in the long direction by a double-$\rm$ path, then~$\Circ_{\srm\srm}(n,2n)$ occurs 
	(look at Figure~\ref{fig:Rect_to_ann2} for inspiration). 
	Using again the FKG inequality and the monotonicity of boundary conditions, we find
	\begin{align*}
		\mu_{\La_{\kappa n}}^{\srm\srm}[\Circ_{\srm\srm}(n,2n)] \geq
		\mu_{\La_{(\kappa +2)  n}}^{\srm\srm}[\calC^v_{\srm\srm}(\Rect_{n/16, n/2})]^K.
	\end{align*}
	Inserting ~\eqref{eq:C4} in the above proves~\eqref{eq:C3}.
	\medskip 
	
	We move on to proving~\eqref{eq:C31}.
	For~$n_1 \geq 1$, using the monotonicity of boundary conditions 
	and the fact that~$\mu_\bbH = \lim_{n\to\infty} \mu_{\La_n}^{\srm\srm}$, we find 
	\begin{align*}
		\mu_{\bbH}\big[ \bigcap_{j \geq 0}\Circ_{\srm\srm}(2^j n_1,2^{j+1}n_1)\big]
		&= \prod_{j\geq 0} \mu_{\bbH}\big[ \Circ_{\srm\srm}(2^jn_1,2^{j+1}n_1) \,\big|\,\bigcap_{\ell > j }\Circ_{\srm\srm}(2^\ell n_1,2^{\ell+1} n_1) \big]\\
		&\geq \prod_{j\geq 0} \mu_{\La_{2^{j+2}n_1}}^{\srm\srm}\big[ \Circ_{\srm\srm}(2^j n_1,2^{j+1} n_1) \big]\\
		&\geq \prod_{j\geq 0} (1 - e^{-C_3 (2^{j }n_1)^{c}}),
	\end{align*}
	where~$C_3$ is given by~\eqref{eq:C3} for~$\kappa = 4$. 
	The right-hand side may be rendered as close to one as desired by taking~$n_1$ sufficiently large. 
\end{proof}

We are ready now to prove Proposition~\ref{prop:dicho-resolve} and thus resolve the dichotomy stated in Corollary~\ref{cor:dicho}. 
Below we show that assumption~\eqref{assumption:exp_decay} contradicts~\eqref{eq:uniqueness_nu}, hence it is false. 


\begin{proof}[Proposition~\ref{prop:dicho-resolve}.]
	Suppose~\eqref{assumption:exp_decay} occurs. Fix~$n_1$ large enough so that 
	\begin{align*}
		\mu_{\bbH}\Big[\bigcap_{j\geq0}\Circ_{\srm\srm}(2^{j}n_1,2^{j+1}n_1)\Big] \geq \tfrac34.
	\end{align*}
	Recall from Theorem~\ref{thm:limit_joint} that~$\mu_{\bbH}$ is invariant under red spin flip, whence
	\begin{align*}
		\mu_{\bbH}\Big[\bigcap_{j\geq0}\Circ_{\srp\srp}(2^{j}n_1,2^{j+1}n_1)\Big] \geq \tfrac34.
	\end{align*}
	Then, the intersection of the two events above occurs with probability at least~$1/2$.
	In particular, for any~$j \geq 0$, 
	\begin{align*}
		\mu_{\Lambda_{2^{j+2}n_1}}^{\srm\srm}\big[\Circ_{\srp\srp}(2^{j}n_1,2^{j+1}n_1) \big] 
		\geq \mu_{\bbH}\big[\Circ_{\srp\srp}(2^{j}n_1,2^{j+1}n_1) \cap \Circ_{\srm\srm}(2^{j+1}n_1,2^{j+2}n_1)\big] 
		\geq \tfrac12.
	\end{align*}
	Notice that $\Circ_{\srp\srp}(2^{j}n_1,2^{j+1}n_1)$ implies the occurrence of the the translate of 
	$\{\La_{2^{j}n_1}\xlra{\srp\srp}\La_{2^{j+1}n_1}^c\}$ by $(2^{j}n_1,0)$.
	By the monotonicity of boundary conditions we deduce that 
	\begin{align*}
		\mu_{\Lambda_{2^{j+3}n_1}}^{\srm\srm}\big[\La_{2^{j}n_1}\xlra{\srp\srp}\La_{2^{j+1}n_1}^c] \geq
		\mu_{\Lambda_{2^{j+2}n_1}}^{\srm\srm}\big[\Circ_{\srp\srp}(2^{j}n_1,2^{j+1}n_1) \big] 
		\geq \tfrac12.
	\end{align*}	
	This contradicts~\eqref{eq:no_red_plus} for~$j$ large enough, and~\eqref{assumption:exp_decay} fails.
	In other words, case~\textit{(i)} of Corollary~\ref{cor:dicho} holds. 
\end{proof}

\subsection{Proof of Theorem~\ref{thm:loops}}\label{sec:proof_polynomial_decay}

In this section we show how the statements about the spin representation proven in previous sections imply Theorem~\ref{thm:loops}.

\begin{proof}[Theorem~\ref{thm:loops}.]
\newcommand{\Loop}{\mathsf{Loop}}	
\noindent\textbf{\textit{(i)}}  
	Recall from Proposition~\ref{prop:loop-to-spin} that for any domain~$\calD$, 
	the measure~$\bbP_{\calD}$ is obtained from~$\mu_\calD^{\srp\srm}$ by considering the edges separating faces of different blue or red spin.
	
	If~$(\calD_n)_{n\geq 1}$  is a sequence of increasing domains with~$\bigcup_{n\geq 1} \calD_n = \bbH$, 
	Theorem~\ref{thm:uniqueness_nu} states that~$\mu_{\calD_n}^{\srp\srm}$ converges to~$\mu_\bbH$.
	As a consequence~$\bbP_{\calD_n}$ converges to the measure~$\bbP_\bbH$ obtained from~$\mu_\bbH$ 
	by the same procedure that produces~$\bbP_\calD$ from~$\mu_\calD^{\srp\srm}$.
	\medskip 

\noindent\textbf{\textit{(ii)}}   
	By Corollary~\ref{cor:circ_dp_dp}, there exists~$\bbP_\bbH$-a.s. no infinite path in~$\omega_r$. 
	Indeed, the existence of such a path is contradictory with the existence of infinitely many~$\rp$-circuits surrounding~$0$;
	the latter event was shown to occur~$\mu_\bbH$-a.s.
	The statement extends to blue paths by symmetry. 

	Alternatively, one may see the proof of Theorem~\ref{thm:uniqueness_nu}, where the absence of infinite paths was proved. 
	\medskip 
	
\noindent\textbf{\textit{(iii)}}  
	The ergodicity and rotation invariance of~$\bbP_\bbH$ follow from the corresponding properties of~$\mu_\bbH$, 
	which were obtained in Theorem~\ref{thm:limit_joint}.
	\medskip

\noindent\textbf{\textit{(iv)}} 
	We start with the lower bound. 
	Fix~$\calD$ a finite domain containing~$\La_{n}$ for some~$n$, or simply~$\calD = \bbH$. 
	The procedure that generates~$\bbP_{\calD}$ from~$\mu_\calD^{\srp\srm}$ is such that 
	\begin{align*}
		\bbP_{\calD}(\text{exists loop in~$\La_{n}$ surrounding~$\La_{n/2}$})
		&\geq \mu_\calD^{\srp\srm} \big[\Circ_{\srp}(n/2,n)\cap \Circ_{\srm}(n/2,n) \big] \\
		&\geq \mu_\calD^{\srp\srp} \big[\Circ_{\srp\srp}(\tfrac12n,\tfrac34n)\cap \Circ_{\srm\srm}(\tfrac34,n) \big]\\
		&\geq \mu_{\La_{3n/4}}^{\srm\srm} \big[\Circ_{\srp\srp}(\tfrac12n,\tfrac34n)\big] 
		 \cdot \mu_{\La_{n}}^{\srp\srp} \big[\Circ_{\srm\srm}(\tfrac34 n,n) \big].
	\end{align*}
	The second and third inequalities are due to the monotonicity of boundary conditions. 
	The last term is bounded uniformly away from~$0$ by Corollary~\ref{cor:RSW_strong}.
	\smallskip 
	
	Let us now prove the upper bound. Fix a finite domain~$\calD$ and set~$n = \textrm{dist}(0,\calD^c)$.  
	Fix some~$\rho <1$ close enough to~$1$; we will see below how~$\rho$ needs to be chosen and that it does not depend on~$n$ or~$\calD$.
	Let~$\Loop^2$ be the event that there exist at least two loops in~$\calD$ that surround~$\La_{\rho n}$.
	Let~$\Loop^{2}(\text{red,blue})$ be the event that~$\Loop^2$ occurs 
	and that the outermost loop surrounding~$\La_{\rho n}$ is blue, while the second outermost is red. 
	Then~$\bbP_{\calD}[\Loop^{2}(\text{red,blue})] =\frac14 \bbP_\calD(\Loop^{2})$. 
	
	Let us consider for a moment spin configurations~$(\sigma_r,\sigma_b)$ chosen according to~$\mu_\calD^{\srp\srm}$  
	that correspond to loop configurations in~$\Loop^{2}(\text{red,blue})$. 
	The outermost blue loop corresponds to a double-$\rp$ circuit in~$\calD \setminus \La_{\rho n}$. 
	Indeed, as any blue loop, it is either a double-$\rp$ circuit or a double-$\rm$ circuit. 
	Since there is no red loop separating it from~$\partial_E\calD$, its spin is the same as that of~$\partial_\int\calD$, namely~$\rp$. 
	The second outermost loop, the red one, induces a simple-$\rm$ circuit that surrounds~$\La_{\rho n}$ 
	and is contained inside the double-$\rp$ circuit above. 
	
	Coming back to general configurations~$(\sigma_r,\sigma_b)$ on~$\calD$, 
	let~$\Xi$ be the outermost double-$\rp$ circuit surrounding~$\La_{\rho n}$;
	if no such circuit exists, set~$\Xi =\emptyset$. 
	Let~$\Int(\Xi)$ be the domain delimited by~$\Xi$. 
	Due to the Spatial Markov property and a standard exploration argument, the measure inside~$\Xi$ is~$\mu_{\Int(\Xi)}^{\srp\srp}$. 
	
	By the discussion above, if~$\Loop^{2}(\text{red,blue})$ occurs, 
	then~$\Xi \neq \emptyset$ and there exists a simple-$\rm$ circuit contained in~$\Int(\Xi)$ and surrounding~$\La_{\rho n}$
	(this is not an equality of events; generally the latter event contains strictly the former). 
	Thus
	\begin{align}\label{eq:calC2}
		\tfrac14 \bbP_\calD(\Loop^{2})
		\leq 
		\sum_{\chi\neq \emptyset} \, \mu_{\calD}^{\srp\srm}\big[ \Circ_{\srm}(\rho n)\cap\{ \Xi = \chi\}\big]
		= \sum_{\chi\neq \emptyset} \, 
		\mu_{\Int(\chi)}^{\srp\srp}\big[\Circ_{\srm}(\rho n)\big]\cdot \mu_{\calD}^{\srp\srm}(\Xi = \chi),
	\end{align}
	where the sum is over all realisations~$\chi \neq \emptyset$ of~$\Xi$.
	The event~$\Circ_{\srm}(\rho n)$
	above refers to the existence of simple-$\rm$ circuit contained in~$\Int(\chi)$ and surrounding~$\La_{\rho n}$. 

	Due to the monotonicity of boundary conditions 
	$\mu_{\Int(\chi)}^{\srp\srp}[\Circ_{\srm}(\rho n)]\leq 	\mu_{\calD}^{\srp\srp}[\Circ_{\srm}(\rho n)]$
	for any~$\chi$ in the sum above. 
	In conclusion 
	\begin{align}\label{eq:v1}
		 \bbP_\calD(\Loop^{2})
		\leq 4 \,
		\mu_{\calD}^{\srp\srp}[\Circ_{\srm}(\rho n)]\cdot \sum_{\chi} \mu_{\calD}^{\srp\srm}(\Xi = \chi )
		\leq 4 \,
		\mu_{\calD}^{\srp\srp}[\Circ_{\srm}(\rho n)].
	\end{align}
	
	Let~$u$ be a point where~$\partial_E \La_{n}$ intersects~$\partial_E \calD$ (such a point exists due to the choice of~$n$). 
	By considering the intersection of the annulus~$u + \La_n \setminus \La_{(1-\rho)n}$ centred at~$u$ with~$\calD$, 
	and using the monotonicity of boundary conditions, we find that 
	\begin{align}\label{eq:v2}
		\mu_{\calD}^{\srp\srp}[\Circ_{\srm}(\rho n)]
		\leq 
		1 - \mu_{\bbH}^{\srp\srp}[\Circ_{\srp}((1-\rho) n, n)]. 
	\end{align}
	Indeed, the trace on~$\calD$ of any configuration~$\sigma_r$ on~$\bbH$ 
	that contains a simple-$\rp$ circuit in the annulus~$u + \La_n \setminus \La_{(1-\rho)n}$ 
	does not belong to~$\Circ_{\srm}(\rho n)$. 
	The above conclusion follows by the domination~$\mu_{\calD}^{\srp\srp}\geq_{\text{st}}\mu_{\bbH}^{\srp\srp}$. 
	
	It is a standard consequence of Corollary~\ref{cor:RSW_strong} that~$\rho <1$ may be chosen so that 
	$\mu_{\bbH}^{\srp\srp}[\Circ_{\srp}((1-\rho)n, n)]  \geq \frac78$ for all~$n$ larger than some fixed threshold. 
	Then, by~\eqref{eq:v1} and~\eqref{eq:v2},~$\bbP_\calD(\Loop^{2}) \leq \frac12$ for all~$n$ above this threshold, as required. 
	Smaller values of~$n$ may be incorporated by altering the constant~$c$. 
	\medskip
    
    \noindent\textbf{\textit{(v)}}
    Fix a domain~$\calD$; we will prove the results for~$\bbE_\calD$; the same proof applied for~$\bbE_\bbH$. 
    First we show the lower bound. 
    Set~$K = \lfloor \log_2 \mathrm{dist}(0,\calD^c)\rfloor$ 
    and for~$k = 1,\dots, K$, let~$\Circ^2(k)$ be the event that there exists a double-$\rp$ circuit 
    and a double-$\rm$ circuit in~$\La_{2^{k+1}}$ surrounding~$\La_{2^k}$.
    By the same reasoning as that used to prove \textit{(iv)} above, 
    \begin{align}\label{eq:circcirc}
    	\bbP_{\calD}(\Circ^2(k) \,|\, \omega_r \text{ on~$\La_{2^{k+1}}^c$})
    	&\geq \mu_{\La_{2^{k+1}}}^{\srp\srp} \big[\Circ_{\srm\srm}(\tfrac32 2^k,2^{k+1})\big] 
    		  \mu_{\La_{\frac32 2^{k}}}^{\srm\srm} \big[\Circ_{\srp\srp}(2^k,\tfrac32  2^{k})\big] 
    		 \geq c,
    \end{align}
    for some constant~$c>0$ independent of~$k$. 
    As a consequence 
    \begin{align*}
    	 \bbE_{\calD}(\#\{ 1\leq k < K-1\, :\, \Circ^2(k) \text{ occurs}\}) \geq c\, (K-1) \geq c \,\big(\log_2 \mathrm{dist}(0,\calD^c) -2\big).
    \end{align*}
    Now observe that each event~$\Circ^2(k)$ that occurs induces a loop in~$\calD$ that surrounds~$0$. 
    This provides the desired lower bound, after alteration of the constant~$c$. 
    
    We turn to the upper bound. 
    Let~$\Gamma_1,\dots, \Gamma_{N_\calD}$ be the loops of~$\omega$ surrounding~$0$, 
    ordered from outermost to innermost, when~$\omega$ is chosen according to~$\bbP_\calD$. 
    Set~$d_j = \textrm{dist}(0,\Gamma_j)$ for~$j =1,\dots, N_\calD$. 
    
    The loop measure inside~$\Gamma_j$, conditionally on~$\Gamma_1,\dots, \Gamma_j$, 
    and more generally on the whole configuration outside~$\Gamma_j$, is simply~$\bbP_{\Int(\Gamma_{j})}$.
    Applying~\eqref{eq:RSW_for_loops_ub} we find that 
    \begin{align}\label{eq:geometric_dom}
	    \bbP_\calD( d_{j+2} < \rho\, d_j \,|\, \Gamma_j \text{ and~$\omega$ outside~$\Gamma_j$}) > c.
    \end{align}
    From the above, it is standard to conclude that~$\bbE_\calD(N_\calD) \leq C \log n$ for some constant~$C$ depending on~$c$ and~$\rho$ only. 
    We sketch this below. 
    
    Let~$T_k = \min\{ j\leq N_\calD :\, d_j  < \rho^k\cdot  \mathrm{dist}(0,\calD^c)\}$ for~$k =1,\dots,K$ where~$K = \lceil \log_{1/\rho}\mathrm{dist}(0,\calD^c)\rceil$.
    Formally~$T_K = N_\calD$ and~$T_0 = 0$. 
    Then~\eqref{eq:geometric_dom} implies that each~$T_{j+1} - T_j$ may be bounded by a random variable~$2 G_j$, 
    where~$G_j$ has a geometric distribution of parameter~$c > 0$. 
    Then~$\bbE_\calD(N_\calD) \leq \sum_{j=1}^K 2\bbE (G_j) = \frac{2 K }c$, as required. 
    
    Finally, \eqref{eq:circcirc} also applies to~$\bbP_\bbH$ instead of~$\bbP_\calD$, and directly implies that 
    \begin{align*}
		\liminf_{K \to \infty} \frac{\#\{ 1\leq k < K \, :\, \Circ^2(k) \text{ occurs}\}}{\log K} > 0.
    \end{align*}
    Thus, there are indeed infinitely many loops surrounding the origin~$\bbP_\bbH$-a.s..
\end{proof}

\subsection{Proof of Theorem~\ref{thm:Gibbs}}\label{sec:Gibbs_proof}

Finally we prove Theorem~\ref{thm:Gibbs}. It may be worth mentioning that the proof below may be adapted to circumvent the use of the results of Section~\ref{sec:dichotomy}. 
Indeed, the non-quantitative delocalisation result of Theorem~\ref{thm:uniqueness_nu} suffices. 

\begin{proof}[Theorem~\ref{thm:Gibbs}]
	From the construction of~$\bbP_\bbH$ as limit of finite-volume measures, it is immediate that it is a Gibbs measure. 
	The rest of the proof is dedicated to showing it is the only one. 

	Let~$\eta$ be a Gibbs measure.
	For any configuration~$\omega$ chosen according to~$\eta$, 
	colour each loop of~$\omega$ independently in red or blue; colour each infinite path in red. 
	Write~$\omega_r$ and~$\omega_b$ for the obtained red and blue configurations, respectively.  
	Extend~$\eta$ to incorporate this additional randomness. 
	
	Additionally, associate to~$(\omega_r,\omega_b)$ a pair of spin configurations~$(\sigma_r,\sigma_b)$ obtained by choosing the spins at~$0$
	$(\sigma_r(0),\sigma_b(0)) \in \{\rm,\rp\} \times \{\bm,\bp\}$ uniformly, 
	then assign spins to all other faces with the constraint that two faces have distinct red spin (and blue spin, respectively) 
	if and only if they are separated by an edge of~$\omega_r$, and~$\omega_b$, respectively.
	Thus~$\eta$ is both a law on pairs of red and blue loop configurations, as well as a law on pairs of red and blue spin configurations. 
	We call the latter the double-spin representation of~$\eta$. 
	
	Let us show that the red-spin marginal of~$\eta$ is equal to~$\nu_\bbH$. 
	Notice that~\eqref{eq:DLR} implies that the double-spin representation of~$\eta$ has the spatial Markov property in that, 
	for any domain~$\calD$, 
	the restriction of~$\eta$ to~$\Int(\calD)$ conditionally on the double-spin configuration outside~$\Int(\calD)$ 
	is measurable in terms of the double-spin configuration on~$\partial_\int\calD$.
	
	Fix~$\eps > 0$ and~$n\geq 1$. 
	Let~$N > n$ be chosen so that 
	\begin{align}\label{eq:eps1}
		\nu^{\srp\srp}_{\La_N}(A) \leq \nu_{\bbH}(A) + \eps,
	\end{align}
	for any event~$A$ that depends only on the spins in~$\La_n$.

	Write $\omega\setminus \Lambda_N$ for the subgraph containing only the edges of $\omega$ that lie outside of $\Lambda_N$. Then, the connected components of $\omega\setminus \Lambda_N$ may be loops, bi-infinite paths, as well as semi-infinite or finite paths with endpoints on $\partial \Lambda_N$. 
	For~$M > N$, let $B(N,M)$ be the event that there exists a finite connected component of $\omega\setminus \Lambda_N$ 
	intersecting both~$\partial \La_N$ and~$\La_M^c$. Choose then $M$ so that 
	\begin{align}\label{eq:eps2}
		\eta[B(N,M)] \leq \eps. 
	\end{align}
	Since $B(N,M)$ is defined in terms of finite connected components, it is always possible to find such a value of $M$. 

	Write $\tilde \omega$ for the union of all connected components of $\omega \setminus \Lambda_N$ that intersect $\Lambda_M^c$. 
	When $B(N,M)$ fails, write $\calD(\omega)$ for the connected component of $\Lambda_N$ in $\bbH \setminus \tilde \omega$
	and $X_1,\dots, X_{2k}$ for the points of degree one in $\tilde \omega$ 
	-- they all lie on~$\partial \Lambda_N$ and are endpoints of infinite paths of $\tilde \omega$. 
	When $B(N,M)$ occurs, set $\calD(\omega) = \emptyset$. 

	Note that both $B(N,M)$ and $(\calD(\omega),X_1,\dots, X_{2k})$ are measurable in terms of $\tilde\omega$, 
	which lies entirely outside of $\calD$. 
	Thus, even tough $\calD(\omega)$ is not formally a domain\footnote{The boundary of $\calD$ is not necessarily formed of a simple loop, due to the potential infinite paths with endpoints on $\partial \Lambda_N$. 
	To deduce \eqref{eq:DLR_non-domain}, first apply \eqref{eq:DLR} in $\Lambda_M$ to identify the restriction of $\eta$ to $\Lambda_M$ conditionally on the configuration outside, then apply the Spatial Markov property of this restriction to deduce its further restriction to $\calD$.}, the DLR property applies in a straightforward way.
	Indeed, for~$D,x_1,\dots, x_{2k}$ a possible realisation of~$\calD(\omega),X_1,\dots, X_{2k}$ with $D \neq \emptyset$, 
	the restriction	of~$\eta[.\,|\, \calD(\omega) = D \text{ and } X_1,\dots,X_{2k} = x_1,\dots, x_{2k}]$ to~$D$ is given by 
	\begin{align}\label{eq:DLR_non-domain}
		\bbP_{D}^{x_1,\dots,x_{2k}}(\xi) 
		=\tfrac{1}{Z} \,2^{\#\ell(\xi)}\ind_{\{\text{the odd vertices of~$\xi$ are~$x_1,\dots,x_{2k}$}\}}, 
		\qquad \text{for all~$\xi \in \{0,1\}^{E(\calD)}$},
	\end{align}
	where~$\ell(\xi)$ is the number of loops of~$\xi$ entirely contained in~$D$.

	Let us now describe the conditional measure above for spin configurations. 
	Since no finite loop intersects~$\partial_E D$, the blue spins on~$\partial_\int D$ are all identical, either~$\bp$ or~$\bm$. 
	The red spins along~$\partial_\int D$ switch from~$\rp$ to~$\rm$ and vice-versa at every point~$x_i$ due to the infinite (red) paths. 
	Write~$\mu_D^{x_1,\dots,x_{2k}, \srp\sbp}$ for the spin measure on~$D$ with boundary conditions~$\bp$ on~$\partial_\int D$, 
	$\rp$ on the segments of~$\partial_\int D$ between~$x_i$ and~$x_{i+1}$ with~$i$ odd, and~$\rm$ on all other parts of~$\partial_\int D$. 
	To be precise, this is the uniform measure on coherent configurations 
	$(\sigma_r,\sigma_b) \in \{\rm,\rp\}^{F(\calD)}\times \{\bm,\bp\}^{F(\calD)}$ that have the values above on~$\partial_\int D$. 
	Define~$\mu_D^{x_1,\dots,x_{2k}, \srm\sbp}$,~$\mu_D^{x_1,\dots,x_{2k}, \srp\sbm}$ and~$\mu_D^{x_1,\dots,x_{2k}, \srm\sbm}$, similarly;
	these are the push-forward of~$\mu_D^{x_1,\dots,x_{2k}, \srp\sbp}$ via 
	$(\sigma_r,\sigma_b) \mapsto (-\sigma_r,\sigma_b)$, 	$(\sigma_r,\sigma_b) \mapsto (\sigma_r,-\sigma_b)$ and 	$(\sigma_r,\sigma_b) \mapsto (-\sigma_r,-\sigma_b)$, respectively.
	Then, due to the uniform choice of the spins at~$0$, 
	\begin{align*}
		\bbP_{D}^{x_1,\dots,x_{2k}}
		=
		\tfrac14 \big[ 
		\mu_D^{x_1,\dots,x_{2k}, \srp\sbp}
		+ \mu_D^{x_1,\dots,x_{2k}, \srm\sbp}
		+ \mu_D^{x_1,\dots,x_{2k}, \srp\sbm}
		+ \mu_D^{x_1,\dots,x_{2k}, \srm\sbm}\big].
	\end{align*}
	
	Due to Corollary~\ref{cor:fkg}~\textit{(iii)}, 
	the red spin marginals of the mesures 
	$\mu_D^{x_1,\dots,x_{2k}, \srp\sbp}$,~$\mu_D^{x_1,\dots,x_{2k}, \srm\sbp}$, 
	$\mu_D^{x_1,\dots,x_{2k}, \srp\sbm}$ and~$\mu_D^{x_1,\dots,x_{2k}, \srm\sbm}$ 
	all satisfy the FKG inequality. 
	In particular, since~$\La_N \subset D$, their restrictions to~$\La_N$ are all dominated by~$\nu_{\La_N}^{\srp\srp}$. 
	In conclusion we find that, for any increasing event~$A$ depending only on the red spins inside~$\La_n$, 
	\begin{align*}
    	\eta(A) 
    	& \leq \eta(\calD(\omega) = \emptyset) + \sum_{\substack{D;x_1,\dots, x_{2k}\\ D\neq\emptyset}} 
				\bbP_{D}^{x_1,\dots,x_{2k}}(A)\cdot\eta(\calD(\omega) = D;\, X_1 = x_1,\dots, X_{2k} = x_{2k})\\
		& \leq \eta(B(N,M)) + \sum_{D \neq \emptyset} \nu_{\La_N}^{\srp\srp}(A) \cdot\eta(\calD(\omega) = D) \\
		& \leq \nu_{\bbH}(A) + 2\eps.
	\end{align*}
	The last inequality is due to \eqref{eq:eps1} and \eqref{eq:eps2}.
	Recall that the choice of~$\eps$ is arbitrary, hence~$\eta(A) \leq \mu_{\bbH}(A)$ for all events~$A$ as above. 
	
	The same argument may be performed with $\rp$ replaced by $\rm$, and yields that for any decreasing event~$B$ depending only on the red spins inside~$\La_n$
	(that is the complement of an increasing event), $\eta(B) \leq \mu_{\bbH}(B)$.
	Thus,~$\eta(A) = \mu_{\bbH}(A)$ for all increasing (and decreasing) events that only depend on the red spins in a finite region. 
	The monotone class theorem allows to conclude that the red spin marginals of~$\eta$ and~$\mu_\bbH$ are equal. 

	
	Now, given the red spin marginal of~$\mu_\bbH$, the blue spins are obtained 
	by awarding uniform blue spins to the  clusters of~$\theta(\sigma_r)$. 
	The same holds for~$\eta$, since~$0$ is surrounded~$\eta$-a.s. by infinitely many disjoint clusters of~$\theta(\sigma_r)$.  
	As a consequence~$\eta = \mu_\bbH$. 
\end{proof}

\subsection{RSW theorem for height functions}\label{sec:RSW_heights}

We finish the paper with a RSW result for the uniform Lipschitz functions model. Recall the notation of Section~\ref{sec:height_functions}.
When considering Lipschitz functions on a domain containing $\La_{2n}$, write $\Circ_{\geq k}(n)$ for the event that 
there exists a closed face-path contained in $\La_{2n}$, surrounding $\La_n$ and formed entirely of faces for which the function is larger than $k$. 

\begin{thm}\label{thm:RSW_heights}
	For any $k \geq 1 $ there exists $c(k) > 0$ such that for all $n$ large enough and 
	any domain $\calD$ containing $\La_{2n}$. 
	\begin{align*}
	\pi_\calD(\Circ_{\geq k}(n)) \geq c(k).
	\end{align*}
\end{thm}

\begin{proof}
	Fix $k \geq 1$. Let $n \geq 1$ be a large integer and $\calD$ a domain  containing $\La_{2n}$. 
	Recall from Propositions~\ref{prop:lip-to-loop} and~\ref{prop:loop-to-spin} 
	that the loop representation of a height function chosen according to $\pi_\calD$
	has law $\bbP_\calD$, and its spin representation has law $\mu^{\srp\srm}_\calD$. 
	
	Write $H$ for the event that there exist $k+1$ closed edge-paths $\gamma_1,\dots, \gamma_{k+1}$ in $\La_{2n}$
	that surround $\La_n$, that are numbered from outer-most to inner-most, and such that $\gamma_j$ is a double-$\rp$ path if $j$ is odd
	and a double-$\rm$ path if $j$ is even. 
	By repeated applications of Corollary \ref{cor:RSW_strong}, 
	there exists a constant $c(k) > 0$ independent of $n$ or $\calD$ such that $\mu^{\srp\srm}_\calD(H) \geq c(k)$.

	When $H$ occurs, there exist at least $k$ loops in the loop representation that are contained in $\La_{2n} \setminus \La_n$ 
	and surround $\La_n$. 
	Write $\tilde H$ for set of loop configurations which contain at least $k$ such loops, and denote by $\Ga_1,\dots, \Ga_k$ the outermost $k$ loops as above. 
	Recall that, in order to obtain the height function $\Phi$ from the loop configuration $\omega$ chosen according to $\bbP_\calD$, 
	loops need to be oriented uniformly, and that the orientation of each loop dictates whether the height inside the loop is larger or smaller than the one outside. 
	By symmetry, conditionally on any loop configuration $\omega \in \tilde H$, 
	with probability at least $1/2$ the height of the faces outside and adjacent to $\Gamma_1$ is at least $0$. 
	Moreover, independently of the above, all paths $\Gamma_1,\dots, \Gamma_k$ are oriented clockwise with probability $2^{-k}$.  
	When both of the above occur, the height of the faces inside and adjacent to $\Gamma_k$ is at least $k$. Thus
	\begin{align*}
		\pi_\calD(\Circ_{\geq k}(n)) \geq
		2^{-(k+1)} \bbP_\calD(\tilde H)
		\geq 	2^{-(k+1)} \mu^{\srp\srm}_\calD(H) 
		\geq 	2^{-(k+1)} c(k),
	\end{align*}
	which is the desired conclusion.
\end{proof}

\bibliographystyle{abbrv}
\bibliography{biblicomplete}

\end{document}